\pdfoutput=1 

\RequirePackage{fix-cm}
\documentclass[smallextended]{svjour3}
\usepackage[utf8]{inputenc}
\usepackage{booktabs}
\usepackage{empheq}
\usepackage{amssymb,amsmath,graphics,graphicx,subfig}
\usepackage[title]{appendix}
\usepackage{mathrsfs}
\usepackage{tikz}
\usepackage{comment}
\usepackage{eurosym,bm,amsmath}     
\usepackage{scalerel}
\usepackage{enumitem}
\usepackage{multirow}
\definecolor{lightblue}{rgb}{0.22,0.45,0.70}
\definecolor{lightgreen}{rgb}{0.22,0.50,0.25}
\usepackage[right = 2cm, left=2cm, top = 2.5cm, bottom =2.5cm]{geometry}
\usepackage[colorlinks=true,breaklinks=true,linkcolor=lightblue,citecolor=lightblue]{hyperref}

\definecolor{darkred}{rgb}{0.82,0.15,0.20}
\definecolor{darkblue}{rgb}{0.82,0.15,0.12}
\newcommand*{\Scale}[2][4]{\scalebox{#1}{$#2$}} 

\renewenvironment{proof}{\noindent{\it Proof.}}{\hfill$\square$}

\numberwithin{equation}{section}
\numberwithin{figure}{section}
\numberwithin{table}{section}
\numberwithin{lemma}{section}
\numberwithin{corollary}{section}
\numberwithin{theorem}{section}
\numberwithin{remark}{section}
\numberwithin{proposition}{section}

\newcommand{\vertiii}[1]{{\left\vert\kern-0.25ex\left\vert\kern-0.25ex\left\vert #1 
		\right\vert\kern-0.25ex\right\vert\kern-0.25ex\right\vert}}

\allowdisplaybreaks

\begin{document}
	\titlerunning{Stabilized VEM for generalized Oseen equation}   
	\title{A unified stabilized virtual element method for the generalized Oseen equation: stability and robustness}    
	\authorrunning{Mishra, Natarajan}
	\author{Sudheer Mishra \and Natarajan E }
	\institute{
		Natarajan E \at
		Department of Mathematics, Indian Institute of Space Science and Technology, Thiruvananthapuram, 695547, India \\
		\email{thanndavam@iist.ac.in}.\\
		Sudheer Mishra \at Department of Mathematics, Indian Institute of Space Science and Technology, Thiruvananthapuram, 695547, India\\
		\email{sudheermishra.20@res.iist.ac.in}.
		\and
	}
	\date{}
	\maketitle
	
\begin{abstract}
		In this thesis, we investigate a novel local projection based stabilized conforming virtual element method for the generalized Oseen problem using equal-order element pairs on general polygonal meshes. To ensure the stability, particularly in the presence of convection-dominated regimes and the utilization of equal-order element pairs, we introduce local projections based stabilization techniques. We demonstrate the discrete inf-sup condition in the energy norm. Moreover, the stability of the proposed method also guarantees the stability properties for the Brinkman equation and the Stokes equation without introducing any additional conditions. Furthermore, we derive an optimal error estimates in the energy norm that underline the uniform convergence in the energy norm for the generalized Oseen problem with small diffusion. In addition, the error estimates remain valid and uniform for the Brinkman equation and the Stokes equation. Additionally, the convergence study shows that the proposed method is quasi-robust with respect to parameters. The proposed method offers several advantages, including simplicity in construction, easier implementation compared to residual-based stabilization techniques, and avoiding coupling between element pairs. We validate our theoretical findings through a series of numerical experiments, including diffusion-dominated and convection-dominated regimes.
	
\end{abstract}


	\section{Introduction}
 \label{sec-1}
The iterative solutions of the Navier-Stokes equation give rise to the subproblem of solving the Oseen equation. The standard Galerkin finite element method provides two challenges: one is convection-dominated, and the other is to satisfy the discrete inf-sup (LBB) condition for the
velocity and pressure approximation spaces.

The streamline upwind Petrov-Galerkin method \cite{mfem29} and the pressure-stabilized Petrov-Galerkin (PSPG) method \cite{mfem28} treat both problems in one framework. Along with this, the element-wise stabilization of the divergence constraint \cite{vemnew1}. To guarantee the consistency of these stabilization methods, the weak formulation needs to be modified significantly by adding several terms to it. The stabilization can be reduced by avoiding the PSPG term \cite{vemnew2}; however, this introduces a strong coupling between velocity and pressure, showing the importance of the PSPG term. To reduce this strong coupling, several approaches are discussed in the literature, like introducing symmetric stabilization terms; look into \cite{vemnew3,vemnew4,vemnew5}.

In recent years, there have been several approaches to extend the classical finite element methods whose computational grids are of triangular/tetrahedral elements to polytopal finite elements involving polygonal/polyhedral elements \cite{vemnew6,vemnew7,vemnew8,vemnew9,vemnew10}. Virtual element method (VEM) \cite{vem1} is the most promising polytopal approach among these. Studies on VEM have been broadly discussed for Stokes equations, such as \cite{vemnew11,vemnew12,vemnew14}. Furthermore, VEM has also been developed in various forms to address several mixed problems, including the conforming and non-conforming VEM for {Navier-Stokes/Oseen \cite{mvem9,mvem10,adak2023nonconforming}}; the divergence-free VEM for Brinkman \cite{mvem12} and coupled Navier-Stokes-Heat \cite{mvem3}; the weak divergence-free VEM for coupled Stokes-Darcy \cite{mvem2}. The construction of conforming/non-conforming/divergence-free methods primarily relies on the construction of finite spaces that approximate the velocity and pressure fields. Conventionally, the velocity vector field is approximated by a virtual element space, while the pressure field is approximated using classical FEM to ensure the discrete inf-sup condition. This raises the question of whether the pressure field can be effectively approximated through virtual element space. If so, would such an approach ensure stability? When we approximate the pressure field using the virtual element space, instability arises due to the violation of the discrete inf-sup condition. Thus, developing stable methods using the fact that velocity and pressure fields are approximated by $H^1$-conforming virtual spaces becomes an active and challenging issue.

In this work, we propose and analyze a novel stabilized VEM for the Oseen problem on general polygonal meshes, showing the validity of the proposed method for different cases, including convection-dominated regimes and diffusion-dominated regimes. The proposed stabilized VEM formulation is derived from \cite{mfem31,vem28m,vem08}. In \cite{mvem14}, the authors proposed a VEM least-squares stabilization for the Navier-Stokes equation, where the stability of the method is shown by performing numerical experiments without establishing the stability or error estimates. In this method, the stabilization term consists of several terms, including second-order derivatives to ensure consistency, which can lead to enormous coupling. In contrast, the proposed VEM does not involve any higher-order derivative terms and provides separate stabilization terms to capture the small scales of velocity and pressure on the element level. Moreover, these terms disregard the cross products between pressure and velocity derivatives to relax the coupling. In \cite{mvem16}, the authors propose a div-free stabilized VEM for the Oseen equation using SUPG-like terms of the vorticity equation, internal jumps, and local VEM stabilization terms, where the computation of second-order derivative terms becomes necessary. In \cite{vem028}, the authors have proposed a local projection-based stabilized virtual element method for the Stokes problem, where the error estimates become quasi-uniform as the diffusion parameter tends to zero. Further, a stabilized VEM for the Oseen problem using the local projection stabilization method has been developed in \cite{mvem15}, where the stability analysis and the error estimates are restrictive. {In contrast, we developed a new variant of the local projection stabilization technique where the stability and error analysis are not restrictive, which guarantees the stability properties and uniform convergence estimates for Stokes/Brinkman problems without introducing any additional assumptions. We further highlight that the authors discussed a local projection based stabilized VEM for unsteady Navier-Stokes problem in \cite{li2024stabilized}, where the velocity stabilizing term (namely $S_1(\cdot,\cdot)$ on page 1787) becomes zero on triangular meshes for VEM order $k=1$, whereas the proposed method would be effective for this case which can be seen mathematically by using the basic of VEM \cite{vem1}.} More recently, in \cite{mvem21}, a stabilized virtual element framework has been investigated for the coupled Stokes-Darcy equation, where the error estimates depend on the diffusion parameter. More precisely, the error estimates in \cite{vem028,mvem21} are not valid for the small diffusion.

In this thesis, we introduce a new discrete bilinear form corresponding to the convective term to establish a {quasi-robust} method following \cite{vem04,mvem16}. The main contributions of the current paper are the following:
\begin{itemize}[label=\textbullet] 
	\item The well-posedness of the stabilized virtual element problem is shown over a mesh-dependent energy norm.
	\item The well-posedness theorem ensures the well-posedness of Brinkman and Stokes equations without imposing additional
	conditions. Hence, the well-posedness is robust with respect to the convective flow field and reaction parameter.
	\item The error estimates are uniform in the energy norm with respect to the diffusion parameter for the considered problem. In addition, the convergence estimates remain valid and uniform for the Brinkman and Stokes equations. Furthermore, the convergence study is  {quasi-robust} with respect to parameters.
	\item The proposed method offers numerous advantages, including simplicity in construction, {easier implementation compared to residual-based stabilization methods}, reduced computational time, and suitability for parallel computing.
	\item Numerical results validate the convergence estimates.
\end{itemize}

The remaining structure of the paper is given as follows: In Section \ref{sec-2}, we present the Oseen equation, illustrating that its weak formulation yields a unique solution. Section \ref{sec-3} focuses on the basics of VEM and outlines the construction of the stabilized virtual element problem for the Oseen equation. In Section \ref{sec-4}, we establish the stability result and the convergence estimates in the energy norm.  Section \ref{sec-5} presents our numerical findings, validating the theoretical aspects. Finally, Section \ref{sec-6} comes with concluding remarks.

\subsection{Notations}
Throughout the paper, {we will follow the standard notations from \cite{boffi_mixed}.} In brief, let $\Omega\subset\mathbb{R}^2$ be a bounded open domain with Lipschitz boundary $\partial \Omega$. We denote the usual $L^2$-inner product and norm by $(\cdot,\cdot)_\Omega$ and $\|\cdot\|_{0,\Omega}$, respectively. Additionally, let $E $ be any measure subset of $\Omega$, we employ $H^s(E)$ as the standard Sobolev space equipped with its usual $H^s(E)$-norm and semi-norms denoted by $\|\cdot\|_{s,E}$ and $|\cdot|_{s,E}$,  respectively. The constant $C$ is positive, independent of the decomposition of $\Omega$, and relies on the stability constants of the VEM. It may depend on the mesh regularity constant and $|\Omega|$. Bold font is frequently used to represent vectors or tensors. The Poincar$\acute{\text{e}}$ constant is denoted by $C_P$.
\section{The generalized Oseen equation }
\label{sec-2}
In this section, we introduce the generalized Oseen equation along with the well-posedness of its primal formulation.

The stationary generalized Oseen equation on a polygonal simply connected domain $\Omega \subset \mathbb{R}^2$ is given as follows:
\begin{align}
	\begin{cases}
		\text{Find $ (\mathbf{u},p)$  such that} \\
		- \mu \Delta \mathbf{u} + (\nabla \mathbf{u})  \boldsymbol{\mathcal{B}} + \gamma \mathbf{u} + \nabla p  = \mathbf{f}  \quad  &\text{in} \quad \Omega, \\
		\quad \nabla \cdot \mathbf{u} = 0  \quad  &\text{in} \quad \Omega, \\
		\quad  \mathbf{u} = \mathbf{g}  \quad  &\text{on} \quad \partial\Omega, 
	\end{cases}
	\label{modeleq}
\end{align}
{where the constants $0<\mu \leq 1 $ and $\gamma>0$} are the viscosity/diffusive coefficient and the reaction coefficient, respectively. The element pair $(\mathbf{u}, p)$ represents the velocity vector field and the scalar pressure field. The convective velocity field is denoted by $ \boldsymbol{\mathcal{B}}\in [W^{1,\infty}(\Omega)]^2$ with $\nabla\cdot  \boldsymbol{\mathcal{B}}=0$, a.e. on $\Omega$. The load $\mathbf{f} \in [L^2(E)]^2$ is the source/sink term. For the sake of simplicity, we assume a homogeneous Dirichlet boundary condition, i.e., $\mathbf{g}=\mathbf{0}$. {Additionally, the non-homogeneous Dirichlet boundary conditions can be employed following \cite[Chapter 2]{gatica2014simple}, and the Neumann boundary can be treated by using the trace inequality. }

\subsection{Weak formulation}
We now introduce the spaces for the velocity vector field and the pressure field as follows:
\begin{center}
	$\mathbf{V} := [H^1_0(\Omega)]^2, \qquad  Q:= L^2_0(\Omega)= \big\{
	z \in L^2(\Omega): \small{\int}_{\Omega} z \, d \Omega=0 \big \},  \qquad\qquad \qquad\qquad \qquad \qquad \qquad$
\end{center}
endowed with their natural norms. Furthermore, we define the following bilinear forms for all $\mathbf{u}, \mathbf{w}\in \mathbf{V}$ and $ q \in Q$,
\begin{itemize}[label=\textbullet] 
	\item $a(\cdot,\cdot):\mathbf{V} \times \mathbf{V}  \rightarrow \mathbb R$ \qquad $a(\mathbf{u},\mathbf{w}):= \mu \int_{\Omega} \nabla \mathbf{u}: \nabla \mathbf{w}\, d\Omega$, {where $A:B$ denotes the contraction of two tensors $A$ and $B$, see \cite{boffi_mixed}};
	\item $b(\cdot,\cdot):\mathbf{V} \times Q  \rightarrow \mathbb R$ \qquad $b(\mathbf{w},q):= \int_{\Omega} (\nabla \cdot \mathbf{w}) q\, d\Omega$;
	\item $c( \cdot,\cdot): \mathbf{V} \times \mathbf{V} \rightarrow \mathbb R$ \qquad $c(\mathbf{u},\mathbf{w}):=  \int_\Omega (\nabla \mathbf{u})  \boldsymbol{\mathcal{B}} \cdot \mathbf{w} \, d\Omega$;   
	\item $d(\cdot,\cdot): \mathbf{V} \times \mathbf{V}  \rightarrow \mathbb R$ \qquad $d(\mathbf{u},\mathbf{w}):= \gamma \int_{\Omega} \mathbf{u} \cdot  \mathbf{w}\, d\Omega$.
\end{itemize} 

As we assume that $\nabla\cdot\boldsymbol{\mathcal{B}}=0$, employing the integration by parts, we can easily conclude that the continuous form $c(\cdot, \cdot)$ satisfies the skew-symmetric property, i.e.
\begin{align}
	c(\mathbf{u},\mathbf{w})= -c(\mathbf{w},\mathbf{u}) \qquad \text{\, for all \,\,} \mathbf{u},\mathbf{w} \in \mathbf{V}.
\end{align}
Following the skew-symmetric property of $c(\cdot,\cdot)$, we define a bilinear form corresponding to the convective term as follows
\begin{align}
	c^{skew}(\mathbf{u},\mathbf{w}) = \dfrac{1}{2} \Big[c(\mathbf{u},\mathbf{w}) -c(\mathbf{w},\mathbf{u}) \Big] \qquad \text{\, for all \,\,} \mathbf{u},\mathbf{w} \in \mathbf{V}.
\end{align}
Hereafter, we introduce the bilinear form $ A : (\mathbf{V} \times Q) \times (\mathbf{V} \times Q) \rightarrow \mathbb{R}$ defined by
\begin{equation}
	A[(\mathbf{u}, p), (\mathbf{w},q) ]:= a(\mathbf{u}, \mathbf{w})- b(\mathbf{w}, p) + b(\mathbf{u}, q) + c^{skew}(\mathbf{u}, \mathbf{w}) + d(\mathbf{u}, \mathbf{w}) \qquad \text{for all} \,\, (\mathbf{u}, p), \, (\mathbf{w}, q) \in \mathbf{V} \times Q.
\end{equation}

The primal formulation for the problem \eqref{modeleq} reads as follows:
\begin{align}
	\begin{cases}
		\text{Find}\,\, (\mathbf{u}, p) \in \mathbf{V} \times Q,\,\, \text{such that} \\
		A[(\mathbf{u}, p), (\mathbf{w}, q)]= (\mathbf{f}, \mathbf{w}) \quad \, \text{for all \,\,} (\mathbf{w}, q) \in \mathbf{V} \times Q.
	\end{cases}
	\label{discf}
\end{align}

We emphasize that the form $A[\cdot,\cdot]$ is continuous and preserves the coercivity property, as the inf-sup condition for $b(\cdot,\cdot)$ is discussed by Ladyzhenskaya-Babu$\check{\text{s}}$ka-Brezzi (LBB) in \cite{mfem22}. Consequently, problem \eqref{modeleq} possesses a unique solution $(\mathbf{u},p) \in \mathbf{V} \times Q$.

\section{ The virtual element framework}
\label{sec-3}
This section focuses on developing a VEM computable stabilized discrete scheme for the Oseen problem \eqref{modeleq}. To achieve this, we begin by introducing the standard mesh assumptions and polynomial projectors, followed by defining the global virtual element space.\newline
\textbf{(A1) Mesh assumptions:}
\label{vemspace}
Let E be a general polygon with an area $|E|$ and diameter $h_E$; $\mathbf{n}^{E}$ denotes the unit outward normal vector to the boundary $\partial E$. Further, we represent a sequence of the decompositions of $\Omega$ into simple polygonal elements $E$ by $\{\Omega_h\}_{h>0}$; the maximum diameter is defined by $h:= \sup_{E\in \Omega_h} h_E$. We adopt the following standard mesh regularity assumptions for any $E \in \{\Omega_h\}_h$  
\begin{itemize}[label=\textbullet] 
	\item $E$ is considered as star-shaped with respect to a ball $D_E$ of radius $ \geq \theta h_E$,
	\item  each  edge $e$ of E has finite length $ \geq \theta h_E$,
\end{itemize}
where the positive constant $\theta$ is known as the mesh regularity constant.

\subsection{The projectors}
Let $\mathcal{O}$ be any geometric objects in $\mathbb{R}^d$ {(for $d=2 \,\,\text{or}\,\, 3$)} with diameter $h_{\mathcal{O}}$ and barycenter $\mathbf{x}_\mathcal{O}$. For $k \in \mathbb N \cup \{0\}$, we define the following set of monomials $\mathcal{M}_{k}(\mathcal{O})$  of degree $\leq k$, given as follows
\begin{equation}
	\mathcal{M}_{k}(\mathcal{O}):= \Big \{ m_\beta \, |\, m_\beta = \Big(\dfrac{\mathbf{x}- \mathbf{x}_\mathcal{O}}{h_\mathcal{O}}\Big)^{\boldsymbol{\beta}} \,\, \text{with} \,\, \boldsymbol{\beta} \in \mathbb{N}^d \,\, \text{such that} \,\, |\boldsymbol{\beta}| \leq k \Big \},
\end{equation}
where $\boldsymbol{\beta}=(\beta_1, \beta_2,..., \beta_d)$ is a multi-index set such that $\mathbf{x}^{\boldsymbol{\beta}}:= x_1^{\beta_1}x_2^{\beta_2}...x_d^{\beta_d}$ with $|\boldsymbol{\beta}|:= \beta_1 + \beta_2 + ...+\beta_d$. Furthermore, we also introduce a restriction of the monomials $\mathcal{M}_{k}(\mathcal{O})$ by  $\mathcal{M}^\ast_{k}(\mathcal{O})$ such that
\begin{equation}
	\mathcal{M}^\ast_{k}(\mathcal{O}):= \Big \{ m_\beta\, |\, m_\beta = \Big(\dfrac{\mathbf{x}- \mathbf{x}_\mathcal{O}}{h_\mathcal{O}}\Big)^{\boldsymbol{\beta}}, \,\, \boldsymbol{\beta} \in \mathbb{N}^d\,\, \text{such that} \,\, |\boldsymbol{\beta}| = k \Big \}.
\end{equation}
Let $\mathbb{P}_k(E)$ be the polynomial space of degree $\leq k$ on $E$, and  $\mathbb{P}_{-1}=\{ 0\}$. Notably, the monomial space $\mathcal{M}_{k}(E)$ represents a basis for $\mathbb{P}_k(E)$. 
{ Hereafter, we define the broken Sobolev and polynomial spaces as follows:
	\begin{itemize}[label=\textbullet] 
		\item  $W^s_k(\Omega_h) := \Big\{ v \in L^2(\Omega) \,\, \text{such that}\,\, v_{|E} \in W^s_k(E) \,\, \text{for all }\,\, E \in \Omega_h \Big\}$, where $s \in \mathbb R^{+}$;
		\item $\mathbb P_k(\Omega_h) := \Big\{ p \in L^2(\Omega) \,\, \text{such that}\,\, p_{|E} \in \mathbb P_k(E) \,\, \text{for all }\,\, E \in \Omega_h \Big\}$.
\end{itemize}}

For any polygonal element $E \in \Omega_h$, we introduce the following polynomial projection operators: 
\begin{itemize}[label=\textbullet] 
	\item The $\mathbf{H^1}$\textbf{-energy projection} operator $\Pi^{\nabla,E}_k: H^1(E) \rightarrow \mathbb{P}_k(E)$, defined by
	\begin{align}
		\begin{cases}
			\int_E \nabla (v - \Pi^{\nabla,E}_k v) \cdot \nabla m_k \, dE = 0 \qquad \, \text{for all}\, \, v \in H^1(E)\,\, \text{and} \,\, m_k \in \mathbb{P}_k(E),\\
			\int_{\partial E} (v - \Pi^{\nabla,E}_k v) \, ds=0.
		\end{cases}
		\label{H1proj}
	\end{align}
	For vector-valued functions, the extension is given by $\boldsymbol{\Pi}^{\nabla,E}_k: [H^1(E)]^2 \rightarrow [\mathbb{P}_k(E)]^2$ such that
	\begin{align}
		\begin{cases}
			\int_E \nabla (\mathbf{v} - \Pi^{\nabla,E}_k \mathbf{v}) : \nabla \mathbf{m}_k \, dE = 0 \qquad \, \text{for all}\, \, \mathbf{v} \in [H^1(E)]^2\,\, \text{and} \,\, \mathbf{m}_k \in [\mathbb{P}_k(E)]^2,\\
			\int_{\partial E} (\mathbf{v} - \Pi^{\nabla,E}_k \mathbf{v}) \, ds=\mathbf{0}.
		\end{cases}
		\label{H1proj2}
	\end{align}
	\item  The $\mathbf{L^2}$\textbf{-projection} operator $\Pi^{0,E}_k: L^2(E) \rightarrow \mathbb{P}_k(E)$, defined as follows
	\begin{align}
		\int_E (v - \Pi^{0,E}_k v) m_k \, dE = 0 \qquad \, \text{for all}\, \, v \in L^2(E)\,\, \text{and} \,\, m_k \in \mathbb{P}_k(E). \label{l2proj}
	\end{align}
	For vector-valued functions, the extension is given by $\boldsymbol{\Pi}^{0,E}_k: [L^2(E)]^2 \rightarrow [\mathbb{P}_k(E)]^2$ such that
	\begin{align}
		\int_E (\mathbf{v}- \Pi^{0,E}_k \mathbf{v}) \cdot \mathbf{m}_k \, dE = 0 \qquad \, \text{for all}\, \, \mathbf{v} \in [L^2(E)]^2\,\, \text{and} \,\, \mathbf{m}_k \in [\mathbb{P}_k(E)]^2. \label{l2proj2}
	\end{align}
\end{itemize} 
\textbf{(P0)} The $L^2$-projection operator is continuous with respect to $L^q$-norm for any $q\geq 2$ (see \cite{mishra2025equal}). In short, for any $v\in L^2(E)$, we have
\begin{align}
	\|\Pi^{0,E}_{k} v\|_{L^q(E)} \leq C \|v\|_{L^q(E)} \qquad \,\, \text{for all }\,\, q\geq 2.
\end{align}
\subsection{Virtual element spaces for the velocity and pressure fields } 
Hereafter, we introduce the virtual element approximations for the spaces $\mathbf{V}$ and $Q$ following \cite{vem22,vem25}. {We define the finite-dimensional local virtual element space $V^k_h(E)$ as follows
	\begin{align}
		V^k_h(E):= \Big\{ v \in H^1(E) &\cap C^0( \partial{E}) \,\, \text{such that}, \nonumber \\ &(i) \,\,\Delta v \in \mathbb P_k(E), \nonumber \\ &(ii)\, \,  v_{|e} \in \mathbb P_k(e) \,\, \forall \,\, e \in \partial E,  \nonumber \\  &(iii) \,\, \big(v-\Pi_k^{\nabla,E}v, m_\alpha \big)_{E}=0, \, \forall \, m_\alpha \in \mathcal{M}^\ast_{k-1}(E) \cup \mathcal{M}^\ast_{k}(E) \Big\}. \nonumber
\end{align}}

Let us introduce the sets of linear operators known as degrees of freedom, $\mathbf {DF}_1(v)$, $ \mathbf {DF}_2(v)$ and $ \mathbf {DF}_3(v)$ associated with any element $v \in V^k_h(E)$:
\begin{itemize}[label=\textbullet] 
	\item $\mathbf{DF}_1(v):$ the values of $v$ evaluated at each vertices of $E$;
	\item $\mathbf{DF}_2(v):$ the values of $v$ evaluated at $(k-1)$ internal Gauss-Lobatto quadrature nodes on each edge $e \in \partial E$;
	\item $\mathbf{DF}_3(v):$ the internal moments of $v$ up to order $k-2$,
	\begin{align*}
		\frac{1}{|E|} \, \int_E v m_\alpha\, dE \qquad \, \text{for all}\,\, m_\alpha \in \mathcal{M}_{k-2}(E).
	\end{align*}
	
\end{itemize}
Therefore, the global virtual element space $V^k_h$ can be defined by gluing the local virtual space $V^k_h(E)$ with the associated degrees of freedom:
\begin{align}
	V^k_h := \Big\{v \in H^1(\Omega) \, \, \text{such that}\, \, v_{|E} \in V^k_h(E) \,\, \text{for all}\,\, E \in \Omega_h \Big\}. \label{wj}
\end{align}
We now introduce the global discrete velocity space $\mathbf{V}^k_{h}$ and pressure space $Q^k_h$ as follows:
\begin{align}
	\mathbf{V}^k_{h} &:= \Big\{ \mathbf{v} \in \mathbf{V} \,\, \text{such that}\,\, \mathbf{v}_{|E} \in [V^k_h(E)]^2 \,\, \text{for all}\,\, E \in \Omega_h \Big\}, \label{dspace-1} \\
	Q^k_{h} &:= \Big\{ q \in Q \,\, \text{such that}\,\, q_{|E} \in V^k_h(E) \,\, \text{for all}\,\, E \in \Omega_h \Big\}. \label{dspace-3}
\end{align} 
{Furthermore, the local contribution of these spaces for any element $E\in \Omega_h$ can be defined by
	\begin{align}
		\mathbf{V}^k_{h}(E) &:= \Big\{ \mathbf{v} \in \mathbf{V} \,\, \text{such that}\,\, \mathbf{v}_{|E} \in [V^k_h(E)]^2 \Big\}, \label{ldspace-1} \\
		Q^k_{h}(E) &:= \Big\{ q \in Q \,\, \text{such that}\,\, q_{|E} \in V^k_h(E) \Big\}. \label{ldspace-3}
\end{align}}

\begin{remark}
	The local velocity space $\mathbf{V}^k_{h}(E)$ preserves the following property: \newline
	\textbf{(P1) } Polynomial inclusion: $[\mathbb{P}_k(E)]^2 \subseteq \mathbf{V}^k_{h}(E)$. \newline
	\textbf{(P2)} The choice of $\mathbf{DF}$ allows us to compute the following operators:
	\begin{align*}
		\boldsymbol{\Pi}^{\nabla,E}_k: \mathbf{V}^k_{h}(E) \rightarrow [\mathbb{P}_k(E)]^2, \qquad\boldsymbol{\Pi}^{0,E}_k: \mathbf{V}^k_{h}(E) \rightarrow [\mathbb{P}_k(E)]^2, \qquad  
		\boldsymbol{\Pi}^{0,E}_{k-1}: \nabla \mathbf{V}^k_{h}(E) \rightarrow [\mathbb{P}_{k-1}(E)]^{2 \times 2}.
	\end{align*}
\end{remark}

\begin{remark} \label{prop_qh}
	The local pressure space $Q^k_{h}(E)$ preserves the following property: 
	\begin{itemize}[label=\textbullet]
		\item Polynomial inclusion: $[\mathbb{P}_k(E)] \subseteq {Q}^k_{h}(E)$. 
		\item The choice of $\mathbf{DF}$ allows us to compute the following operators:
		\begin{align*}
			&{\Pi}^{\nabla,E}_k: {Q}^k_{h}(E) \rightarrow \mathbb{P}_k(E), &&{\Pi}^{0,E}_k: {Q}^k_{h}(E) \rightarrow \mathbb{P}_k(E), 
			&\boldsymbol{\Pi}^{0,E}_{k}: \nabla {Q}^k_{h}(E) \rightarrow [\mathbb{P}_{k}(E)]^{2 }.
		\end{align*}
	\end{itemize}
	We emphasize that we will use the same bold symbols/projection operators for both velocity and pressure spaces.
\end{remark}

\subsection{Virtual element forms}
We now discuss the next step of the proposed method, which involves the construction of a discrete version of the bilinear form $A[\cdot,\cdot]$ and the load term. For any  $(\mathbf{w}_h, q_h),(\mathbf{z}_h, s_h) \in \mathbf{V}^k_{h} \times Q^k_h$, the discrete bilinear form $A[(\mathbf{w}_h, q_h), (\mathbf{z}_h, s_h)]$ or the load $(\mathbf{f}, \mathbf{w}_h)$ can not be computable as they are not in closed form. Therefore, we will introduce the virtual element approximation of the continuous form and load term, which are computable in the sense of VEM employing \textbf{DF}'s.
\begin{itemize}[label=\textbullet]
	\item  Following VEM literature \cite{vem22,vem25,vem26}, the local discrete bilinear forms corresponding to continuous forms are given as follows:
	\begin{align}
		a^E_{h}(\mathbf{w}_h, \mathbf{z}_h) :&=  \mu \int_E \boldsymbol{\Pi}^{0,E}_{k-1} \nabla \mathbf{w}_h:	\boldsymbol{\Pi}^{0,E}_{k-1}\nabla \mathbf{z}_h\, dE + \mu S^E_ \nabla \Big((\mathbf{I}-\boldsymbol{\Pi}^{\nabla,E}_k) \mathbf{w}_h,(\mathbf{I}-\boldsymbol{\Pi}^{\nabla,E}_k) \mathbf{z}_h \Big)  \label{form-a}\\ 
		b^E_{h}(\mathbf{z}_h, q_h) :&= \int_E \Pi^{0,E}_{k-1} \nabla \cdot \mathbf{z}_h \Pi^{0,E}_{k} q_h\, dE \label{form-b} \\
		c^E_h(\mathbf{w}_h, \mathbf{z}_h):&= \int_{E}  (\nabla \boldsymbol{\Pi}^{0,E}_{k} \mathbf{w}_h) \boldsymbol{\mathcal{B}} \cdot \boldsymbol{\Pi}^{0,E}_{k} \mathbf{z}_h\, dE + \int_{\partial E}  \boldsymbol{\mathcal{B}} \cdot \mathbf{n}^E (\boldsymbol{I}-\boldsymbol{\Pi}^{0,E}_{k}) \mathbf{w}_h \cdot \boldsymbol{\Pi}^{0,E}_{k} \mathbf{z}_h\, ds \label{form-c1} \\
		c^{skew, E}_h(\mathbf{w}_h, \mathbf{z}_h):&= \dfrac{1}{2} \Big[c^E_h(\mathbf{w}_h, \mathbf{z}_h) - 	c^E_h(\mathbf{z}_h, \mathbf{w}_h) \Big] \label{form-c} \\
		d^E_{h}(\mathbf{w}_h, \mathbf{z}_h) :&= \gamma \int_E  \boldsymbol{\Pi}^{0,E}_{k} \mathbf{w}_h \cdot	\boldsymbol{\Pi}^{0,E}_{k} \mathbf{z}_h\, dE + \gamma S^E_0\Big((\mathbf{I}-\boldsymbol{\Pi}^{0,E}_k) \mathbf{w}_h,(\mathbf{I}-\boldsymbol{\Pi}^{0,E}_k) \mathbf{z}_h \Big)  \label{form-d}.
	\end{align}
	We remark that the discrete form $c^E_h(\cdot, \cdot)$ has an additional term on the boundary compared to the classical approximation for the convective term \cite{vem26,mvem9}. Recently, the discrete from $c^E_h(\cdot, \cdot)$ was first introduced by LB da {Veiga et al.} in \cite{vem04} in order to establish a fully robust method for the convection-diffusion problem. Later, this applies to the Oseen problem in \cite{mvem16}. For the completeness of the thesis, we will discuss the discrete convective form introduced in \cite{vem26,mvem9} in Section \ref{sec:5b}.   
	
	Moreover, the global bilinear forms are given by
	\begin{align*}
		a_{h}(\mathbf{w}_h,\mathbf{z}_h) &= \sum_{E \in \Omega_{h}} a^E_{h}(\mathbf{w}_h,\mathbf{z}_h), & & b_{h}(\mathbf{z}_h,q_h) = \sum_{E \in \Omega_{h}} b^{E}_{h}(\mathbf{z}_h,q_h), \nonumber \\
		c^{skew}_h(\mathbf{w}_h,\mathbf{z}_h) &= \sum_{E \in \Omega_{h}} c^{skew,E}_h(\mathbf{w}_h,\mathbf{z}_h), &&	d_{h}(\mathbf{w}_h,\mathbf{z}_h) = \sum_{E \in \Omega_{h}} d^E_{h}(\mathbf{w}_h,\mathbf{z}_h).
	\end{align*}
	Additionally, we will follow the same notation for the continuous case. In short, the forms $a^E(\cdot, \cdot)$, $b^E(\cdot,\cdot)$, $c^{skew,E}(\cdot, \cdot)$ and $d^E(\cdot,\cdot)$ represent the local contribution of the continuous forms $a(\cdot, \cdot)$, $b(\cdot,\cdot)$, $c^{skew}(\cdot, \cdot)$ and $d(\cdot,\cdot)$ on an element $E \in \Omega_h$, respectively.
	\item For each $E \in \Omega_h$, the local VEM stabilization terms $S^E_\nabla(\cdot,\cdot), S^E_0(\cdot, \cdot) : \mathbf{V}^k_{h}(E) \times \mathbf{V}^k_{h}(E) \rightarrow \mathbb{R}$ are symmetric positive definite bilinear forms such that they satisfy
	\begin{align}
		\lambda_{1\ast} \ \|\nabla \mathbf{z}_h\|_{0,E}^2 \leq S^E_\nabla(\mathbf{z}_h, \mathbf{z}_h)	 \leq
		\lambda_1^\ast  \|\nabla \mathbf{z}_h\|_{0,E}^2 \qquad   \,\, \text{for all} \, \, \mathbf{z}_h \in \mathbf{V}^k_{h}(E) \cap ker(\boldsymbol{\Pi}^{\nabla,E}_{k}), \label{vem-a} \\
		\lambda_{2\ast} \ \|\mathbf{z}_h\|_{0,E}^2 \leq S^E_0(\mathbf{z}_h, \mathbf{z}_h)	 \leq
		\lambda_2^\ast  \| \mathbf{z}_h\|_{0,E}^2 \qquad  \,\, \text{for all} \, \, \mathbf{z}_h \in \mathbf{V}^k_{h}(E) \cap ker(\boldsymbol{\Pi}^{0,E}_{k}), \label{vem-b}
	\end{align}
	where $\lambda_{1\ast} \leq\lambda_1^\ast$ and $\lambda_{2\ast} \leq\lambda_2^\ast$ are positive constants independent of $h$ {\cite{vem22}}.
	
	\item The discrete form $a^E_{h}(\cdot, \cdot)$ has the following properties:\newline
	\textbf{(P3)} Polynomial consistency: for any $\mathbf{z}_h \in \mathbf{V}^k_{h}(E)$, it holds that
	\begin{align}
		a^E_{h}(\mathbf{z}_h, \mathbf{m}_\alpha) = a^E(\mathbf{z}_h, \mathbf{m}_\alpha) \qquad \,\, \text{for all} \, \, \mathbf{m}_\alpha \in [\mathbb{P}_k(E)]^2.
	\end{align}
	\textbf{(P4)} Stability: there exist two positive constants $\lambda_1$ and $\lambda_2$ independent of $h$, such that
	\begin{align}
		\lambda_1 a^E(\mathbf{z}_h, \mathbf{z}_h) \leq a^E_{h}(\mathbf{z}_h, \mathbf{z}_h) \leq \lambda_2 a^E(\mathbf{z}_h, \mathbf{z}_h)\qquad \,\, \text{for all} \,\, \mathbf{z}_h \in \mathbf{V}^k_{h}(E).
	\end{align}
	\textbf{(P5)} Continuity: for any $\mathbf{w}_h, \mathbf{z}_h \in \mathbf{V}^k_{h}(E)$, we have
	\begin{align}
		a^E_{h}(\mathbf{w}_h, \mathbf{z}_h) \leq C \mu \max \{1, \lambda_1^\ast\} |\mathbf{w}_h|_{1,E} | \mathbf{z}_h|_{1,E}.
	\end{align}
	
	\item The discrete form $d^E_{h}(\cdot, \cdot)$ satisfies the following properties:\newline
	\textbf{(P6)} Polynomial consistency: for all $\mathbf{z}_h \in \mathbf{V}^k_{h}(E)$, we infer
	\begin{align}
		d^E_{h}(\mathbf{z}_h, \mathbf{m}_\alpha) = d^E(\mathbf{z}_h, \mathbf{m}_\alpha) \qquad  \,\, \text{for all} \, \, \mathbf{m}_\alpha \in [\mathbb{P}_k(E)]^2.
	\end{align}
	\textbf{(P7)} Stability: there exist two positive constants $\xi_1$ and $\xi_2$ independent of the mesh size, satisfying
	\begin{align}
		\xi_1 d^E(\mathbf{z}_h, \mathbf{z}_h) \leq d^E_{h}(\mathbf{z}_h, \mathbf{z}_h) \leq \xi_2 d^E(\mathbf{z}_h, \mathbf{z}_h)\qquad \text{for all} \,\, \mathbf{z}_h \in \mathbf{V}^k_{h}(E).
	\end{align}
	\textbf{(P8)} Continuity: for every $\mathbf{w}_h, \mathbf{z}_h \in \mathbf{V}^k_{h}(E)$, it gives
	\begin{align}
		d^E_{h}(\mathbf{w}_h, \mathbf{z}_h) \leq C \gamma \max \{1, \lambda_2^\ast\} \|\mathbf{w}_h\|_{E} \| \mathbf{z}_h\|_{E}.
	\end{align}
	
	\item \textbf{Discrete load term}: $(\mathbf{f}_h, \mathbf{z}_h):= \sum \limits_{E \in \Omega_{h}} (\mathbf{f}, \boldsymbol{\Pi}^{0,E}_k\mathbf{z}_h) \qquad \text{for all} \,\,\mathbf{z}_h \in \mathbf{V}^k_{h} $.
\end{itemize}

\begin{remark}
	The classical virtual element (VE) problem for the Oseen equation is given as follows:
	\begin{align}
		\begin{cases}
			\text{Find}\,\, (\mathbf{u}_h, p_h) \in \mathbf{V}^k_h \times Q^k_h,\,\, \text{such that} \\
			A_h[(\mathbf{u}_h, p_h), (\mathbf{w}_h, q_h)]= (\mathbf{f}_h, \mathbf{w}_h) \quad  \,\, \text{for all} \, \, (\mathbf{w}_h, q_h) \in \mathbf{V}^k_h \times Q^k_h,
		\end{cases}
		\label{std_vem}
	\end{align}
	where $A_h[(\mathbf{u}_h, p_h), (\mathbf{w}_h,q_h) ]:= a_h(\mathbf{u}_h, \mathbf{w}_h)- b_h(\mathbf{w}_h, p_h) + b_h(\mathbf{u}_h, q_h) + c^{skew}_h(\mathbf{u}_h, \mathbf{w}_h) + d_h(\mathbf{u}_h, \mathbf{w}_h)$. The VE problem \eqref{std_vem} exhibits non-physical oscillations whenever problem \eqref{std_vem} is convection-dominated or equal-order element pairs are employed. In such cases, a stabilized form of the VE problem \eqref{std_vem} is required in order to circumvent these issues.
\end{remark}

\subsection{ Stabilized virtual element problem }
\label{sec-3.4}
In this section, we will address the instability of velocity and pressure fields that disrupt {solution's accuracy} by causing spurious or non-physical oscillations near boundary layers. The velocity field becomes unstable {in the presence} of convection-dominated regimes, whereas the pressure instability often arises due to the equal-order element pairs framework. To address these issues, we introduce the following stabilization terms:

\begin{itemize}[label=\textbullet]
	\item Velocity field stabilizing bilinear form
	\begin{align*}
		\mathcal{L}_{1,h}(\mathbf{w}_h,\mathbf{z}_h) &:= \sum_{E \in \Omega_h}\tau_{1,E}  \mathcal{B}_E^2 S^E_\nabla\Big((\mathbf{I}-\boldsymbol{\Pi}^{\nabla,E}_{k-1}) \mathbf{w}_h,(\mathbf{I}-\boldsymbol{\Pi}^{\nabla,E}_{k-1}) \mathbf{z}_h \Big)\Big],
	\end{align*}
	where  $\mathcal{B}_E=\|\boldsymbol{\mathcal{B}}\|_{0,\infty,E}$ for all $E \in \Omega_{h}$. {Furthermore, using the norm equivalence discussed in \cite{vem28}, we obtain
		\begin{align}
			\|\nabla(\mathbf{I}-\boldsymbol{\Pi}^{\nabla,E}_{k-1}) \mathbf{w}_h\|_{0,E} \simeq h^{-1}_E	\|(\mathbf{I}-\boldsymbol{\Pi}^{0,E}_{k-1}) \mathbf{w}_h\|_{0,E},
		\end{align}
		this allows us to take projection onto the polynomial space $[\mathbb{P}_{k-1}(E)]^2$. Another advantage of this term is that it will be effective for the VEM order $k=1$ on triangular meshes, whereas the velocity stabilizing term $S_1(\cdot,\cdot)$ discussed in \cite{li2024stabilized} vanishes on triangular meshes with VEM order $k=1$.}
	\item The mass equation stabilizing term
	\begin{align*}
		\mathcal{L}_{2,h}(\mathbf{w}_h,\mathbf{z}_h) &:= \sum_{E \in \Omega_h}\tau_{2,E} \Big[\Big( \Pi^{0,E}_{k-1}\nabla \cdot \mathbf{w}_h, \Pi^{0,E}_{k-1} \nabla \cdot \mathbf{z}_h \Big)  + S^E_\nabla\Big((\mathbf{I}-\boldsymbol{\Pi}^{\nabla,E}_k) \mathbf{w}_h,(\mathbf{I}-\boldsymbol{\Pi}^{\nabla,E}_k) \mathbf{z}_h \Big)\Big]. 
	\end{align*}
	\item Pressure stabilizing term
	\begin{align*}
		\mathcal{L}_{3,h}(p_h,q_h) := \sum_{E \in \Omega_h}\tau_{3,E} \Big[\Big( \widehat{\mathbf{r}}_h(\nabla p_h), \widehat{\mathbf{r}}_h(\nabla q_h) \Big) + S^E_p\Big((I-\Pi^{\nabla,E}_{k-1}) p_h, (I-\Pi^{\nabla,E}_{k-1}) q_h \Big)\Big], 
	\end{align*}
	where $\widehat{\mathbf{r}}_h := \boldsymbol{\Pi}^{0,E}_k - \boldsymbol{\Pi}^{0,E}_{k-1}$, and $\tau_i$ are stabilization parameters for $i=1,2$ and $3$. 
	Additionally, the bilinear form $S^E_p(\cdot,\cdot)$ satisfies the following  
	\begin{align}
		\lambda_{3\ast} \ \|\nabla q_h\|_{0,E}^2 \leq S^E_p(q_h, q_h) \leq
		\lambda_3^\ast  \|\nabla q_h\|_{0,E}^2 \qquad  \,\, \text{for all}\,\, q_h \in Q^k_{h}(E) \cap ker({\Pi}^{\nabla,E}_{k-1}), \label{vem-c}
	\end{align}
	where the positive constants $\lambda_{3\ast} \leq \lambda_3^\ast$ independent of the mesh size.
\end{itemize} 
\begin{remark} \label{rem_tau}
	Recalling \cite{smfem1,mishra2025equal,mvem21}, we introduce the following definition of the stabilization parameters $\tau_{i,E}$, as follows
	\begin{align}
		\tau_{1,E} \sim h_E, \qquad \tau_{2,E} \sim \mathcal{O}(1), \qquad \tau_{3,E} \sim h^2_E.
	\end{align}  
	Moreover, we will use the above definitions throughout this work.
\end{remark}

We now set $\mathcal{L}_{h}[(\mathbf{w}_h, p_h), (\mathbf{z}_h,q_h)]:=\mathcal{L}_{1,h}(\mathbf{w}_h,\mathbf{z}_h) + \mathcal{L}_{2,h}(\mathbf{w}_h,\mathbf{z}_h)+ \mathcal{L}_{3,h}(p_h, q_h)$. Therefore, the stabilized VE problem for the Oseen equation is given as follows:
\begin{align}
	\begin{cases}
		\text{Find}\,\, (\mathbf{u}_h, p_h) \in \mathbf{V}^k_h \times Q^k_h,\,\, \text{such that} \\
		A_h[(\mathbf{u}_h, p_h), (\mathbf{w}_h, q_h)] + \mathcal{L}_{h}[(\mathbf{u}_h, p_h), (\mathbf{w}_h,q_h)] = (\mathbf{f}_h, \mathbf{w}_h) \quad  \,\, \text{for all} \, \, (\mathbf{w}_h, q_h) \in \mathbf{V}^k_h \times Q^k_h.
	\end{cases}
	\label{nvem}
\end{align}

\begin{remark}
	The problem \eqref{nvem} can be interpreted as the extension of the classic VE problem through the inclusion of element-wise local projection-based stabilizing terms {\cite{vem028,fem11}}. The formulation of these stabilization terms does not involve the second-order derivative terms like SUPG-stabilizing/least-square stabilizing methods.
\end{remark}

\section{Theoretical analysis} \label{sec-4}
This section delineates the establishment of the existence and uniqueness of the stabilized discrete solution. Additionally, we derive the optimal error estimate in the energy norm.
\subsection{Preliminary results}
In this section, we recall the following results under the assumption \textbf{(A1)}:
\begin{lemma}(Polynomial approximation \cite{scott})
	\label{lemmaproj1}
	For any sufficiently smooth functions $z \in H^s(E)$ with $E\in \Omega_h$, we have the following
	\begin{align}
		\|z -\Pi^{\nabla}_k z\|_{l,E} &\leq C h^{s-l}_E |z|_{s,E} \quad s, l \in \mathbb{N}, \,\,\, s \geq 1,\,\, \, l \leq s \leq k+1,\label{eqproj1}	 \\
		\|z - \Pi^{0}_k \,z \|_{l,E} &\leq C h^{s-l}_E |z|_{s,E} \quad s, l \in \mathbb{N},\,\,\, l \leq s \leq k+1.
	\end{align}
\end{lemma}

\begin{lemma}(Interpolation approximation {\cite{cangiani2017posteriori}})\label{lemmaproj2}
	Let the assumption \textbf{(A1)} be true. Then for any $z \in  H^{s+1}(E)$, there exists $z_I \in  V^k_h(E)$ with $E\in \Omega_h$ satisfying
	\begin{equation}
		\| z - z_I\|_{0,E} + h |z - z_I|_{1,E} \leq C h^{s+1}\, \|z \|_{s+1,E},\quad 0 \leq s\leq k, \label{projlll2}
	\end{equation}
	where $C$ is a positive constant depending only on $k$ and the mesh regularity constant $\theta$.
\end{lemma}

\begin{lemma}
	(Inverse inequality for virtual elements \cite{vem28}) \label{inverse}
	For any  ${z}_h \in V^k_h(E)$ with $E \in \Omega_{h}$, there exists a positive constant $C$ independent of the decomposition of $\Omega$ such that
	\begin{align}
		|{z}_h|_{1,E} \leq C h^{-1}_E \|{z}_h\|_{0,E}.
	\end{align}
\end{lemma}

\begin{lemma}(Trace inequality \cite{scott})\label{trace} 
	For any ${z} \in H^1(E)$ with $E \in \Omega_{h}$, we have the following result
	\begin{align}
		\|{z} \|^2_{0,\partial E} \leq C \big[h^{-1}_E \|{z}\|^2_{0,E} + h_E | z|^2_{1,E} \big],
	\end{align}
	where the constant $C$ is independent of the mesh size.
\end{lemma}

\begin{remark} \label{estimate}
	For any $E \in \Omega_h$ and for any {${z} \in H^1(E)$}, employing the orthogonality of the projectors, we can derive the following results:
	\begin{align}
		\|(\mathbf{I}-\boldsymbol{\Pi}^{0,E}_{k-1})\nabla {z} \|_{0,E} & \leq \| \nabla ({I}- {\Pi}^{\nabla,E}_{k}){z}\|_{0,E} \leq  \| \nabla (I- {\Pi}^{\nabla,E}_{k-1}) {z}\|_{0,E}, \label{est-a}\\
		\|\nabla{\Pi}^{\nabla,E}_k {z} \|_{0,E} & \leq  \| \nabla {z}\|_{0,E}.
	\end{align}
\end{remark}

\begin{lemma} \label{estimate2}
	For any  ${z}_h \in V^k_h(E)$, it holds the following:
	\begin{align}
		\| \nabla \Pi^{0,E}_k z_h\|_{0,E} &\leq C\|\nabla z_h\|_{0,E}, \label{est1}\\
		\| \nabla ({I}- {\Pi}^{0,E}_{k}){z}_h\|_{0,E} &\leq C \| \nabla (I- {\Pi}^{\nabla,E}_{k}) {z}_h\|_{0,E}. \label{est3}
		\intertext{Additionally, for any $\mathbf{z}_h \in \mathbf{V}^k_h(E)$, we infer}
		\|({I}-{\Pi}^{0,E}_{k-1}) \nabla \cdot \mathbf{z}_h \|_{0,E} & \leq  C \| \nabla (
		\mathbf{I}- \boldsymbol{\Pi}^{\nabla,E}_{k}) \mathbf{z}_h\|_{0,E}. \label{est2}
	\end{align}
\end{lemma}
\begin{proof}
	Concerning \eqref{est1},  we use Lemmas \ref{inverse} and \ref{lemmaproj1}:
	\begin{align}
		\| \nabla \Pi^{0,E}_k z_h\|_{0,E} &\leq C \big(\| \nabla z_h\|_{0,E} + \|\nabla(I-\Pi^{0,E}_k) z_h \|_{0,E}\big) \nonumber \\
		& \leq C \big(\| \nabla z_h\|_{0,E} + h^{-1}_E \|(I-\Pi^{0,E}_k) z_h \|_{0,E}\big) \nonumber \\
		& \leq C \big(\| \nabla z_h\|_{0,E} + h^{-1}_E h_E \| \nabla z_h \|_{0,E}\big) \nonumber \\ & \leq C \|\nabla z_h\|_{0,E}. \nonumber
	\end{align}
	Concerning \eqref{est3}, we proceed as follows
	\begin{align}
		\| \nabla ({I}- {\Pi}^{0,E}_{k}){z}_h\|_{0,E} &\leq C h^{-1}_E  \|({I}- {\Pi}^{0,E}_{k}){z}_h\|_{0,E} \nonumber \\
		& \leq C h^{-1}_E  \|({I}- {\Pi}^{0,E}_{k})({I}- {\Pi}^{\nabla,E}_{k}){z}_h\|_{0,E} \nonumber \\
		& \leq C h^{-1}_E h_E \| \nabla({I}- {\Pi}^{\nabla,E}_{k}){z}_h\|_{0,E}. \nonumber
	\end{align}
	Using the definition of the projection operator, we infer
	\begin{align}
		\|( I- \Pi^{0,E}_{k-1}) \nabla \cdot \mathbf{z}_h \|^2_{0,E} &= \big( (I- \Pi^{0,E}_{k-1}) \nabla \cdot \mathbf{z}_h, (I- \Pi^{0,E}_{k-1}) \nabla \cdot \mathbf{z}_h  \big)  \nonumber \\
		& =\big( (I- \Pi^{0,E}_{k-1}) \nabla \cdot \mathbf{z}_h,  \nabla \cdot \mathbf{z}_h  \big)  \nonumber \\
		& =\big( (I- \Pi^{0,E}_{k-1}) \nabla \cdot \mathbf{z}_h,  \nabla \cdot \mathbf{z}_h - \nabla \cdot \boldsymbol{\Pi}^{\nabla, E}_k \mathbf{z}_h  \big)  \nonumber \\
		& \leq  \|(I- \Pi^{0,E}_{k-1}) \nabla \cdot \mathbf{z}_h\|_{0,E}  \|\nabla \cdot (\mathbf{z}_h -  \boldsymbol{\Pi}^{\nabla, E}_k \mathbf{z}_h) \|_{0,E}  \nonumber \\
		& \leq C \|(I- \Pi^{0,E}_{k-1}) \nabla \cdot \mathbf{z}_h\|_{0,E}  \|\nabla (\mathbf{z}_h -  \boldsymbol{\Pi}^{\nabla, E}_k \mathbf{z}_h) \|_{0,E}. \nonumber
	\end{align}
	Thus, the result \eqref{est2} can be easily obtained from the above analysis.
\end{proof}
Additionally, applying the Lemma \ref{trace} and Lemma \ref{inverse}, we observe that for any $\mathbf{z}_h \in \mathbf{V}^k_h$
\begin{align}
	\| (\mathbf{I}- \boldsymbol{\Pi}^{0,E}_k) \mathbf{z}_h\|_{0,\partial E}& \leq C \big[ h^{-1/2}_E\| (\mathbf{I}- \boldsymbol{\Pi}^{0,E}_k) \mathbf{z}_h\|_{0, E} + h^{1/2}_E \| \nabla(\mathbf{I}- \boldsymbol{\Pi}^{0,E}_k) \mathbf{z}_h\|_{0,E}\big] \nonumber \\
	& \leq C  h^{-1/2}_E \| (\mathbf{I}- \boldsymbol{\Pi}^{0,E}_k) \mathbf{z}_h\|_{0, E}  \nonumber \\
	& \leq C  h^{1/2}_E \| \nabla \mathbf{z}_h\|_{0, E}. \label{bd1}
\end{align}
\subsection{Weak inf-sup condition}
Hereafter, we introduce the following result, which can be easily derived using  \cite[Thereom 3.1]{vem2} and \cite[Lemma 4.3]{vem028}:
\begin{lemma} \label{Beiga_1}
	Under the mesh regularity assumption \textbf{(A1)}, there exists a positive constant $\Theta_0$ independent of h, such that for any $k\geq 2$, there holds that
	\begin{align}
		\sup \limits_{\mathbf{0} \neq \mathbf{u}_h \in \mathbf{V}_{h}} \dfrac{b_{h}(\mathbf{u}_h, p_h)}{\|\mathbf{u}_h\|_{1,\Omega}} \geq \Theta_0 \|\Pi^{\nabla}_{k-1} p_h\|_{0,\Omega}  \quad \,\, \text{for all} \, \, p_h \in Q^k_h. \label{inf-sup}
	\end{align}
\end{lemma} 
The weak inf-sup condition for the discrete bilinear form $b_h(\cdot, \cdot)$ is given as follows: 
\begin{theorem}(weak inf-sup condition {\cite{smfem1,vem028})} \label{infsup} 
	Under the mesh regularity assumption \textbf{(A1)} {with quasi-uniform mesh}, there exist positive constants $\Theta_1$ and $\Theta_2$ independent of h, satisfying the following
	\begin{align}
		\sup \limits_{\mathbf{0} \neq \mathbf{u}_h \in \mathbf{V}^k_{h}} \dfrac{b_{h}(\mathbf{u}_h, p_h)}{\|\mathbf{u}_h\|_{1,\Omega}} \geq \Theta_1 \|p_h\|_{0,\Omega} - \Theta_2 [\mathcal{L}_{3,h}(p_h,p_h)]^{\frac{1}{2}} \quad \, \, \text{for all}\,\, p_h \in Q^k_h.  \label{inf-sup1}
	\end{align}
\end{theorem}

\begin{proof}	
	We derive the proof of Theorem \ref{infsup} in the following steps: \newline
	\textbf{Step 1.} Form the continuous inf-sup condition, it is obvious that for any $p_h \in Q_h^k$ there exists $\mathbf{v} \in \mathbf{V}$ such that
	\begin{align}
		\big( \nabla \cdot \mathbf{v}, p_h\big) \geq C_1 \|\mathbf{v} \|_{1,\Omega} \|p_h\|_{0,\Omega}. \label{weak-01}
	\end{align}	
	Let $\mathbf{v}_I \in \mathbf{V}_{h}$ be the virtual element approximation of $\mathbf{v} \in \mathbf{V}$. {Recalling Lemma \ref{lemmaproj2}, we have}
	\begin{align}
		\|\mathbf{v} - \mathbf{v}_I\|_{0,\Omega} \leq C h \|\mathbf{v}\|_{1,\Omega} \qquad \text{and} \quad \,\, \|\mathbf{v}_I\|_{1,\Omega} \leq C \|\mathbf{v}\|_{1,\Omega}. \label{weak-11}
	\end{align}	
	We employ the orthogonality property of the projectors and \eqref{weak-11}, it holds that
	\begin{align}
		\sup \limits_{\mathbf{u}_h \in \mathbf{V}^k_h} \dfrac{b_h(\mathbf{u}_h, p_h)}{\|\mathbf{u}_h\|_{1,\Omega}} & = \sup \limits_{\mathbf{u}_h \in \mathbf{V}^k_h} \dfrac{ \sum_{E \in \Omega_h} \big( \Pi^{0,E}_{k-1} \nabla \cdot \mathbf{u}_h, \Pi^{0,E}_k p_h \big)}{\|\mathbf{u}_h\|_{1,\Omega}} \nonumber \\
		& \geq \dfrac{ \sum_{E \in \Omega_h} \big( \Pi^{0,E}_{k-1} \nabla \cdot \mathbf{v}_I, \Pi^{0,E}_k p_h \big)}{\|\mathbf{v}_I\|_{1,\Omega}} \nonumber \\
		& \geq \dfrac{ \sum_{E \in \Omega_h} \big(  \nabla \cdot \mathbf{v}_I, \Pi^{\nabla,E}_{k-1} p_h \big)}{C \|\mathbf{v}\|_{1,\Omega}} \nonumber \\
		& \geq - \dfrac{| \sum_{E \in \Omega_h} \big(  \nabla \cdot \mathbf{v}_I, p_h -\Pi^{\nabla,E}_{k-1}p_h \big)|}{C \|\mathbf{v}\|_{1,\Omega}} - \dfrac{| \sum_{E \in \Omega_h} \big( \nabla \cdot (\mathbf{v} - \mathbf{v}_I), p_h \big)|}{C \|\mathbf{v}\|_{1,\Omega}} + \dfrac{ \big(  \nabla \cdot \mathbf{v},  p_h \big)}{C \|\mathbf{v}\|_{1,\Omega}} \label{weak-21}.
	\end{align}
	{Recalling the Cauchy-Schwarz inequality and the bound \eqref{weak-11}, we infer
		\begin{align}
			\big| \sum_{E \in \Omega_h} \big(  \nabla \cdot \mathbf{v}_I, p_h -\Pi^{\nabla,E}_{k-1}p_h \big) \big| &\leq  \sum_{E \in \Omega_h} \| p_h -\Pi^{\nabla,E}_{k-1}p_h \|_{0,E} \|\nabla \mathbf{v}_I\|_{0,E} \nonumber \\
			& \leq  \| p_h -\Pi^{\nabla}_{k-1}p_h \|_{0,\Omega} \| \mathbf{v}\|_{1,\Omega}. \label{infsup_a1}
		\end{align}
		Applying integration by parts, we obtain
		\begin{align}
			\big| \sum_{E \in \Omega_h} \big( \nabla \cdot (\mathbf{v} - \mathbf{v}_I), p_h \big) \big| & = \big| -\sum_{E \in \Omega_h} \big( \mathbf{v} - \mathbf{v}_I, \nabla p_h \big) + \sum_{E \in \Omega_h} \int_{\partial E} (\mathbf{v} - \mathbf{v}_I) \cdot \mathbf{n}^E p_h ds  \big| \nonumber
			\intertext{Using the continuity property of $\mathbf{v}, \mathbf{v}_I \in [H^1_0(\Omega)]^2$ and $p_h \in H^1(\Omega)$, the boundary integral vanishes. Concerning the first term, we use \eqref{weak-11}:  }
			\big| \sum_{E \in \Omega_h} \big( \nabla \cdot (\mathbf{v} - \mathbf{v}_I), p_h \big) \big| & \leq \sum_{E \in \Omega_h} \| \mathbf{v} - \mathbf{v}_I\|_{0,E} \|\nabla p_h \|_{0,E} \nonumber \\
			& \leq C h \|\nabla p_h\|_{0,\Omega} \|\mathbf{v}\|_{1,\Omega}. \label{infsup_a2}
		\end{align}
		Combining the bounds \eqref{weak-21}, \eqref{infsup_a1}, and \eqref{infsup_a2}, and using the result \eqref{weak-01}, we obtain
		\begin{align}
			\sup \limits_{\mathbf{u}_h \in \mathbf{V}^k_h} \dfrac{b_h(\mathbf{u}_h, p_h)}{\|\mathbf{u}_h\|_{1,\Omega}} 
			& \geq - C\Big[\| (I- \Pi^{\nabla}_{k-1}) p_h \|_{0,\Omega} + 
			h \| \nabla p_h \|_{0,\Omega}  \Big] + C_1 \| p_h \|_{0,\Omega}. \label{infsup_a3}  
	\end{align}}
	\textbf{Step 2.} Using the Poincar$\acute{\text{e}}$ inequality \cite[Lemma 2.2]{vem28}, we infer
	\begin{align}
		\| (I- \Pi^{\nabla}_{k-1}) p_h \|^2_{0,\Omega} &= \sum_{E \in \Omega_h} \| (I- \Pi^{\nabla,E}_{k-1}) p_h \|^2_{0,E} \leq C \sum_{E \in \Omega_h} h^2_E \| \nabla(I- \Pi^{\nabla,E}_{k-1}) p_h \|^2_{0,E}. \label{w4}
	\end{align}
	{Employing the property of $\Pi^\nabla_0$, we infer
		\begin{align}
			h\|\nabla p_h\|_{0, \Omega} = h \|\nabla(I - \Pi^\nabla_0) p_h\|_{0, \Omega}. \label{infsup_a4}
	\end{align}}
	Substituting \eqref{w4} and \eqref{infsup_a4} in \eqref{infsup_a3}, we obtain for $k=1$
	\begin{align}
		\sup \limits_{\mathbf{u}_h \in \mathbf{V}^k_h} \dfrac{b_h(\mathbf{u}_h, p_h)}{\|\mathbf{u}_h\|_{1,\Omega}} & \geq - C  h \| \nabla (I- \Pi^{\nabla}_{0}) p_h \|_{0,\Omega}  + C_1 \| p_h \|_{0,\Omega}. \label{weak-03}
	\end{align}
	Recalling Lemma \ref{inverse}, it holds that
	\begin{align}
		h^2 \| \nabla p_h \|^2_{0,\Omega} &\leq C \sum_{E \in \Omega_h} \big[h^2 \| \nabla \Pi^{\nabla,E}_{k-1} p_h \|^2_{0,E} + 	h^2 \| \nabla(I-\Pi^{\nabla,E}_{k-1} )p_h\|^2_{0,E} \big] \nonumber \\
		&\leq C \sum_{E \in \Omega_h} \big[ \| \Pi^{\nabla,E}_{k-1} p_h \|^2_{0,E} + 	h^2 \| \nabla(I-\Pi^{\nabla,E}_{k-1} )p_h\|^2_{0,E} \big] \nonumber \\
		&\leq C \big[ \| \Pi^{\nabla}_{k-1} p_h \|^2_{0,\Omega} + 	h^2 \| \nabla(I-\Pi^{\nabla}_{k-1} )p_h\|^2_{0,\Omega} \big]. \label{w5}
	\end{align}
	We now combine \eqref{infsup_a3}, \eqref{w4} and \eqref{w5}, and we conclude that the following holds for $k\geq 2$
	\begin{align}
		\sup \limits_{\mathbf{u}_h \in \mathbf{V}^k_h} \dfrac{b_h(\mathbf{u}_h, p_h)}{\|\mathbf{u}_h\|_{1,\Omega}} 
		&\geq - C \Big[ \|  \Pi^{\nabla}_{k-1} p_h \|_{0,\Omega} + h \| \nabla(I- \Pi^{\nabla}_{k-1}) p_h \|_{0,\Omega}  \Big] + C_1 \| p_h \|_{0,\Omega}. \label{weak-22}
	\end{align}
	Using Lemma \ref{Beiga_1}, it gives
	\begin{align}
		\sup \limits_{\mathbf{u}_h \in \mathbf{V}^k_h} \dfrac{b_h(\mathbf{u}_h, p_h)}{\|\mathbf{u}_h\|_{1,\Omega}} & \geq - C  h \| \nabla (I- \Pi^{\nabla}_{k-1}) p_h \|_{0,\Omega}  +  {\Theta_1} \| p_h \|_{0,\Omega} \qquad \,\, \text{for all} \,\, k\geq 2. \label{weak-04}
	\end{align}	
	\textbf{Step 3.} Applying Remark \ref{estimate} and \eqref{vem-c}, the following hold for any $k \in \mathbb N$
	\begin{align}
		h^2\| \nabla (I- \Pi^{\nabla}_{k-1}) p_h \|^2_{0,\Omega} & \leq	h^2 \sum_{E \in \Omega_h } \| \nabla (I- \Pi^{\nabla,E}_{k-1}) p_h \|^2_{0,E} \nonumber \\
		& \leq h^2 \sum_{E \in \Omega_h} \Big[ \|\boldsymbol{\Pi}^{0,E}_k \nabla p_h- \boldsymbol{\Pi}^{0,E}_{k-1} \nabla p_h\|^2_{0,E} +  \| \nabla p_h- \nabla \Pi^{\nabla,E}_{k-1} p_h\|^2_{0,E} \Big] \nonumber \\
		&\leq \frac{ h^2}{ \min_{E \in \Omega_h}\tau_{3,E}} \sum_{E \in \Omega_h} \Big[ \tau_{3,E} \|\boldsymbol{\Pi}^{0,E}_k \nabla p_h- \boldsymbol{\Pi}^{0,E}_{k-1} \nabla p_h\|^2_{0,E}\, + \nonumber \\ & \qquad \frac{\tau_{3,E}}{\lambda_{3\ast}} S^E_p \big( (I-\Pi^{\nabla,E}_{k-1})p_h,(I-\Pi^{\nabla,E}_{k-1})p_h \big) \Big] \nonumber \\
		& \leq \frac{ h^2}{ \min_{E \in \Omega_h}\tau_{3,E}} \max \Big\{1, \frac{1}{\lambda_{3\ast}} \Big\} \mathcal{L}_{3,h}(p_h,p_h). \label{w6}
		\intertext{Following \cite{smfem1},  we assume  $\tau_{3,E} \sim h^2_E$. Consequently, \eqref{w6} gives}
		h^2 \| \nabla (I- \Pi^{\nabla}_{k-1}) p_h \|^2_{0,\Omega} & \leq C \max \Big\{1, \frac{1}{\lambda_{3\ast}} \Big\} \mathcal{L}_{3,h}(p_h,p_h). \label{w7}
	\end{align}
	Finally, combining the bounds \eqref{weak-03}, \eqref{weak-04} and \eqref{w7}, we easily obtain the result \eqref{inf-sup1}. 
\end{proof}

\subsection{Stability}
{To establish the well-posedness and error estimates for the stabilized VE problem \eqref{nvem}, we introduce a mesh-dependent energy norm  
	\begin{align}
		\vertiii{(\mathbf{z}_h, q_h)}^2 =  \mu \|\nabla \mathbf{z}_h\|^2_{0, \Omega} +  \gamma \| \mathbf{z}_h\|^2_{0, \Omega} + \alpha \|q_h\|^2_{0, \Omega} + \mathcal{L}_{1,h}(\mathbf{z}_h,\mathbf{z}_h) + \mathcal{L}_{2,h}(\mathbf{z}_h,\mathbf{z}_h) + \mathcal{L}_{3,h}(q_h,q_h), \label{tnorm}
	\end{align}
	and the constant $\alpha$ is given as follows
	\begin{align}
		\alpha = \begin{cases}
			\min \Big\{ \dfrac{ C_{\lambda_{1\ast}} \mu}{8C^2_{gen} (1 + \mu + \mathcal{B} + \mathcal{B}^2)^2}, \dfrac{C_{\lambda_{2\ast}}}{8 C^2_{gen} \gamma}, \dfrac{ \Theta_1}{\Theta_2^2 \xi_0}  \Big\} \qquad \text{for diffusion-dominated case}, \\ \\
			\dfrac{4 \Theta_1}{2 + 5 \Theta_2 + 4(1+ C_{gen}) (\mu + C_P^2 \gamma + C_{gen})^{1/2} + \frac{8C_P^2 C_{gen}^2 \mathcal{B}^2}{\mu + C_P^2 \gamma}} \qquad \text{for convection-dominated case $\mu \rightarrow 0$}, \nonumber
		\end{cases}
	\end{align}
	where
	\begin{align}
		C_{\lambda1\ast}&:= \min\{ 1, \lambda_{1\ast}\}, &C_{\lambda2\ast} &:= \min\{ 1, \lambda_{2\ast}\},   \nonumber \\ 
		\mathcal{B}&:=\max \limits_{E \in \Omega_{h}}\|\boldsymbol{\mathcal{B}}\|_{0,\infty,E}, & \xi_0&:=\min \Big\{ 1, \dfrac{\Theta_1}{2} \Big\}, & C_{gen}&:= \max \big\{1, \lambda_1^\ast, C, C|\Omega|^{1/2}, \lambda_1^\ast |\Omega|^{1/2}, \lambda_2^\ast C \big\}, \label{reg1}
	\end{align}
	with the positive constant $C$ depends on the inverse inequality constant, the Poincar$\acute{\text{e}}$ constant and the trace inequality constant, and independent of $h$. \newline
	The above definition of $\alpha$ concerning convection-dominated regimes is also applicable for diffusion-dominated regimes. Here we first focus on demonstrating the stability of the discrete problem \eqref{nvem} for the diffusion-dominated case, which can be applicable in developing stability of coupled Navier-Stokes problems \cite{mvem3} or the Boussinesq problem. The well-posedness theorem is given as follows:}

\begin{theorem}{(Well-posedness)} \label{wellposed}
	Under the mesh regularity assumption \textbf{(A1)}, there exists a positive constant $\Theta$ independent of the mesh size $h$, such that for any $(\mathbf{u}_h, p_h) \in \mathbf{V}^k_{h} \times Q^k_h$, there holds that 
	\begin{align}
		\sup \limits_{(\mathbf{0},0) \neq(\mathbf{z}_h, q_h) \in \mathbf{V}^k_{h} \times Q^k_h} \dfrac{(A_{h}+\mathcal{L}_h)[(\mathbf{u}_h, p_h), (\mathbf{z}_h, q_h)]}{\vertiii{(\mathbf{z}_h, q_h)}}  \geq \Theta \vertiii{(\mathbf{u}_h, p_h)}, \label{wellpos-0}
	\end{align}
	Moreover, the stabilized VE problem \eqref{nvem} has a unique solution.
\end{theorem}
\begin{proof}
	The proof of Theorem  \ref{wellposed} is accomplished through the following steps: \newline
	\textbf{Step 1.} Replacing $(\mathbf{z}_h, q_h)$ by $(\mathbf{u}_h, p_h)$ in the definition of $A_h[\cdot,\cdot]$ and $\mathcal{L}_h[\cdot,\cdot]$, it gives 
	\begin{align}
		(A_h + \mathcal{L}_h)&[(\mathbf{u}_h, p_h), (\mathbf{u}_h, p_h)]\nonumber \\ &=a_{h}(\mathbf{u}_h,\mathbf{u}_h) + d_{h}(\mathbf{u}_h,\mathbf{u}_h) + \mathcal{L}_{1,h}(\mathbf{u}_h,\mathbf{u}_h)+ \mathcal{L}_{2,h}(\mathbf{u}_h,\mathbf{u}_h)+ \mathcal{L}_{3,h}(p_h,p_h) \nonumber \\ 
		& \geq  \sum_{E \in \Omega_{h}}  \Big( \min\{ 1, \lambda_{1\ast}\} \mu \| \nabla \mathbf{u}_h\|^2_{0, E}  + \min\{ 1, \lambda_{2\ast}\}\gamma \|\mathbf{u}_h\|^2_{0,E} + \mathcal{L}^E_{1,h}(\mathbf{u}_h,\mathbf{u}_h) \nonumber \\ & \qquad + \mathcal{L}^E_{2,h}(\mathbf{u}_h,\mathbf{u}_h)+ \mathcal{L}^E_{3,h}(p_h,p_h) \Big) \nonumber \\ 
		& \geq   \Big( C_{\lambda1\ast} \mu \| \nabla \mathbf{u}_h\|^2_{0, \Omega}  + C_{\lambda2\ast}\gamma \|\mathbf{u}_h\|^2_{0,\Omega} + \mathcal{L}_{1,h}(\mathbf{u}_h,\mathbf{u}_h) + \mathcal{L}_{2,h}(\mathbf{u}_h,\mathbf{u}_h)+ \mathcal{L}_{3,h}(p_h,p_h) \Big). \label{wellpos-1}
	\end{align}
	\textbf{Step 2.} It is evident from Theorem \ref{infsup} that for any discrete pressure $q_h \in Q_h^k \subset Q$  there exists a discrete velocity $\mathbf{r}_h \in \mathbf{V}^k_h$ such that
	\begin{align}
		b_{h}(\mathbf{r}_h, {p_h}) \geq \|\mathbf{r}_h\|_{1,\Omega} \Big(\Theta_1 \|{p_h}\|_{0,\Omega} - \Theta_2 [\mathcal{L}_{3,h}({p_h},{p_h})]^{\frac{1}{2}} \Big). \nonumber 
	\end{align}
	We now choose $\mathbf{w}_h =\dfrac{ \xi_0\|p_h\|_{0,\Omega}}{\| \nabla \mathbf{r}_h\|_{0,\Omega}} \mathbf{r}_h$, then we obtain the following:
	\begin{align}
		\|\nabla \mathbf{w}_h\|_{0,\Omega} = \xi_0 \|p_h\|_{0,\Omega}, \qquad \,\,\text{with} \,\, \xi_0=\min \big\{ 1, \dfrac{\Theta_1}{2} \big\}. \label{wellpos-2}
	\end{align}
	Employing Eq. \eqref{wellpos-2} , we infer 
	\begin{align}
		b_{h}(\mathbf{w}_h, p_h) \geq \xi_0\|p_h\|_{0,\Omega} \Big(\Theta_1 \|p_h\|_{0,\Omega} - \Theta_2 [\mathcal{L}_{3,h}(p_h,p_h)]^{\frac{1}{2}} \Big). \label{wellpos-3}
	\end{align}
	\textbf{Step 3.} Let $\alpha$ be any positive constant. Replacing $(\mathbf{z}_h, q_h)$ by $(-\alpha \mathbf{w}_h, 0)$, we arrive at 
	\begin{align}
		(A_h+\mathcal{L}_h)&[(\mathbf{u}_h, p_h), (-\alpha \mathbf{w}_h, 0)]\nonumber \\
		&= \sum_{E \in \Omega_{h}} \Big(- \alpha a^E_{h}(\mathbf{u}_h,\mathbf{w}_h) + \alpha b^E_{h}(\mathbf{w}_h,p_h) - \alpha c^{skew,E}_{h}(\mathbf{u}_h,\mathbf{w}_h) - \nonumber \\ &\qquad \qquad \alpha d^E_{h}(\mathbf{u}_h,\mathbf{w}_h) -  \alpha \mathcal{L}^E_{1,h}(\mathbf{u}_h,\mathbf{w}_h) - \alpha \mathcal{L}^E_{2,h}(\mathbf{u}_h,\mathbf{w}_h) \Big) \nonumber \\ 
		&=: -s_1 + s_2 - s_3 - s_4 -s_5 -s_6. \label{well_b}
	\end{align}
	$\bullet$ Using the definition of $a^E_h(\cdot,\cdot)$ and \eqref{vem-a}, we obtain 
	\begin{align}
		s_1 &\leq  \sum_{E \in \Omega_h} \max \{ 1, \lambda_1^{\ast} \} \alpha \mu \|\nabla \mathbf{u}_h\|_{0,E}	\|\nabla \mathbf{w}_h\|_{0,E} \nonumber \\
		& \leq   \alpha \max \{ 1, \lambda_1^{\ast} \} \mu \|\nabla \mathbf{u}_h\|_{0,\Omega} \|\nabla \mathbf{w}_h\|_{0,\Omega}. \label{wellpos-4} 
	\end{align}
	$\bullet$ Recalling \eqref{wellpos-3} and the Young inequality, it holds that
	\begin{align}
		s_2 &\geq  \alpha \xi_0 \Big( \Theta_1 \|p_h\|^2_{0,\Omega} -  \Theta_2 \|p_h\|_{0,\Omega} [\mathcal{L}_{3,h}(p_h,p_h)]^{\frac{1}{2}} \Big) \nonumber \\
		& \geq  \alpha \xi_0 \Big( \Theta_1 \|p_h\|^2_{0,\Omega} - \dfrac{\Theta^2_2}{t_1} \|p_h\|^2_{0,\Omega} - {\dfrac{t_1}{4}} \mathcal{L}_{3,h}(p_h,p_h) \Big) \nonumber \\
		& \geq  \alpha \xi_0 \Big( \Big(\Theta_1 - \dfrac{\Theta^2_2}{t_1} \Big) \|p_h\|^2_{0,\Omega} - {\dfrac{t_1}{4}} \mathcal{L}_{3,h}(p_h,p_h) \Big), \label{wellpos-5}
	\end{align}
	where $t_1$ is a positive constant. \newline
	$\bullet$ Employing the stability properties of projection operators, Lemmas \ref{trace} and \ref{estimate2}, and the bound \eqref{bd1}, we get
	\begin{align}
		s_3 &\leq  \alpha \sum_{E \in \Omega_{h}} \Big( \|c^E_h(\mathbf{u}_h, \mathbf{w}_h) \|_{0,E} + 	\|c^E_h(\mathbf{w}_h, \mathbf{u}_h{)}\|_{0,E}  \Big) \nonumber \\
		&\leq  \alpha \sum_{E \in \Omega_h} \Big( {\mathcal{B}}_E \| \nabla \mathbf{u}_h\|_{0,E} \|\mathbf{w}_h \|_{0,E} + {\mathcal{B}}_E \| (\mathbf{I}- \boldsymbol{\Pi}^{0,E}_k) \mathbf{u}_h\|_{0,\partial E} \|\boldsymbol{\Pi}^{0,E}_k\mathbf{w}_h \|_{0,\partial E } \,+ \nonumber \\ & \qquad    {\mathcal{B}}_E \| \nabla \mathbf{w}_h\|_{0,E} \|\mathbf{u}_h \|_{0,E} +  {\mathcal{B}}_E \| (\mathbf{I}- \boldsymbol{\Pi}^{0,E}_k) \mathbf{w}_h\|_{0,\partial E} \|\boldsymbol{\Pi}^{0,E}_k\mathbf{u}_h \|_{0,\partial E } \Big)	\nonumber \\
		& \leq C \alpha \sum_{E \in \Omega_{h}} {\mathcal{B}}_E \Big( \| \nabla \mathbf{u}_h\|_{0,E} \| \mathbf{w}_h \|_{0,E} + h^{1/2}_E \big(h^{-1/2}_E \| \mathbf{w}_h\|_{0,E} + h^{1/2}_E \| \nabla \mathbf{w}_h\|_{0,E}\big) \| \nabla \mathbf{u}_h\|_{0, E} \,+  \nonumber \\ & \qquad \| \nabla \mathbf{w}_h\|_{0,E} \| \mathbf{u}_h \|_{0,E} +  h^{1/2}_E \big(h^{-1/2}_E \| \mathbf{u}_h\|_{0,E} + h^{1/2}_E \| \nabla \mathbf{u}_h\|_{0,E}\big) \|\nabla \mathbf{w}_h\|_{0, E} \Big)\nonumber \\
		& \leq C \alpha \sum_{E \in \Omega_{h}} {\mathcal{B}}_E \Big( \| \nabla \mathbf{u}_h\|_{0,E} \| \mathbf{w}_h \|_{0,E} + h_E  \| \nabla \mathbf{u}_h\|_{0, E} \| \nabla \mathbf{w}_h\|_{0,E} + \| \nabla \mathbf{w}_h\|_{0,E} \| \mathbf{u}_h \|_{0,E} \Big). \nonumber
		\intertext{Applying the H$\ddot{\text{o}}$lder and Poincar$\acute{\text{e}}$ inequalities, we infer}
		s_3& \leq C \alpha \max\{1, |\Omega|^{{1/2}}\}  \mathcal{B}   \| \nabla \mathbf{u}_h\|_{0,\Omega} \|\nabla \mathbf{w}_h \|_{0,\Omega}.  \label{wellpos-6}
	\end{align}
	$\bullet$ We use the bound \eqref{vem-b}, Lemmas \ref{lemmaproj1} and \ref{inverse}, and the Poincar$\acute{\text{e}}$ inequality:
	\begin{align}
		s_4 &\leq  \alpha \max\{1, \lambda_2^\ast\} \sum_{E \in \Omega_{h}} \gamma \Big(  \|  \mathbf{u}_h\|_{0,E} \| \mathbf{w}_h \|_{0,E} +  \| (\mathbf{I}- \boldsymbol{\Pi}^{0,E}_k) \mathbf{u}_h\|_{0, E} \|(\boldsymbol{I}-\boldsymbol{\Pi}^{0,E}_k)\mathbf{w}_h \|_{0,E } \Big) \nonumber \\
		& \leq C \alpha \max\{1, \lambda_2^\ast\} \sum_{E \in \Omega_{h}} \gamma \Big( \|  \mathbf{u}_h\|_{0,E} \| \mathbf{w}_h \|_{0,E} + h^2_E \| \nabla \mathbf{u}_h\|_{0, E} \| \nabla \mathbf{w}_h \|_{0,E } \Big) \nonumber \\
		&\leq C \alpha \max\{1, \lambda_2^\ast\} \gamma   \|  \mathbf{u}_h\|_{0,\Omega} \| \mathbf{w}_h \|_{0,\Omega}\nonumber \\
		&\leq C \alpha \max\{1, \lambda_2^\ast\} \gamma   \|  \mathbf{u}_h\|_{0,\Omega} \| \nabla \mathbf{w}_h \|_{0,\Omega}. \label{wellpos-7}
	\end{align}
	$\bullet$ Employing the fact $\tau_{1,E} \sim h_E$ from Remark \ref{rem_tau}, the stability of projectors and \eqref{vem-a}, it holds that
	\begin{align}
		s_5 & \leq  \alpha \sum_{E \in \Omega_{h}} \tau_{1,E}  \lambda_1^\ast \Big(  {\mathcal{B}}_E^2\| \nabla \mathbf{u}_h\|_{0,E} \|\nabla \mathbf{w}_h\|_{0,E} \Big) \nonumber \\  
		& \leq  \lambda_1^\ast |\Omega|^{1/2} \alpha\mathcal{B}^2 \| \nabla \mathbf{u}_h\|_{0,\Omega} \|\nabla \mathbf{w}_h\|_{0,\Omega}.  \label{wellpos-8}
	\end{align}
	$\bullet$ We apply \eqref{vem-a} and the fact $\tau_{2,E} \sim \mathcal{O}(1)$:
	\begin{align}
		s_6 & \leq \alpha \sum_{E \in \Omega_{h}} \tau_{2,E} \max\{ 1, \lambda_1^\ast\}\Big(\| \nabla \mathbf{u}_h\|_{0,E} \|\nabla \mathbf{w}_h\|_{0,E} \Big) \nonumber \\  
		& \leq \max\{ 1, \lambda_1^\ast\} \alpha \| \nabla \mathbf{u}_h\|_{0,\Omega} \|\nabla \mathbf{w}_h\|_{0,\Omega}.   \label{wellpos-9}
	\end{align}
	Therefore, combining the estimates \eqref{wellpos-4}--\eqref{wellpos-9}, we infer 
	\begin{align}
		(A_h+\mathcal{L}_h)&[(\mathbf{u}_h, p_h), (-\alpha \mathbf{w}_h, 0)] \nonumber \\
		& \geq - \alpha C_{gen}\big(1 + \mu + \mathcal{B} + \mathcal{B}^2 \big) \| \nabla \mathbf{u}_h\|_{0,\Omega} \|\nabla \mathbf{w}_h\|_{0,\Omega} - \alpha C_{gen} \gamma \| \mathbf{u}_h \|_{0,\Omega} \|\nabla \mathbf{w}_h\|_{0,\Omega}  \nonumber \\ & \qquad + \alpha \xi_0 \Big( \Theta_1 -\dfrac{\Theta_2^2}{t_1}\Big) \|p_h\|^2_{0,\Omega}  - {\dfrac{\alpha \xi_0 t_1}{4}} \mathcal{L}_{3,h}(p_h,p_h)  \nonumber \\
		&\geq - \alpha \Big( 4C^2_{gen}\big(1 + \mu + \mathcal{B} + \mathcal{B}^2 \big)^2 \| \nabla \mathbf{u}_h\|^2_{0,\Omega} + \frac{\xi_0^2}{4}\|p_h\|^2_{0,\Omega} \Big) - \alpha \Big(4C^2_{gen} \gamma^2 \| \mathbf{u}_h \|^2_{0,\Omega} + \frac{\xi_0^2}{4}\|p_h\|_{0,\Omega}\Big)  \nonumber \\ & \qquad + \alpha \xi_0 \Big( \Theta_1 -\dfrac{\Theta_2^2}{t_1}\Big) \|p_h\|^2_{0,\Omega}  - \dfrac{\alpha \xi_0 t_1}{4} \mathcal{L}_{3,h}(p_h,p_h)   \qquad \qquad \text{(using \eqref{wellpos-2})} \nonumber \\
		&\geq - \alpha \Big( 4C^2_{gen}\big(1 + \mu + \mathcal{B} + \mathcal{B}^2 \big)^2 \| \nabla \mathbf{u}_h\|^2_{0,\Omega} + 4C^2_{gen} \gamma^2 \| \mathbf{u}_h \|^2_{0,\Omega} + \frac{\xi_0^2}{2}\|p_h\|^2_{0,\Omega} \Big)   \nonumber \\ & \qquad + \alpha \xi_0 \Big( \Theta_1 -\dfrac{\Theta_2^2}{t_1}\Big) \|p_h\|^2_{0,\Omega}  - \dfrac{\alpha \xi_0 t_1}{4} \mathcal{L}_{3,h}(p_h,p_h). \label{wellpos-10}
	\end{align}
	\textbf{Step 4.} {Employing $(\mathbf{z}_h, q_h)=(\mathbf{u}_h - \alpha \mathbf{w}_h, p_h)$ and the estimates \eqref{wellpos-1} and \eqref{wellpos-10}, we arrive at}
	\begin{align}
		(A_h+\mathcal{L}_h)&[(\mathbf{u}_h, p_h), (\mathbf{u}_h-\alpha \mathbf{w}_h, p_h)] \nonumber \\
		&= (A_h+\mathcal{L}_h)[(\mathbf{u}_h, p_h), (\mathbf{u}_h, p_h)] + (A_h+\mathcal{L}_h)[(\mathbf{u}_h, p_h), (-\alpha \mathbf{w}_h, 0)] \nonumber \\
		& \geq  \big( C_{\lambda_{1\ast}}\mu- 4\alpha C^2_{gen}(1 + \mu + \mathcal{B} + \mathcal{B}^2 )^2 \big)\|\nabla \mathbf{u}_h \|^2_{0,\Omega} + \big( C_{\lambda_{2\ast}}\gamma - 4\alpha C^2_{gen} \gamma^2  \big)\| \mathbf{u}_h \|^2_{0,\Omega} \nonumber \\ & \qquad - \dfrac{\alpha \xi^2_0}{2} \|p_h\|^2_{0,\Omega} + \alpha \xi_0 \Big( \Theta_1 -\dfrac{\Theta_2^2}{t_1}\Big) \|p_h\|^2_{0,\Omega}   + \mathcal{L}_{1,h}(\mathbf{u}_h,\mathbf{u}_h) +\mathcal{L}_{2,h}(\mathbf{u}_h,\mathbf{u}_h) \nonumber\\ & \qquad +  {\Big(1- \dfrac{\alpha \xi_0 t_1}{4} \Big)}\mathcal{L}_{3,h}(p_h,p_h).
	\end{align}
	We now choose $t_1= \dfrac{2\Theta_2^2}{\Theta_1}$, it holds that
	\begin{align}
		(A_h+\mathcal{L}_h)&[(\mathbf{u}_h, p_h),(\mathbf{u}_h-\alpha \mathbf{w}_h, p_h)] \nonumber \\ 
		& \geq  \big( C_{\lambda_{1\ast}}\mu- 4\alpha C^2_{gen}(1 + \mu + \mathcal{B} + \mathcal{B}^2 )^2 \big)\|\nabla \mathbf{u}_h \|^2_{0,\Omega} + \big( C_{\lambda_{2\ast}}\gamma - 4\alpha C^2_{gen} \gamma^2  \big)\| \mathbf{u}_h \|^2_{0,\Omega}\, + \nonumber \\ &  \quad \big( \frac{\alpha \xi_0 \Theta_1}{2}- \frac{\alpha \xi^2_0}{2} \big) \|p_h\|^2_{0,\Omega} + \mathcal{L}_{1,h}(\mathbf{u}_h,\mathbf{u}_h) +\mathcal{L}_{2,h}(\mathbf{u}_h,\mathbf{u}_h)  + \big(1-  \frac{\alpha \xi_0 \Theta_2^2}{2 \Theta_1} \big)\mathcal{L}_{3,h}(p_h,p_h)\nonumber \\
		& \geq  \big( C_{\lambda_{1\ast}}\mu- 4\alpha C^2_{gen}(1 + \mu + \mathcal{B} + \mathcal{B}^2 )^2 \big)\|\nabla \mathbf{u}_h \|^2_{0,\Omega} + \big( C_{\lambda_{2\ast}}\gamma - 4\alpha C^2_{gen} \gamma^2  \big)\| \mathbf{u}_h \|^2_{0,\Omega}\, + \nonumber \\ &  \quad \big( \frac{\alpha \xi_0 \Theta_1}{2}- \frac{\alpha \xi_0 \Theta_1}{4} \big) \|p_h\|^2_{0,\Omega} + \mathcal{L}_{1,h}(\mathbf{u}_h,\mathbf{u}_h) +\mathcal{L}_{2,h}(\mathbf{u}_h,\mathbf{u}_h)  + \big(1-  \frac{\alpha \xi_0 \Theta_2^2}{2 \Theta_1} \big)\mathcal{L}_{3,h}(p_h,p_h),
	\end{align}
	where the last line is obtained using \eqref{wellpos-2}. Further, employing $\alpha = \min \Big\{ \dfrac{ C_{\lambda_{1\ast}} \mu}{8C^2_{gen} (1 + \mu + \mathcal{B} + \mathcal{B}^2)^2}, \dfrac{C_{\lambda_{2\ast}}}{8 C^2_{gen} \gamma}, \dfrac{ \Theta_1}{\Theta_2^2 \xi_0}  \Big\}$, we obtain  
	\begin{align}
		(A_h+\mathcal{L}_h)&[(\mathbf{u}_h, p_h), (\mathbf{u}_h-\alpha \mathbf{w}_h, p_h)] \nonumber \\& \geq  \Big( \dfrac{C_{\lambda_{1\ast}}\mu}{2} \|\nabla \mathbf{u}_h \|^2_{0,\Omega} + \dfrac{C_{\lambda_{2\ast}}\gamma}{2} \| \mathbf{u}_h \|^2_{0,\Omega} + \mathcal{L}_{1,h}(\mathbf{u}_h,\mathbf{u}_h) +\mathcal{L}_{2,h}(\mathbf{u}_h,\mathbf{u}_h)  \nonumber \\ & +  \dfrac{1}{2}\mathcal{L}_{3,h}(p_h,p_h) + \dfrac{\xi_0 \Theta_1}{4} \min \Big\{ \dfrac{ C_{\lambda_{1\ast}} \mu}{8C^2_{gen} (1 + \mu + \mathcal{B} + \mathcal{B}^2)^2}, \dfrac{C_{\lambda_{2\ast}}}{8 C^2_{gen} \gamma}, \dfrac{ \Theta_1}{\Theta_2^2 \xi_0}  \Big\} \|p_h\|^2_{0,\Omega} \Big) \nonumber \\
		& \geq  \min \Big\{\frac{1}{2}, \frac{C_{\lambda_{1\ast}}}{2}, \frac{C_{\lambda_{2\ast}}}{2}, \frac{ \xi_0 \Theta_1 C_{\lambda_{1\ast}} \mu}{32C^2_{gen} (1 + \mu + \mathcal{B} + \mathcal{B}^2)^2}, \dfrac{\xi_0 \Theta_1 C_{\lambda_{2\ast}}}{32 C^2_{gen} \gamma}, \dfrac{ \Theta_1^2}{4\Theta_2^2}  \Big\} \Big( \mu \|\nabla \mathbf{u}_h \|^2_{0,\Omega}  \nonumber \\ & \qquad  + \gamma \| \mathbf{u}_h \|^2_{0,\Omega}  + \mathcal{L}_{1,h}(\mathbf{u}_h,\mathbf{u}_h) +\mathcal{L}_{2,h}(\mathbf{u}_h,\mathbf{u}_h)  + \mathcal{L}_{3,h}(p_h,p_h) + \|p_h\|^2_{0,\Omega} \Big) \nonumber \\
		& \geq \mathbb N_1 \vertiii{(\mathbf{u}_h, p_h)}^2, \label{welllpos-9}
	\end{align}
	where $\mathbb N_1:= \min \Big\{\dfrac{1}{2}, \dfrac{C_{\lambda_{1\ast}}}{2}, \dfrac{C_{\lambda_{2\ast}}}{2}, \dfrac{ \xi_0 \Theta_1 C_{\lambda_{1\ast}} \mu}{32C^2_{gen} (1 + \mu + \mathcal{B} + \mathcal{B}^2)^2}, \dfrac{\xi_0 \Theta_1 C_{\lambda_{2\ast}}}{32 C^2_{gen} \gamma}, \dfrac{ \Theta_1^2}{4\Theta_2^2}  \Big\}$. \newline
	\textbf{Step 5.} This is our final step in order to achieve the estimate \eqref{wellpos-0}. We begin with the definition of the energy norm:
	\begin{align}
		\vertiii{(\mathbf{z}_h, q_h)}^2 &= \vertiii{(\mathbf{u}_h-\alpha \mathbf{w}_h, p_h)}^2 \leq 2\Big[ \vertiii{(\mathbf{u}_h, p_h)}^2  + \alpha^2	\vertiii{( \mathbf{w}_h, 0)}^2 \Big] \nonumber \\
		& \leq 2 \max\big\{1, \lambda_1^\ast \big\} \Big[ \mu \|\nabla \mathbf{u}_h \|^2_{0,\Omega} + \gamma \| \mathbf{u}_h \|^2_{0,\Omega}  + \mathcal{L}_{1,h}(\mathbf{u}_h,\mathbf{u}_h) +\mathcal{L}_{2,h}(\mathbf{u}_h,\mathbf{u}_h)  + \mathcal{L}_{3,h}(p_h,p_h) \nonumber \\ & \qquad  + \Big(1 + \alpha^2 \xi_0^2 \big( 1 + \mu + \mathcal{B}^2 \big) \Big) \|p_h\|^2_{0,\Omega} + \alpha^2 \gamma \| \mathbf{w}_h \|^2_{0,\Omega} \Big] \nonumber \\
		& \leq 2 \max\big\{1, \lambda_1^\ast,C^2_p \big\}\Big[ \mu \|\nabla \mathbf{u}_h \|^2_{0,\Omega} + \gamma  \| \mathbf{u}_h \|^2_{0,\Omega}  + \mathcal{L}_{1,h}(\mathbf{u}_h,\mathbf{u}_h) +\mathcal{L}_{2,h}(\mathbf{u}_h,\mathbf{u}_h)  + \mathcal{L}_{3,h}(p_h,p_h) \nonumber \\ & \qquad  + \Big(1 + \alpha^2 \xi_0^2 \big( 1 + \mu +  \gamma + \mathcal{B}^2 \big) \Big) \|p_h\|^2_{0,\Omega} \Big] \nonumber \\
		& \leq 2 \max\big\{1, \lambda_1^\ast,C^2_p \big\} \big(1 + \alpha^2 \xi_0^2 \big( 1 + \mu  + \gamma +  \mathcal{B}^2 \big) \big) \Big[ \mu \|\nabla \mathbf{u}_h \|^2_{0,\Omega} + \gamma  \| \mathbf{u}_h \|^2_{0,\Omega} \nonumber \\ & \qquad  + \mathcal{L}_{1,h}(\mathbf{u}_h,\mathbf{u}_h) +\mathcal{L}_{2,h}(\mathbf{u}_h,\mathbf{u}_h)  + \mathcal{L}_{3,h}(p_h,p_h)  +  \|p_h\|^2_{0,\Omega} \Big] \nonumber \\
		& \leq \mathbb N_2 \vertiii{(\mathbf{u}_h, p_h)}^2, \label{wellpos-9a}
	\end{align}
	where $\mathbb N_2:= 2 \max\big\{1, \lambda_1^\ast,C^2_p \big\} \big(1 + \alpha^2 \xi_0^2 \big( 1 + \mu  + \gamma +  \mathcal{B}^2 \big) \big)$. Combining \eqref{welllpos-9} and \eqref{wellpos-9a}, we easily obtain the result \eqref{wellpos-0} with $\Theta=\mathbb N_1/\sqrt{\mathbb{N}_2}$. \newline
	\textbf{Step 6.} In this step, we establish the uniqueness of the discrete solution. Since all the contributed discrete bilinear forms in $(A_h+\mathcal{L}_h)$ are continuous with respect to the energy norm $\vertiii{\cdot}$, obviously, we arrive at 
	\begin{align}
		|(A_h + \mathcal{L}_h)[(\mathbf{u}_h, p_h), (\mathbf{z}_h, q_h)]| \leq  C_\ast \vertiii{(\mathbf{u}_h,  p_h)}\, \vertiii{(\mathbf{z}_h,  q_h)} \qquad \text{for all } \,\,\, (\mathbf{z}_h,  q_h) \in \mathbf{V}^k_h \times Q^k_h.\label{wellpos-11}
	\end{align}
	Thus, combining \eqref{wellpos-0} and \eqref{wellpos-11} guarantees the well-posedness of the VE problem \eqref{nvem}.
	
\end{proof}
\begin{remark}
	The most noteworthy aspect of Theorem \ref{wellposed} is that its proof is established in a generalized form, which is robust with respect to $\boldsymbol{\mathcal{B}}$ and $\gamma$. The choice of $t_1$ and $\alpha$ ensures that the validity of the proof remains unaffected when either the reaction parameter $\gamma=0$ or the convective flow field $ \boldsymbol{\mathcal{B}}=\mathbf{0}$. Consequently, there is no need to introduce any supplementary conditions to confirm the validity of \eqref{wellpos-0}.
\end{remark}
\begin{remark}
	\texttt{The case $\boldsymbol{\mathcal{B}} = \mathbf{0}$ and $\gamma >0$ (Brinkman problem)}.
	The proof of Theorem \ref{wellposed} is valid for the parameters $\boldsymbol{\mathcal{B}} = \mathbf{0}$ and $\gamma >0$, thus extending to the stability of the Brinkman problem as a special case without imposing additional assumptions. In this case, the constants $\mathbb N_1$ and $\mathbb N_2$ are given by 
	\begin{align}
		\Scale[0.9]{	\mathbb N_1:= \min \Big\{\dfrac{1}{2}, \dfrac{C_{\lambda_{1\ast}}}{2}, \dfrac{C_{\lambda_{2\ast}}}{2}, \dfrac{ \xi_0 \Theta_1 C_{\lambda_{1\ast}} \mu}{32C^2_{gen} (1 + \mu)^2}, \dfrac{\xi_0 \Theta_1 C_{\lambda_{2\ast}}}{32 C^2_{gen} \gamma}, \dfrac{ \Theta_1^2}{4\Theta_2^2}  \Big\} \qquad \mathbb N_2:= 2 \max\big\{1, \lambda_1^\ast,C^2_p \big\} \big(1 + \alpha^2 \xi_0^2 \big( 1 + \mu  + \gamma \big) \big).} \nonumber
	\end{align}
\end{remark}

\begin{remark}
	\texttt{The case $\boldsymbol{\mathcal{B}} \neq \mathbf{0}$ and $\gamma =0$ (classic Oseen problem)}.
	Theorem \ref{wellposed} is valid with $\Theta= \mathbb N_1/\sqrt{\mathbb N_2}$, where the constants $\mathbb N_1$ and $\mathbb N_2$ are given by 
	\begin{align}
		\Scale[0.9]{\mathbb N_1:= \min \Big\{\dfrac{1}{2}, \dfrac{C_{\lambda_{1\ast}}}{2}, \dfrac{C_{\lambda_{2\ast}}}{2}, \dfrac{ \xi_0 \Theta_1 C_{\lambda_{1\ast}} \mu}{32C^2_{gen} (1 + \mu + \mathcal{B} + \mathcal{B}^2)^2},  \dfrac{ \Theta_1^2}{4\Theta_2^2}  \Big\} \qquad \mathbb N_2:= 2 \max\big\{1, \lambda_1^\ast \big\} \big(1 + \alpha^2 \xi_0^2 \big( 1 + \mu +  \mathcal{B}^2 \big) \big).} \nonumber
	\end{align}
	Thus, the well-posedness theorem demonstrates the stability of the classic Oseen problem. 
\end{remark}

\begin{remark}
	\texttt{The case $\boldsymbol{\mathcal{B}}= \mathbf{0}$ and $\gamma =0$ (Stokes problem)}.
	Theorem \ref{wellposed} holds for this case with $\Theta= \mathbb N_1/\sqrt{\mathbb N_2}$, where the constants $\mathbb N_1$ and $\mathbb N_2$ are given by
	\begin{align}
		\mathbb N_1:=\min \Big\{\dfrac{1}{2}, \dfrac{C_{\lambda_{2\ast}}}{2}, \dfrac{ \xi_0 \Theta_1 C_{\lambda_{1\ast}} \mu}{32C^2_{gen} (1 + \mu )^2}, \dfrac{ \Theta_1^2}{4\Theta_2^2}  \Big\} \qquad \mathbb N_2:= 2 \max\big\{1, \lambda_1^\ast \big\} \big(1 + \alpha^2 \xi_0^2 \big( 1 + \mu \big) \big). \nonumber
	\end{align}
	Thus, the well-posedness theorem is still applicable to the parameters $\boldsymbol{\mathcal{B}} = \mathbf{0}$ and $\gamma =0$, signifying the stability of the Stokes equation without requiring any additional assumptions, similar to \cite{vem028}.
\end{remark}

{Hereafter, we establish the stability of the discrete VE problem \eqref{nvem} in a more general way with vanishing diffusion $\mu \rightarrow 0$.
	\begin{theorem} \label{rob_stab}
		Under the mesh regularity assumption \textbf{(A1)}, there exists a positive constant $\Theta$ independent of mesh size and data of the problem such that for any $(\mathbf{u}_h, p_h) \in \mathbf{V}^k_{h} \times Q^k_h$, it holds the following  
		\begin{align}
			\sup \limits_{(\mathbf{0},0) \neq(\mathbf{z}_h, q_h) \in \mathbf{V}^k_{h} \times Q^k_h} \dfrac{(A_{h}+\mathcal{L}_h)[(\mathbf{u}_h, p_h), (\mathbf{z}_h, q_h)]}{\vertiii{(\mathbf{z}_h, q_h)}}  \geq \Theta \vertiii{(\mathbf{u}_h, p_h)}, \label{newwellpos-0}
		\end{align}
		where $\Theta \geq \frac{1}{4}$. Furthermore, the stabilized VE problem \eqref{nvem} has a unique solution.
	\end{theorem}
	\begin{proof}
		See appendix. 
\end{proof}}
\subsection{Error analysis}
In this section, we analyze the error estimates for the proposed VE problem \eqref{nvem}. 
Let $(\mathbf{u}_I, \mathbf{p}_I) \in \mathbf{V}^k_h \times Q_h^k$ denote the virtual interpolant of $(\mathbf{u}, p) \in \mathbf{V} \times Q$. We do the following settings: 
\begin{align*}
	e^\mathbf{u}_I := \mathbf{u}-\mathbf{u}_I, \qquad e^\mathbf{u}_h := \mathbf{u}_h-\mathbf{u}_I, \qquad 
	e^p_I := p-p_I, \qquad e^p_h := p_h- p_I.
\end{align*}
Further, we introduce the following parameters useful in subsequent analysis:
\begin{align}
	\Lambda_E = \min \Big\{ \dfrac{1}{\sqrt{\mu}}, \dfrac{1}{\mathcal{B}_E \sqrt{\tau_{1,E}}}, \dfrac{1}{\sqrt{\tau_{2,E}}}\Big\}, \qquad \delta_E = \min \Big\{ \dfrac{h_E}{\sqrt{\mu}}, \dfrac{1}{\sqrt{\gamma}}\Big\} \qquad \,\, \text{for all} \,\,E \in \Omega_{h}. \label{param}
\end{align}
\noindent \textbf{(A2) Regularity assumptions for error analysis:} We impose the following conditions: 
\begin{align}
	\mathbf{u} \in \mathbf{V} \cap [H^{k+1}(\Omega_h)]^2, \qquad p \in Q \cap H^{k}(\Omega_h), \qquad
	\boldsymbol{\mathcal{B}} \in [W^{k}_{\infty}(\Omega_h)]^2, \qquad \mathbf{f} \in [H^{k+1}(\Omega_h)]^2.
\end{align}
\begin{lemma}   \label{ei}
	
	Under the assumptions \textbf{(A1)} and \textbf{(A2)}, let $(\mathbf{u}_I,p_I) \in \mathbf{V}^k_h \times Q^k_h$ denote the virtual interpolant of $(\mathbf{u},p) \in \mathbf{V}\times Q$. Then, we obtain the following result:
	\begin{align}
		\vertiii{(\mathbf{u} - \mathbf{u}_I, p-p_I)}^2 \leq C \sum_{E \in \Omega_{h}} h_E^{2k} \Big( (\mu + h_E^2 \gamma + \mathcal{B}^2_E h_E + 1) \|\mathbf{u}\|_{k+1,E}^2 +   \|p\|_{k,E}^2\Big). \label{ei-0}
	\end{align}
\end{lemma}

\begin{proof} We begin with the definition of energy norm, the stability of the projectors and Lemma \ref{lemmaproj2}:
	\begin{align}
		\vertiii{(\mathbf{u} - \mathbf{u}_I, p-p_I)}^2 & = \mu \|\nabla e^\mathbf{u}_I\|^2_{0,\Omega} + \gamma \|e^\mathbf{u}_I\|^2_{0,\Omega} + {\alpha} \|e^p_I\|^2_{0,\Omega} + \mathcal{L}_{1,h}(e^\mathbf{u}_I, e^\mathbf{u}_I) + \mathcal{L}_{2,h}(e^\mathbf{u}_I, e^\mathbf{u}_I) + \mathcal{L}_{3,h}(e^p_I, e^p_I) \nonumber \\
		& \leq C \sum_{E \in \Omega_h} \Big( \mu h_E^{2k} \|\mathbf{u}\|^2_{k+1,E}  +  \gamma h^{2k+2}_E \|\mathbf{u}\|^2_{k+1,E} + h^{2k}_E {\|p\|^2_{k,E}} + \mathcal{L}^E_1 + \mathcal{L}^E_2 +\mathcal{L}^E_3\Big). \label{ei-1}
		\intertext{Employing the stability of the projectors, {Remark \ref{rem_tau}} and \eqref{vem-a}, we infer}
		\mathcal{L}^E_1 &\leq  \tau_{1,E}  \lambda^\ast _1 \mathcal{B}^2_E \|\nabla (\mathbf{I}-\boldsymbol{\Pi}^{\nabla,E}_{k-1}) e^\mathbf{u}_I\|^2_{0,E} \nonumber \\
		& \leq  \lambda_1^\ast \mathcal{B}^2_E h^{2k+1}_E \|\mathbf{u}\|^2_{k+1,E}. \label{ei-2}
		\intertext{We use the continuity of $\Pi^{0,E}_{k-1}$, bound \eqref{vem-a}, $\tau_{2,E} \sim \mathcal{O}(1)$ and Lemma \ref{lemmaproj2}:} 
		\mathcal{L}^E_2 & \leq  \Big( \tau_{2,E}\| \nabla e^\mathbf{u}_I\|^2_{0,E}  + \lambda_1^\ast \tau_{2,E} \| \nabla (\mathbf{I} - \boldsymbol{\Pi}^{\nabla,E}_k) e^\mathbf{u}_I \|^2_{0,E}  \Big) \nonumber \\
		& \leq  \max\{1, \lambda_1^\ast\} h_E^{2k} \|\mathbf{u}\|^2_{k+1,E}. \label{ei-3}
		\intertext{Following the estimation of $\mathcal{L}^E_1$ and $\mathcal{L}^E_2$, we achieve  $\mathcal{L}^E_3$ as follows}
		\mathcal{L}^E_3 & \leq  \max \{ 1, \lambda_3^\ast\} h^{2k}_E \|p\|^2_{k,E}. \label{ei-4} 
	\end{align}
	Thus, the result \eqref{ei-0} readily follows by substituting the bounds \eqref{ei-2}, \eqref{ei-3} and \eqref{ei-4} into \eqref{ei-1}. 	
\end{proof}

\begin{proposition} \label{prop}
	Let the assumptions \textbf{(A1)} and \textbf{(A2)} be valid. Further, we assume the pairs $(\mathbf{u}, p) \in \mathbf{V} \times Q$ and $(\mathbf{u}_h, p_h) \in \mathbf{V}^k_h \times Q^k_h$ are the solution of problems \eqref{discf} and \eqref{nvem}, respectively. Then, for $(\mathbf{v}_h,  q_h) \in \mathbf{V}^k_h \times Q^k_h$,  it holds that
	\begin{align}
		\vertiii{(\mathbf{u}_h - \mathbf{u}_I, p_h -p_I)}\, \vertiii{(\mathbf{v}_h,  q_h)} \leq C \Big(  \eta_F + \eta_\mathcal{A} + \eta_{\mathcal{L}} \Big),  \label{prop-0}
	\end{align}
	where 
	\begin{align}
		\eta_F&:= (\mathbf{f}_h, \mathbf{v}_h) - (\mathbf{f}, \mathbf{v}_h), \\
		\eta_A &:=  A [(\mathbf{u}, p), (\mathbf{v}_h, q_h)] -  A_h[(\mathbf{u}_I, p_I), (\mathbf{v}_h, q_h)],  \\
		\eta_{\mathcal{L}} &:=   -\mathcal{L}_h[(\mathbf{u}_I, p_I), (\mathbf{v}_h, q_h)].  
	\end{align}
\end{proposition}

\begin{proof}
	Using Theorem \ref{wellposed} or \ref{rob_stab}, we have
	\begin{align}
		\Theta	\vertiii{(\mathbf{u}- \mathbf{u}_h, p-p_h)}\, \vertiii{(\mathbf{v}_h, q_h)} 
		& \leq (A_h + \mathcal{L}_h)[(\mathbf{u}_h - \mathbf{u}_I, p_h-p_I), (\mathbf{v}_h, q_h)]  \nonumber \\
		& \leq 	(\mathbf{f}_h, \mathbf{v}_h) - (A_h + \mathcal{L}_h)[( \mathbf{u}_I, p_I), (\mathbf{v}_h, q_h)]  \nonumber \\
		& \leq  (\mathbf{f}_h, \mathbf{v}_h) - (\mathbf{f}, \mathbf{v}_h) + A[( \mathbf{u}, p), (\mathbf{v}_h, q_h)] - A_h[( \mathbf{u}_I, p_I), (\mathbf{v}_h, q_h)] \nonumber \\ & \qquad - \mathcal{L}_h[( \mathbf{u}_I, p_I), (\mathbf{v}_h, q_h)], \nonumber
	\end{align}
	which completes the proof.
\end{proof}

\begin{lemma}(Estimation of $\eta_F$) \label{eta_f}
	Under the assumptions \textbf{(A1)} and \textbf{(A2)}, the term $\eta_F$ can be bounded as follows
	\begin{align}
		\eta_F \leq C \Big(\sum_{E \in \Omega_{h}} \Lambda_E^2 h_E^{2(k+2)} \|\mathbf{f}\|^2_{k+1,E}\Big)^{1/2} {\vertiii{(\mathbf{v}_h,q_h)}}.\label{ef-0}
	\end{align}
\end{lemma}

\begin{proof} Applying the orthogonality of the projectors, the Poincar$\acute{\text{e}}$ inequality, Lemmas \ref{lemmaproj1} and \ref{estimate2}, it holds that
	\begin{align}
		\eta_F&= \sum_{E \in \Omega_h} \Big((\mathbf{f}, \boldsymbol{\Pi}^{0,E}_k \mathbf{v}_h) - (\mathbf{f}, \mathbf{v}_h) \Big) \nonumber \\
		&= \sum_{E \in \Omega_h} \Big((\mathbf{f}-\boldsymbol{\Pi}^{0,E}_k\mathbf{f}, \boldsymbol{\Pi}^{0,E}_k \mathbf{v}_h - \mathbf{v}_h  )\Big)	\nonumber \\
		&\leq C \sum_{E \in \Omega_h} \|\mathbf{f}-\boldsymbol{\Pi}^{0,E}_k\mathbf{f}\|_{0,E} \|\mathbf{v}_h - \boldsymbol{\Pi}^{0,E}_k \mathbf{v}_h\|_{0,E}	\nonumber \\
		&\leq C \sum_{E \in \Omega_h} h^{k+2}_E\|\mathbf{f}\|_{k+1,E} \|\nabla(\mathbf{I} - \boldsymbol{\Pi}^{0,E}_k )\mathbf{v}_h\|_{0,E}	\nonumber \\
		&\leq C \sum_{E \in \Omega_h} h^{k+2}_E\|\mathbf{f}\|_{k+1,E} \|\nabla(\mathbf{I} - \boldsymbol{\Pi}^{\nabla,E}_k )\mathbf{v}_h\|_{0,E}	\nonumber
		\intertext{{Recalling the definition of energy norm (or see \cite{mvem16}), we infer}}
		\eta_F	&\leq C \sum_{E \in \Omega_h} \min \Big\{ \dfrac{1}{\sqrt{\mu}}, \dfrac{1}{\mathcal{B}_E \sqrt{\tau_{1,E}}}, \dfrac{1}{\sqrt{\tau_{2,E}}} \Big\} h^{k+2}_E\|\mathbf{f}\|_{k+1,E}  \vertiii{(\mathbf{v}_h, q_h)}_E.  \nonumber 
	\end{align}
\end{proof}
\begin{lemma} \label{etaA}
	Let the assumptions \textbf{(A1)} and \textbf{(A2)} be valid. Then the term $\eta_{A}$ can be bounded as follows
\end{lemma}
\begin{align}
	\eta_A &\leq C \Big[ \Big( \sum_{E \in \Omega_h} \Big( \mu + {\dfrac{1}{\alpha}}+ \gamma h^2_E + \dfrac{h^2_E}{\tau_{3,E}} + \Lambda^2_E h^2_E\big(\|\boldsymbol{\mathcal{B}}\|_{k,\infty,E} + \mathcal{B}_E + \gamma\big)^2 + \delta^2_E \mathcal{B}^2_E \Big) h^{2k}_E \|\mathbf{u}\|^2_{k+1,E}\Big)^{1/2} +  \nonumber \\ & \qquad \qquad \Big( \sum_{E \in \Omega_h} \tau_{2,E}^{-1} h^{2k}_E \|p\|_{k,E}^2  \Big)^{1/2} \Big] {\vertiii{(\mathbf{v}_h, q_h)}}. \label{eta0}
\end{align}

\begin{proof}
	Concerning $\eta_{A}$, we proceed as follows
	\begin{align}
		\eta_A &= A[(\mathbf{u}, p), (\mathbf{v}_h, q_h)] - A_h [(\mathbf{u}_I,p_I), (\mathbf{v}_h, q_h)] \nonumber \\
		& =\Big( a(\mathbf{u}, \mathbf{v}_h)- a_h(\mathbf{u}_I, \mathbf{v}_h)\Big) - \Big(b(\mathbf{v}_h, p) - b_h( \mathbf{v}_h, p_I) \Big) + \Big( b(\mathbf{u}, q_h)- b_h( \mathbf{u}_I, q_h)\Big) \nonumber \\
		& \qquad + \Big(c^{skew}(\mathbf{u}, \mathbf{v}_h)- c^{skew}_h(\mathbf{u}_I, \mathbf{v}_h)\Big) + \Big(d(\mathbf{u}, \mathbf{v}_h)- d_h(\mathbf{u}_I, \mathbf{v}_h)\Big) \nonumber \\
		& =:  |\eta_{A,a} - \eta_{A,b1} +\eta_{A,b2} + \eta_{A,c} + \eta_{A,d}|.  \label{etaA-1} 
	\end{align}
	The estimation of the above terms can be given in the following steps: \newline
	\textbf{Step 1.} Estimation of $\eta_{A,a}$: We recall the definition of the projectors, Remark \ref{estimate} and Lemmas \ref{lemmaproj1} and \ref{lemmaproj2}: 
	\begin{align}
		\eta_{A,a} &= \sum_{E \in \Omega_h} \Big( \mu\Big( \nabla \mathbf{u}, \nabla \mathbf{v}_h \Big) - \mu \Big( \boldsymbol{\Pi}^{0,E}_{k-1} \nabla \mathbf{u}_I, \boldsymbol{\Pi}^{0,E}_{k-1} \nabla \mathbf{v}_h  \Big) - \mu S^E_\nabla \Big( (\mathbf{I} - \boldsymbol{\Pi}^{\nabla,E}_k) \mathbf{u}_I, (\mathbf{I} - \boldsymbol{\Pi}^{\nabla,E}_k) \mathbf{v}_h\Big) \Big) \nonumber \\
		&= \sum_{E \in \Omega_h} \Big( \mu\Big( \nabla \mathbf{u} -\nabla \mathbf{u}_I, \nabla \mathbf{v}_h \Big) + \mu \Big( \nabla \mathbf{u}_I - \boldsymbol{\Pi}^{0,E}_{k-1} \nabla \mathbf{u}_I, \nabla \mathbf{v}_h  - \boldsymbol{\Pi}^{0,E}_{k-1} \nabla \mathbf{v}_h  \Big) - \nonumber \\& \qquad \mu S^E_\nabla \Big( (\mathbf{I} - \boldsymbol{\Pi}^{\nabla,E}_k) \mathbf{u}_I, (\mathbf{I} - \boldsymbol{\Pi}^{\nabla,E}_k) \mathbf{v}_h\Big) \Big) \nonumber \\
		&\leq C \sum_{E \in \Omega_h} \Big( \mu \| \nabla \mathbf{u} -\nabla \mathbf{u}_I\|_{0,E} \|\nabla \mathbf{v}_h \|_{0,E} + \mu \| \nabla \mathbf{u}_I - \boldsymbol{\Pi}^{0,E}_{k-1} \nabla \mathbf{u}_I\|_{0,E} \|\nabla \mathbf{v}_h  - \boldsymbol{\Pi}^{0,E}_{k-1} \nabla \mathbf{v}_h \|_{0,E}  \nonumber \\ & \qquad + \lambda_1^\ast \mu \|\nabla (\mathbf{I} - \boldsymbol{\Pi}^{\nabla,E}_k) \mathbf{u}_I\|_{0,E} \|\nabla (\mathbf{I} - \boldsymbol{\Pi}^{\nabla,E}_k) \mathbf{v}_h\|_{0,E} \Big) \nonumber \\
		&\leq C \sum_{E \in \Omega_h} \Big( \mu \| \nabla \mathbf{u} -\nabla \mathbf{u}_I\|_{0,E} \|\nabla \mathbf{v}_h \|_{0,E} + (1+\lambda_1^\ast) \mu \|\nabla (\mathbf{I} - \boldsymbol{\Pi}^{\nabla,E}_k) \mathbf{u}_I\|_{0,E} \|\nabla (\mathbf{I} - \boldsymbol{\Pi}^{\nabla,E}_k) \mathbf{v}_h\|_{0,E} \Big) \nonumber \\
		&\leq C \sum_{E \in \Omega_h} \Big( \sqrt{\mu} h_E^k \| \mathbf{u} \|_{k+1,E}  + (1+\lambda_1^\ast)\sqrt{\mu} h_E^k \| \mathbf{u} \|_{k+1,E}  \Big) \vertiii{(\mathbf{v}_h, q_h)}_E \nonumber \\ 
		&\leq C \sum_{E \in \Omega_h}  \sqrt{\mu} h_E^k \| \mathbf{u} \|_{k+1,E}  \vertiii{(\mathbf{v}_h, q_h)}_E. \label{etaA-2}
	\end{align}
	\textbf{Step 2.} Estimation of $\eta_{A,b1}$: Using the property of the projectors and Lemmas \ref{lemmaproj1} and \ref{lemmaproj2}, we obtain
	\begin{align}
		\eta_{A,b1} &= \sum_{E \in \Omega_h} \Big( \Big( \nabla \cdot \mathbf{v}_h, p\Big) - \Big( {\Pi}^{0,E}_{k-1} \nabla \cdot \mathbf{v}_h, \Pi^{0,E}_k p_I \Big)  \Big) \nonumber \\
		&= \sum_{E \in \Omega_h} \Big( \Big( \nabla \cdot \mathbf{v}_h - {\Pi}^{0,E}_{k-1} \nabla \cdot \mathbf{v}_h, p \Big) + \Big(  \Pi^{0,E}_{k-1} \nabla \cdot \mathbf{v}_h,  p - \Pi^{0,E}_{k-1} p_I \Big)  \Big) \nonumber \\  	
		&= \sum_{E \in \Omega_h} \Big( \Big( ({I} - {\Pi}^{0,E}_{k-1})\nabla \cdot \mathbf{v}_h, p - \Pi^{0,E}_{k-1} p\Big) + \Big(  \Pi^{0,E}_{k-1} \nabla \cdot \mathbf{v}_h,  p - \Pi^{0,E}_{k-1} p_I \Big) \Big) \nonumber \\  
		&\leq C \sum_{E \in \Omega_h} \Big( \|({I} - {\Pi}^{0,E}_{k-1})\nabla \cdot \mathbf{v}_h\|_{0,E} \|p - \Pi^{0,E}_{k-1} p\|_{0,E} + \| \Pi^{0,E}_{k-1} \nabla \cdot \mathbf{v}_h\|_{0,E} \|p-\Pi^{0,E}_{k-1} p_I\|_{0,E} \Big) \nonumber \\  
		&\leq C \sum_{E \in \Omega_h} \Big( h^k_E \|p\|_{k,E} \|\nabla(\boldsymbol{I} -\boldsymbol{\Pi}^{\nabla,E}_{k}) \mathbf{v}_h\|_{0,E} + h^k_E \|\Pi^{0,E}_{k-1} \nabla \cdot \mathbf{v}_h\|_{0,E} \|p\|_{k,E} \Big) \qquad \text{{\big(using \eqref{est2}\big)}}\nonumber \\
		&\leq C \sum_{E \in \Omega_h} \dfrac{h^{k}_E}{\sqrt{\tau_{2,E}}}   \|p\|_{k,E} \vertiii{ (\mathbf{v}_h, q_h)}_E. \label{etaA-b1}
	\end{align}
	\textbf{Step 3.} Estimation of $\eta_{A,b2}$: We use the orthogonal property of the projection operator, Lemma \ref{lemmaproj1} and Lemma \ref{lemmaproj2}:  
	\begin{align}
		\eta_{A,b2} &= \sum_{E \in \Omega_h} \Big( \Big( \nabla \cdot \mathbf{u}, q_h\Big) - \Big(  {\Pi}^{0,E}_{k-1} \nabla \cdot \mathbf{u}_I, \Pi^{0,E}_k q_h  \Big) \Big) \nonumber \\
		&= \sum_{E \in \Omega_h} \Big( \Big( \nabla \cdot \mathbf{u}, q_h\Big) - \Big( \nabla \cdot \mathbf{u}_I, {\Pi^{0,E}_{k-1} q_h}  \Big) \Big) \nonumber \\
		&= \sum_{E \in \Omega_h} \Big( \Big( \nabla \cdot (\mathbf{u}-\mathbf{u}_I), q_h\Big) + \Big(\nabla \cdot \mathbf{u}_I -  {\Pi}^{0,E}_{k-1} \nabla \cdot \mathbf{u}_I, q_h - \Pi^{0,E}_{k-1} q_h  \Big) \Big) \nonumber \\
		& \leq \sum_{E \in \Omega_h} \Big(  \|\nabla\cdot( \mathbf{u}-\mathbf{u}_I) \|_{0,E} \| q_h\|_{0,E} +  \| (I - {\Pi}^{0,E}_{k-1}) \nabla \cdot \mathbf{u}_I \|_{0,E} \|(I- \Pi^{0,E}_{k-1} )q_h \|_{0,E}  \Big) \nonumber 
		\intertext{ {Recalling the estimate \eqref{est2}, Lemmas \ref{lemmaproj1} and \ref{lemmaproj2}, and  the property $\big(I- \Pi^{0,E}_{t}\big)q_h= \big(I- \Pi^{0,E}_{t}\big) \big(I- \Pi^{\nabla,E}_{t} \big)q_h$ for any $t\in \mathbb N \cup \{0\}$, it holds  }}
		\eta_{A,b2} & \leq C \sum_{E \in \Omega_h} \Big(  \| \nabla(\mathbf{u}-\mathbf{u}_I) \|_{0,E} \| q_h\|_{0,E} + \|\nabla (\mathbf{I} - \boldsymbol{\Pi}^{\nabla,E}_{k}) \mathbf{u}_I \|_{0,E} \|(I- \Pi^{0,E}_{k-1})(I- \Pi^{\nabla,E}_{k-1} )q_h \|_{0,E}  \Big) \nonumber \\
		& \leq C \sum_{E \in \Omega_h} \Big(  h^k_E \|\mathbf{u} \|_{k+1,E} \| q_h\|_{0,E} + h^{k+1}_E\| \mathbf{u} \|_{k+1,E} \|\nabla(I- \Pi^{\nabla,E}_{k-1} )q_h \|_{0,E}  \Big)  \qquad \text{\big(using Lemma \ref{lemmaproj1}\big)}\nonumber \\
		& \leq C \sum_{E \in \Omega_h} \Big( {\dfrac{1}{\sqrt{\alpha}}} +  \dfrac{h_E}{\sqrt{\tau_{3,E}}} \Big) h^{k}_E \|\mathbf{u}\|_{k+1,E} \vertiii{(\mathbf{v}_h, q_h)}_E. \label{etaA-b2}
	\end{align}
	\textbf{Step 4.} Estimation of $\eta_{A,c}$:  We proceed as follows
	\begin{align}
		\eta_{A,c} &= c^{skew}(\mathbf{u}, \mathbf{v}_h) - c^{skew}_h(\mathbf{u}_I, \mathbf{v}_h) \nonumber \\
		&= \dfrac{1}{2}\sum_{E \in \Omega_h } \Big[   c^E(\mathbf{u}, \mathbf{v}_h) - c^{E}_h(\mathbf{u}_I, \mathbf{v}_h)+c^{E}_h(\mathbf{v}_h, \mathbf{u}_I) -c^E(\mathbf{v}_h, \mathbf{u}) \Big] \nonumber \\
		&= \dfrac{1}{2} \sum_{E \in \Omega_h} \Big[ \Big( (\nabla \mathbf{u}) \boldsymbol{\mathcal{B}}, \mathbf{v}_h \Big) - \Big( (\nabla  \boldsymbol{\Pi}^{0,E}_k \mathbf{u}_I)\boldsymbol{\mathcal{B}},  \boldsymbol{\Pi}^{0,E}_k \mathbf{v}_h\Big) - \int_{\partial E} \boldsymbol{\mathcal{B}} \cdot \mathbf{n}^E (\mathbf{I}- \boldsymbol{\Pi}^{0,E}_k) \mathbf{u}_I\cdot  \boldsymbol{\Pi}^{0,E}_k  \mathbf{v}_h ds \nonumber \\ & \qquad + \Big( (\nabla  \boldsymbol{\Pi}^{0,E}_k \mathbf{v}_h)\boldsymbol{\mathcal{B}},  \boldsymbol{\Pi}^{0,E}_k \mathbf{u}_I\Big) -\Big((\nabla \mathbf{v}_h) \boldsymbol{\mathcal{B}}, \mathbf{u} \Big) + \int_{\partial E} \boldsymbol{\mathcal{B}} \cdot \mathbf{n}^E (\mathbf{I}- \boldsymbol{\Pi}^{0,E}_k) \mathbf{v}_h\cdot  \boldsymbol{\Pi}^{0,E}_k  \mathbf{u}_I ds \Big]   \nonumber \\ 
		&= \dfrac{1}{2}  \sum_{E \in \Omega_h} \Big[ \Big( (\nabla \mathbf{u}) \boldsymbol{\mathcal{B}}, \mathbf{v}_h - \boldsymbol{\Pi}^{0,E}_k \mathbf{v}_h\Big) + \Big( \big(\nabla (\mathbf{u} -\boldsymbol{\Pi}^{0,E}_k \mathbf{u}_I) \big)\boldsymbol{\mathcal{B}},  \boldsymbol{\Pi}^{0,E}_k \mathbf{v}_h\Big) \nonumber \\ & \qquad - \int_{\partial E} \boldsymbol{\mathcal{B}} \cdot \mathbf{n}^E (\mathbf{I}- \boldsymbol{\Pi}^{0,E}_k) \mathbf{u}_I\cdot  \boldsymbol{\Pi}^{0,E}_k  \mathbf{v}_h ds  + \Big( (\nabla  \boldsymbol{\Pi}^{0,E}_k \mathbf{v}_h)\boldsymbol{\mathcal{B}},  \boldsymbol{\Pi}^{0,E}_k \mathbf{u}_I - \mathbf{u} \Big)\, + \nonumber \\ & \qquad \quad \Big( \big(\nabla (\boldsymbol{\Pi}^{0,E}_k \mathbf{v}_h - \mathbf{v}_h)\big) \boldsymbol{\mathcal{B}}, \mathbf{u} \Big) + \int_{\partial E} \boldsymbol{\mathcal{B}} \cdot \mathbf{n}^E (\mathbf{I}- \boldsymbol{\Pi}^{0,E}_k) \mathbf{v}_h\cdot  \boldsymbol{\Pi}^{0,E}_k  \mathbf{u}_I ds \Big]. \label{ac1}
	\end{align}
	Employing $\nabla \cdot \boldsymbol{\mathcal{B}}=0$ and integration by parts for the second and fifth terms, we infer
	\begin{align}
		\Big( \big(\nabla (\mathbf{u} -\boldsymbol{\Pi}^{0,E}_k \mathbf{u}_I)\big)\boldsymbol{\mathcal{B}},  \boldsymbol{\Pi}^{0,E}_k \mathbf{v}_h\Big) & = -\int_E (\mathbf{u} -\boldsymbol{\Pi}^{0,E}_k \mathbf{u}_I) \cdot ( \nabla \boldsymbol{\Pi}^{0,E}_k \mathbf{v}_h) \boldsymbol{\mathcal{B}} \, dE \nonumber \\ &\qquad +  \int_{\partial E} \boldsymbol{\mathcal{B}} \cdot \mathbf{n}^E (\mathbf{u} -\boldsymbol{\Pi}^{0,E}_k \mathbf{u}_I) \cdot \boldsymbol{\Pi}^{0,E}_k \mathbf{v}_h ds,\label{ac2} \\
		\Big( \big(\nabla (\boldsymbol{\Pi}^{0,E}_k \mathbf{v}_h - \mathbf{v}_h)\big) \boldsymbol{\mathcal{B}}, \mathbf{u} \Big) &= -\int_E  (\boldsymbol{\Pi}^{0,E}_k \mathbf{v}_h - \mathbf{v}_h) \cdot (\nabla \mathbf{u}) \boldsymbol{\mathcal{B}} \, dE - \int_{\partial E} \boldsymbol{\mathcal{B}} \cdot \mathbf{n}^E (\mathbf{I}- \boldsymbol{\Pi}^{0,E}_k) \mathbf{v}_h\cdot \mathbf{u}\, ds. \label{ac3}
	\end{align}
	Combining \eqref{ac1}, \eqref{ac2} and \eqref{ac3}, we arrive at
	\begin{align}
		\eta_{A,c} &= \dfrac{1}{2} \sum_{E \in \Omega_h} \Big[ 2 \Big( (\nabla \mathbf{u}) \boldsymbol{\mathcal{B}}, \mathbf{v}_h - \boldsymbol{\Pi}^{0,E}_k \mathbf{v}_h\Big) - 2\Big( (\mathbf{u} -\boldsymbol{\Pi}^{0,E}_k \mathbf{u}_I),  (\nabla\boldsymbol{\Pi}^{0,E}_k \mathbf{v}_h) \boldsymbol{\mathcal{B}} \Big) \nonumber \\ & \qquad + \int_{\partial E} \boldsymbol{\mathcal{B}} \cdot \mathbf{n}^E  e^\mathbf{u}_I \cdot  \boldsymbol{\Pi}^{0,E}_k  \mathbf{v}_h ds  + \int_{\partial E} \boldsymbol{\mathcal{B}} \cdot \mathbf{n}^E (\mathbf{I}- \boldsymbol{\Pi}^{0,E}_k) \mathbf{v}_h\cdot  (\boldsymbol{\Pi}^{0,E}_k  \mathbf{u}_I- \mathbf{u}) ds \Big] \nonumber \\
		&= \sum_{E \in \Omega_h} \Big[ \Big( (\nabla \mathbf{u}) \boldsymbol{\mathcal{B}}-\boldsymbol{\Pi}^{0,E}_{k-1}((\nabla \mathbf{u}) \boldsymbol{\mathcal{B}}), \mathbf{v}_h - \boldsymbol{\Pi}^{0,E}_k \mathbf{v}_h\Big) + \Big(\boldsymbol{\Pi}^{0,E}_k \mathbf{u}_I - \mathbf{u},  (\nabla\boldsymbol{\Pi}^{0,E}_k \mathbf{v}_h) \boldsymbol{\mathcal{B}} \Big) \nonumber \\ & \qquad +  \dfrac{1}{2} \Big( \int_{\partial E} \boldsymbol{\mathcal{B}} \cdot \mathbf{n}^E  e^\mathbf{u}_I \cdot  \boldsymbol{\Pi}^{0,E}_k  \mathbf{v}_h ds  + \int_{\partial E} \boldsymbol{\mathcal{B}} \cdot \mathbf{n}^E (\mathbf{I}- \boldsymbol{\Pi}^{0,E}_k) \mathbf{v}_h\cdot  (\boldsymbol{\Pi}^{0,E}_k  \mathbf{u}_I- \mathbf{u}) ds \Big) \Big] \nonumber \\
		& =: \eta_{A,c,1} + \eta_{A,c,2} + \eta_{A,c,3}. \label{etaA_c0}
	\end{align}
	$\bullet$ Recalling the Cauchy-Schwarz and Poincar$\acute{\text{e}}$ inequalities, and Lemmas \ref{lemmaproj1} and \ref{estimate2}, we get
	\begin{align}
		\eta_{A,c,1} &\leq C \sum_{E \in \Omega_h} \| (\mathbf{I}-\boldsymbol{\Pi}^{0,E}_{k-1})(\nabla \mathbf{u}) \boldsymbol{\mathcal{B}}\|_{0,E} \|(\mathbf{I}- \boldsymbol{\Pi}^{0,E}_k) \mathbf{v}_h\|_{0,E} \nonumber \\
		&\leq C \sum_{E \in \Omega_h} h^{k+1}_E\| (\nabla \mathbf{u})\boldsymbol{\mathcal{B}}\|_{k,E} \|\nabla(\mathbf{I}- \boldsymbol{\Pi}^{0,E}_k) \mathbf{v}_h\|_{0,E} \nonumber \\
		&\leq C \sum_{E \in \Omega_h} h^{k+1}_E\| (\nabla \mathbf{u})\boldsymbol{\mathcal{B}}\|_{k,E} \|\nabla(\mathbf{I}- \boldsymbol{\Pi}^{\nabla,E}_k) \mathbf{v}_h\|_{0,E} \nonumber \\
		&\leq C \sum_{E \in \Omega_h}  h^{k+1}_E \min \Big\{  \dfrac{1}{\sqrt{\mu}}, \dfrac{1}{\mathcal{B}_E \sqrt{\tau_{1,E}}}, \dfrac{1}{\sqrt{\tau_{2,E}}} \Big\} \|\boldsymbol{\mathcal{B}}\|_{k,\infty,E}  \|\mathbf{u}\|_{k+1,E} \vertiii{ (\mathbf{v}_h, q_h)}_E . \label{etaA-c1}
	\end{align}
	$\bullet$ Using Lemmas \ref{lemmaproj1} and \ref{lemmaproj2}, it holds that
	\begin{align}
		\eta_{A,c,2} &\leq C \sum_{E \in \Omega_h}  h^{k+1}_E \mathcal{B}_E \|\mathbf{u}\|_{k+1,E} \|\nabla\boldsymbol{\Pi}^{0,E}_k \mathbf{v}_h \|_{0,E}. \nonumber
		\intertext{We use  Lemma \ref{inverse}, the stability of the projection operator and Lemma \ref{estimate2}:}
		\eta_{A,c,2}&\leq C \sum_{E \in \Omega_h} \min \Big\{ \dfrac{ h_E}{\sqrt{\mu}}, \dfrac{1}{\sqrt{\gamma}} \Big \} \mathcal{B}_E h^{k}_E \|\mathbf{u}\|_{k+1,E} \vertiii{ (\mathbf{v}_h, q_h)}_E. \label{etaA-c2}
	\end{align} 
	$\bullet$ Concerning $	\eta_{A,c,3}$, we proceed as follows: 
	\begin{align}
		\eta_{A,c,3} &= \sum_{E \in \Omega_h} \dfrac{1}{2} \Big( \int_{\partial E} \boldsymbol{\mathcal{B}} \cdot \mathbf{n}^E  (\boldsymbol{\Pi}^{0,E}_k  \mathbf{u}_I- \mathbf{u}- e^\mathbf{u}_I ) \cdot  (\mathbf{I}- \boldsymbol{\Pi}^{0,E}_k)  \mathbf{v}_h ds  + \int_{\partial E} \boldsymbol{\mathcal{B}} \cdot \mathbf{n}^E  \mathbf{v}_h\cdot e^\mathbf{u}_I ds \Big). \nonumber 
	\end{align}
	Recalling $\boldsymbol{\mathcal{B}} \in [W^{1}_{\infty}(\Omega_h)]^2$ and $\mathbf{v}_h, e^\mathbf{u}_I \in \mathbf{V}$, we obtain
	\begin{align*}
		\sum_{E \in \Omega_h } \int_{\partial E} \boldsymbol{\mathcal{B}} \cdot \mathbf{n}^E  \mathbf{v}_h\cdot e^\mathbf{u}_I ds = 0.
	\end{align*}
	{Therefore, using Lemma \ref{trace}, the estimation of bound \eqref{bd1} and Lemma \ref{lemmaproj1}, yields}
	\begin{align}
		\eta_{A,c,3}& \leq C \sum_{E \in \Omega_h} \mathcal{B}_E \|(\mathbf{I} - \boldsymbol{\Pi}^{0,E}_k) \mathbf{v}_h\|_{0,\partial E}  \|\boldsymbol{\Pi}^{0,E}_k  \mathbf{u}_I- \boldsymbol{\Pi}^{0,E}_k  \mathbf{u} + \boldsymbol{\Pi}^{0,E}_k  \mathbf{u} - \mathbf{u}- e^\mathbf{u}_I\|_{0, \partial E}  \nonumber \\
		& \leq C \sum_{E \in \Omega_h} \mathcal{B}_E \|(\mathbf{I} - \boldsymbol{\Pi}^{0,E}_k) \mathbf{v}_h\|_{0,\partial E} \Big(  \| \boldsymbol{\Pi}^{0,E}_k e^\mathbf{u}_I\|_{0, \partial E} + \|(\mathbf{I} - \boldsymbol{\Pi}^{0,E}_k)\mathbf{u}\|_{0,\partial E} + \| e^\mathbf{u}_I\|_{0, \partial E}\Big) \nonumber \\
		& \leq C \sum_{E \in \Omega_h} \mathcal{B}_E h^{-1/2}_E \Big(h^{-1/2}_E  \|\boldsymbol{\Pi}^{0,E}_k e^\mathbf{u}_I\|_{0, E} + h^{1/2}_E \|\nabla\boldsymbol{\Pi}^{0,E}_k e^\mathbf{u}_I\|_{0, E} + h^{k+1/2}_E \|\mathbf{u}\|_{k+1, E}\, + \nonumber \\ & \qquad \qquad  h^{-1/2}_E  \|e^\mathbf{u}_I\|_{0, E} + h^{1/2}_E \|\nabla e^\mathbf{u}_I\|_{0, E}\Big) \|(\mathbf{I} - \boldsymbol{\Pi}^{0,E}_k) \mathbf{v}_h\|_{0, E}\nonumber \\
		\intertext{ {Recalling the Poincar$\acute{\text{e}}$ inequality, the stability of the projectors, Lemma \ref{lemmaproj2} and \eqref{est1}, we obtain}}
		\eta_{A,c,3} & \leq C \sum_{E \in \Omega_h} \mathcal{B}_E h^{-1/2}_E \Big( h^{-1/2}_E  \|e^\mathbf{u}_I\|_{0, E} + h^{1/2}_E \|\nabla e^\mathbf{u}_I\|_{0, E} + h^{k+1/2}_E \|\mathbf{u}\|_{k+1, E} \Big) h_E \|\nabla(\mathbf{I} - \boldsymbol{\Pi}^{0,E}_k) \mathbf{v}_h\|_{0, E}\nonumber \\
		& \leq C \sum_{E \in \Omega_h} \mathcal{B}_E h^{k+1}_E \| \mathbf{u} \|_{k+1,E} \|\nabla (\mathbf{I} - \boldsymbol{\Pi}^{\nabla,E}_k) \mathbf{v}_h\|_{0, E} \qquad \text{ {\big(using \eqref{est3}\big)}} \nonumber \\
		&\leq C \sum_{E \in \Omega_h}  h^{k+1}_E \min \Big\{ \dfrac{1}{\sqrt{\mu}}, \dfrac{1}{\mathcal{B}_E \sqrt{\tau_{1,E}}}, \dfrac{1}{\sqrt{\tau_{2,E}}} \Big\} \mathcal{B}_E \|\mathbf{u}\|_{k+1,E} \vertiii{ (\mathbf{v}_h, q_h)}_E. \label{etaA-c3}
	\end{align}
	\textbf{Step 5.} Estimation of $\eta_{A,d}$: Employing \eqref{vem-b} and Lemmas \ref{lemmaproj1} and \ref{lemmaproj2}, it holds that  
	\begin{align}
		\eta_{A,d} &= \sum_{E \in \Omega_h} \Big( \Big( \gamma \mathbf{u}, \mathbf{v}_h \Big) - \Big( \gamma \boldsymbol{\Pi}^{0,E}_k \mathbf{u}_I, \boldsymbol{\Pi}^{0,E}_k \mathbf{v}_h \Big) - \gamma S^E_0 \Big( (\mathbf{I} -  \boldsymbol{\Pi}^{0,E}_k) \mathbf{u}_I, (\mathbf{I} -  \boldsymbol{\Pi}^{0,E}_k) \mathbf{v}_h\Big) \Big) \nonumber \\
		&= \sum_{E \in \Omega_h} \Big(  \gamma \Big( \mathbf{u} -  \boldsymbol{\Pi}^{0,E}_k\mathbf{u}, \mathbf{v}_h -  \boldsymbol{\Pi}^{0,E}_k \mathbf{v}_h \Big) + \gamma \Big(\mathbf{u}- \boldsymbol{\Pi}^{0,E}_k \mathbf{u}_I, \boldsymbol{\Pi}^{0,E}_k \mathbf{v}_h \Big) \nonumber \\ & \qquad - \gamma S^E_0 \Big( (\mathbf{I} -  \boldsymbol{\Pi}^{0,E}_k) \mathbf{u}_I, (\mathbf{I} -  \boldsymbol{\Pi}^{0,E}_k) \mathbf{v}_h\Big) \Big) \nonumber \\  
		&\leq C \sum_{E \in \Omega_h} \Big(  \gamma \| \mathbf{u} -  \boldsymbol{\Pi}^{0,E}_k\mathbf{u}\|_{0,E} \|(\boldsymbol{I}-  \boldsymbol{\Pi}^{0,E}_k) \mathbf{v}_h\|_{0,E} + \gamma \|\mathbf{u}- \boldsymbol{\Pi}^{0,E}_k \mathbf{u}_I\|_{0,E} \|\boldsymbol{\Pi}^{0,E}_k \mathbf{v}_h \|_{0,E} \nonumber \\ & \qquad + \lambda_2^\ast \gamma \| (\mathbf{I} -  \boldsymbol{\Pi}^{0,E}_k) \mathbf{u}_I\|_{0,E}  \| (\mathbf{I} -  \boldsymbol{\Pi}^{0,E}_k) \mathbf{v}_h\|_{0,E} \Big) \nonumber 
		\intertext{{We use the Poincar$\acute{\text{e}}$ inequality $\| (\mathbf{I} -  \boldsymbol{\Pi}^{0,E}_k) \mathbf{v}_h\|_{0,E} \leq h_E \| \nabla (\mathbf{I} -  \boldsymbol{\Pi}^{0,E}_k) \mathbf{v}_h\|_{0,E} $ (vector form) and the bound \eqref{est3}:} }
		& \leq C\sum_{E \in \Omega_h} \Big(  \gamma h^{k+2} \| \mathbf{u}\|_{k+1,E} \|\nabla(\boldsymbol{I}-  \boldsymbol{\Pi}^{\nabla,E}_{k}) \mathbf{v}_h\|_{0,E} + \gamma h^{k+1}_E \|\mathbf{u}\|_{k+1,E} \| \mathbf{v}_h \|_{0,E} \nonumber \\ & \qquad +  \lambda_2^\ast\gamma h^2_E\| \nabla(\mathbf{I} -  \boldsymbol{\Pi}^{\nabla,E}_{k}) \mathbf{u}_I\|_{0,E}  \| \nabla(\mathbf{I} -  \boldsymbol{\Pi}^{\nabla,E}_{k}) \mathbf{v}_h\|_{0,E} \Big) \nonumber \\  
		&\leq C \sum_{E \in \Omega_h}  h^{k+1}_E \Big( \min \Big\{   \dfrac{\gamma}{\sqrt{\mu}}, \dfrac{\gamma}{\mathcal{B} \sqrt{\tau_{1,E}}}, \dfrac{\gamma}{\sqrt{\tau_{2,E}}} \Big\}  + \sqrt{\gamma} \Big) \|\mathbf{u}\|_{k+1,E} \vertiii{ (\mathbf{v}_h, q_h)}_E.   \label{etaA-d} 
	\end{align}
	Finally, the estimate \eqref{eta0} readily follows from combining the estimates \eqref{etaA-2}, \eqref{etaA-b1}, \eqref{etaA-b2}, \eqref{etaA-c1}, \eqref{etaA-c2}, \eqref{etaA-c3} and \eqref{etaA-d}. 
\end{proof}
\begin{lemma} \label{etaL}
	Under the assumptions \textbf{(A1)} and \textbf{(A2)}, the term $\mathcal{L}_h[\cdot, \cdot]$ can be estimated as given below
	\begin{align}
		\eta_{\mathcal{L}} &\leq C \Big[\Big( \sum_{E \in \Omega_h} \Big(\tau_{1,E} \mathcal{B}_E^2 + \tau_{2,E} \Big) h^{2k}_E \|\mathbf{u}\|_{k+1,E}^2 \Big)^{1/2}   +  \Big( \sum_{E \in \Omega_h}\tau_{3,E} h^{2k-2}_E \|p\|^2_{k,E} \Big)^{1/2} \Big] {\vertiii{ (\mathbf{v}_h, q_h)}}.      \label{etaL-0}
	\end{align}
\end{lemma}
\begin{proof} Concerning \eqref{etaL-0}, we proceed as follows
	\begin{align}
		\eta_{\mathcal{L}} &= \mathcal{L}_h[(\mathbf{u}_I, p_I), (\mathbf{v}_h, q_h)]=  \mathcal{L}_{1,h}(\mathbf{u}_I, \mathbf{v}_h) + \mathcal{L}_{2,h}(\mathbf{u}_I, \mathbf{v}_h) +\mathcal{L}_{3,h}(p_I, q_h) \nonumber \\
		&=: \eta_{\mathcal{L}_1} + \eta_{\mathcal{L}_2} + \eta_{\mathcal{L}_3}. \label{eta-L}
	\end{align} 
	\textbf{Step 1.} Estimation of $\eta_{\mathcal{L}_1}$: Employing the estimate \eqref{vem-a}, the triangle inequality and Lemmas \ref{lemmaproj1} and \ref{lemmaproj2}, we infer 
	\begin{align}
		\eta_{\mathcal{L}_1} &= \sum_{E \in \Omega_h} \tau_{1,E}  \mathcal{B}^2_E S^E_\nabla \Big((\mathbf{I}-  \boldsymbol{\Pi}^{\nabla,E}_{k}) \mathbf{u}_I, (\mathbf{I}-  \boldsymbol{\Pi}^{\nabla,E}_{k}) \mathbf{v}_h\Big) \nonumber \\
		&\leq  C \sum_{E \in \Omega_h} \tau_{1,E}^{1/2} \lambda_1^\ast \mathcal{B}_E \|\nabla(\mathbf{I}-  \boldsymbol{\Pi}^{\nabla,E}_{k}) \mathbf{u}_I\|_{0,E} \vertiii{(\mathbf{v}_h, q_h)}_E \nonumber \\
		&\leq C \sum_{E \in \Omega_h} \tau_{1,E}^{1/2} h^k_E \mathcal{B}_E  \| \mathbf{u} \|_{k+1,E}  \vertiii{(\mathbf{v}_h, q_h)}_E. \label{etaL-1}
	\end{align}
	\textbf{Step 2.} Estimation of $\eta_{\mathcal{L}_2}$: Concerning $\eta_{\mathcal{L}_2}$, we use the property $\nabla \cdot \mathbf{u}=0$ and Lemmas \ref{lemmaproj1} and \ref{lemmaproj2}:
	\begin{align}
		\eta_{\mathcal{L}_2} &= \sum_{E \in \Omega_h} \tau_{2,E} \Big( \Pi^{0,E}_{k-1}\nabla \cdot\mathbf{u}_I,  \Pi^{0,E}_{k-1} \nabla \cdot \mathbf{v}_h \Big) + S^E_\nabla \Big((\mathbf{I}-  \boldsymbol{\Pi}^{\nabla,E}_{k}) \mathbf{u}_I, (\mathbf{I}-  \boldsymbol{\Pi}^{\nabla,E}_{k}) \mathbf{v}_h\Big) \Big)\nonumber \\
		&= \sum_{E \in \Omega_h} \tau_{2,E} \Big( \Pi^{0,E}_{k-1} (\nabla \cdot\mathbf{u}_I-\nabla \cdot\mathbf{u}),  \Pi^{0,E}_{k-1} \nabla \cdot \mathbf{v}_h \Big) + S^E_\nabla \Big((\mathbf{I}-  \boldsymbol{\Pi}^{\nabla,E}_{k}) \mathbf{u}_I, (\mathbf{I}-  \boldsymbol{\Pi}^{\nabla,E}_{k}) \mathbf{v}_h\Big) \Big)\nonumber \\
		&\leq C \sum_{E \in \Omega_h} \tau^{1/2}_{2,E}\Big( \|  \Pi^{0,E}_{k-1} (\nabla \cdot\mathbf{u}_I-\nabla \cdot\mathbf{u})\|_{0,E} +  \lambda_1^\ast \|\nabla(\mathbf{I}-  \boldsymbol{\Pi}^{\nabla,E}_{k}) \mathbf{u}_I\|_{0,E} \Big)\vertiii{(\mathbf{v}_h, q_h)}_{E}\nonumber \\
		&\leq C \sum_{E \in \Omega_h} \tau^{1/2}_{2,E}  h^k_E \| \mathbf{u}\|_{k+1,E} \vertiii{(\mathbf{v}_h, q_h)}_{E}. \label{etaL-2}
	\end{align}
	\textbf{Step 3.} Estimate of $\eta_{\mathcal{L}_3}$: Following \cite[Lemma 5.10]{vem28m}, we arrive at  
	\begin{align}
		\eta_{\mathcal{L}_3} &\leq C \sum_{E \in \Omega_h} \tau^{1/2}_{3,E} h^{k-1}_E \|p\|_{k,E} \vertiii{(\mathbf{v}_h, q_h)}_{E}. \label{etaL-3}    
	\end{align}
	Thus combining \eqref{etaL-1}, \eqref{etaL-2} and \eqref{etaL-3}, we obtain the result \eqref{etaL-0}. 
\end{proof}
Following the above analysis, the convergence theorem is given as follows:
\begin{theorem} \label{convergence}
	Under the assumptions \textbf{(A1)} and \textbf{(A2)}, let $(\mathbf{u}, p) \in \mathbf{V} \times Q$  and $(\mathbf{u}_h, p_h) \in \mathbf{V}^k_h \times Q^k_h$  be the solution of problem \eqref{discf} and VE problem \eqref{nvem}, respectively. Further we consider $\tau_{1,E} \sim h_E$, $\tau_{2,E} \sim \mathcal{O}(1)$ and $\tau_{3,E} \sim h^2_E$ for all $E\in \Omega_h$, then it holds that
	\begin{align}
		\vertiii{(\mathbf{u}- \mathbf{u}_h, p-p_h)}^2 &\lesssim \sum_{E \in \Omega_h} h^{2k}_E \Big(1 + \mu + {\dfrac{1}{\alpha}} + \gamma h^2_E  + \Lambda^2_E h^2_E(\|\boldsymbol{\mathcal{B}}\|_{k,\infty,E} + \mathcal{B}_E + \gamma)^2  + \delta^2_E \mathcal{B}^2_E \, +  \nonumber \\ & \qquad  h_E \mathcal{B}_E^2 \Big) \|\mathbf{u}\|^2_{k+1,E} + \sum_{E \in \Omega_h} h^{2k}_E \|p\|^2_{k,E} +  \sum_{E \in \Omega_h} h^{2k+2}_E  \Lambda_E^2 h^2_E \|\mathbf{f}\|^2_{k+1,E}. \label{main_est}
	\end{align}
	
\end{theorem}
\begin{proof}
	The proof of Theorem \ref{convergence} follows from Lemma \ref{ei}, Proposition \ref{prop}, and Lemmas \ref{eta_f}, \ref{etaA} and \ref{etaL}. 
\end{proof}
\begin{remark} \label{vel_first}
	The validity of the convergence Theorem \ref{convergence} depends on the parameters $\Lambda_E$ and $\delta_E$ as defined in \eqref{param}. Notably, the convergence theorem \ref{convergence} shows \texttt{uniform convergence in the energy norm as $\mu \rightarrow 0$ for the generalized Oseen problem \eqref{modeleq}.} Furthermore, Theorem \ref{convergence} is not restrictive to $\mathcal{B}_E >0$ since $\mathcal{B}_E =0 $ gives $\boldsymbol{\mathcal{B}}_{|E}=\mathbf{0}$ and thus the corresponding terms vanish. Therefore, when  $\boldsymbol{\mathcal{B}}= \mathbf{0} $ and $\gamma > 0$, the convergence theorem guarantees \texttt{uniform convergence in the energy norm for the Brinkman equation}. 
	Additionally, the proposed method effectively represents the diffusion-dominated regimes. The inclusion of $\mathcal{L}_{1,h}$ does not affect the convergence behavior of the considered problem. 
\end{remark}

\begin{remark} \label{last1}
	\texttt{The case $\gamma=0$ and $\boldsymbol{\mathcal{B}} \neq \mathbf{0}$ (classic Oseen problem)}: From \eqref{main_est}, we obtain
	\begin{align*}
		\vertiii{(\mathbf{u}- \mathbf{u}_h, p-p_h)}^2 & \lesssim \sum \limits_{E \in \Omega_h} h^{2k}_E \Big( 1 + \mu  + {\dfrac{1}{\alpha}} + \Lambda^2_E h^2_E(\|\boldsymbol{\mathcal{B}}\|_{k,\infty,E} + \mathcal{B}_E)^2  + \frac {h^2_E\mathcal{B}^2_E}{{\mu}} +  h_{E} \mathcal{B}_E^2 \Big)\|\mathbf{u}\|^2_{k+1,E}\nonumber \\ & \qquad + \sum_{E \in \Omega_h} h^{2k}_E \big(\|p\|^2_{k,E} +  \Lambda_E^2 h^4_E \|\mathbf{f}\|^2_{k,E}\big).
	\end{align*}
	In the convection-dominated regimes, we achieve uniform convergence in the $L^2$-norm of pressure, whereas the error estimate in energy norm contains a degenerative term $\frac {\mathcal{B}^2_E}{\mu} $ that is weighted by a factor of $h^2_E$, which effectively reduces the influence of the diffusion coefficient. {Additionally, when $\gamma=0$, we obtain $\alpha \sim \mu$ as $\mu \rightarrow 0$. Therefore, we obtain a quasi-uniform convergence in the energy norm when $\mu\rightarrow 0$. For this case, the uniform convergence in the energy norm is an open problem, and will be considered as future work.}
\end{remark}

\begin{remark} \label{last}
	\texttt{The case $\gamma=0$ and $\boldsymbol{\mathcal{B}} = \mathbf{0}$ (Stokes problem)}: From \eqref{main_est}, we obtain the following estimate: 
	\begin{align}
		\vertiii{(\mathbf{u}- \mathbf{u}_h, p-p_h)}^2 & \lesssim \sum_{E \in \Omega_h} h^{2k}_E \Big(1+ {\mu} +  \frac{2 + 5 \Theta_2 + 4(1+ C_{gen}) (\mu + C_{gen})^{1/2} }{4 \Theta_1}  \Big) \|\mathbf{u}\|^2_{k+1,E} \nonumber \\ & \qquad  + \sum_{E \in \Omega_h} h^{2k}_E \big(\|p\|^2_{k,E} +  \tilde{\Lambda}_E^2 h^4_E \|\mathbf{f}\|^2_{k,E}\big),\label{main_stk}
	\end{align}
	where $\tilde{\Lambda}_E:=\min\{ \frac{1}{\sqrt{\mu}}, \frac{1}{\sqrt{\tau_{2,E}}}\}$ with $\tau_{2,E} \sim \mathcal{O}(1)$.
	
	Thus, the error estimate  \eqref{main_stk} converges uniformly in the energy norm as $\mu \rightarrow 0$, whereas in \cite{vem028,mvem21} the $L^2$-error in pressure becomes quasi-uniform.
\end{remark}

\subsection{Alternative discrete convective form} \label{sec:5b}
In this section we introduce an alternative discrete convective form to be used in the place of $c^{skew,E}_h(\cdot,\cdot)$ in the stabilized VE problem \eqref{nvem}. Following \cite{vem26,mvem15}, we consider the following:
\begin{align}
	\widehat{c}^E_h(\mathbf{w}_h, \mathbf{z}_h):&= \int_{E}  (\boldsymbol{\Pi}^{0,E}_{k-1} \nabla \mathbf{w}_h) \boldsymbol{\mathcal{B}} \cdot \boldsymbol{\Pi}^{0,E}_{k} \mathbf{z}_h\, dE  \qquad \,\, \text{for all} \,\, \mathbf{w}_h, \mathbf{z}_h \in \mathbf{V}_h^k(E),\nonumber
	\intertext{with the associated skew-symmetric part}
	\widehat{c}^{skew, E}_h(\mathbf{w}_h, \mathbf{z}_h):&= \dfrac{1}{2} \Big[\widehat{c}^E_h(\mathbf{w}_h, \mathbf{z}_h) - \widehat{c}^E_h(\mathbf{z}_h, \mathbf{w}_h) \Big].\label{newc}
\end{align}
Notably, the global bilinear form $\widehat{c}^{skew}_h(\cdot, \cdot)$ associated to \eqref{newc} satisfies the following: 

\noindent$\bullet$ $\widehat{c}^{skew}_h(\mathbf{w}_h, \mathbf{w}_h)=0 \qquad \,\, \text{for all} \,\, \mathbf{w}_h \in \mathbf{V}^k_h$. \\
\noindent$\bullet$ 	$\widehat{c}^{skew}_h(\mathbf{w}_h, \mathbf{z}_h) \leq C \mathcal{B} \|\nabla \mathbf{w}_h\|_{0,\Omega} \|\nabla \mathbf{z}_h\|_{0,\Omega} \qquad \,\, \text{for all} \,\, \mathbf{w}_h, \mathbf{z}_h \in \mathbf{V}^k_h$.\\

We will now investigate the well-posedness of the VE problem \eqref{nvem} along with the discrete convective form $\widehat{c}^{skew,E}_h(\cdot,\cdot)$. In this sequel, we first bound $s_3$ (see Step 3 in  proof of Theorem \ref{wellposed}) as follows:
\begin{align}
	s_3 := \alpha \widehat{c}^{skew, E}_h(\mathbf{u}_h, \mathbf{w}_h) &\leq C \alpha \mathcal{B} \|\nabla \mathbf{u}_h\|_{0,\Omega} \|\nabla \mathbf{w}_h\|_{0,\Omega}, \label{well-c}
\end{align}
provided that $\|\nabla \mathbf{w}_h\|_{0,\Omega} = \xi_0 \|p_h\|_{0,\Omega}$.

We now present the well-posedness theorem for the VE problem \eqref{nvem} along with the discrete convective
form $\widehat{c}^{skew}_h$:
\begin{theorem}{(Well-posedness)} \label{wellposed-c}
	Under the mesh regularity assumption \textbf{(A1)}, there exists a positive constant $\Theta$ independent of the mesh size $h$, such that for any $(\mathbf{u}_h, p_h) \in \mathbf{V}^k_{h} \times Q^k_h$, it holds that 
	\begin{align}
		\sup \limits_{(\mathbf{0},0) \neq(\mathbf{z}_h, q_h) \in \mathbf{V}^k_{h} \times Q^k_h} \dfrac{(A_{h}+\mathcal{L}_h)[(\mathbf{u}_h, p_h), (\mathbf{z}_h, q_h)]}{\vertiii{(\mathbf{z}_h, q_h)}}  \geq \Theta \vertiii{(\mathbf{u}_h, p_h)}, \label{wellpos-0c}
	\end{align}
	Moreover, the stabilized VE problem \eqref{nvem} along with the discrete convective
	form $\widehat{c}^{skew}_h$ has a unique solution.
\end{theorem}
\begin{proof}
	The proof of Theorem \ref{wellposed-c} follows from Theorem \ref{wellposed} and the bound \eqref{well-c} with $\Theta=  \mathbb N_1/\sqrt{\mathbb N_2}$. 
\end{proof}

\begin{remark}
	Well-posedness Theorem \ref{wellposed-c} is robust with respect to parameters $\boldsymbol{\mathcal{B}}$ and $\gamma$, as we observed in Theorem \ref{wellposed}. Hence, Theorem \ref{wellposed-c} also guarantees the well-posedness of the Brinkman and Stokes problems.
\end{remark}

Hereafter, we derive the error estimate in the energy norm for the VE problem \eqref{nvem} along with the discrete form $\widehat{c}^{skew}_h$. We first establish the classic error estimate for the term:
\begin{align}
	\eta_{\widehat{c}}:&= \widehat{c}^{skew}(\mathbf{u}, \mathbf{v}_h) - \widehat{c}^{skew}_h(\mathbf{u}_I, \mathbf{v}_h).
\end{align}

\begin{lemma} \label{hatc}
	(Estimation of $\eta_{\widehat{c}}$) Let the assumptions \textbf{(A1)} and \textbf{(A2)} be valid. Then the term $\eta_{\widehat{c}}$ can be bounded as follows
	\begin{align}
		\Scale[0.99]{	|\eta_{\widehat{c}}| \leq C \Big(\sum \limits_{E \in \Omega_h } h^{2k}_E\big( \Lambda_E^2 \|\boldsymbol{\mathcal{B}}\|_{k,\infty,E}^2 + \frac{\mathcal{B}^2_E}{\gamma} + \delta_E^2(\mathcal{B}_E + \|\boldsymbol{\mathcal{B}}\|_{1,\infty,E} + \|\boldsymbol{\mathcal{B}}\|_{k,\infty,E} )^2 \big) \|\mathbf{u}\|^2_{k+1,E}\Big)^{1/2} \vertiii{(\mathbf{v}_h,q_h)}. }\label{hatc0}
	\end{align}
\end{lemma}
\begin{proof}
	Concerning $\eta_{\widehat{c}}$, we proceed as follows:
	\begin{align}
		\eta_{\widehat{c}}:&= \widehat{c}^{skew}(\mathbf{u}, \mathbf{v}_h) - \widehat{c}^{skew}_h(\mathbf{u}_I, \mathbf{v}_h) \nonumber \\
		&= \widehat{c}^{skew}(\mathbf{u}, \mathbf{v}_h) - \widehat{c}^{skew}_h(\mathbf{u}, \mathbf{v}_h) + \widehat{c}^{skew}_h(\mathbf{u}-\mathbf{u}_I, \mathbf{v}_h) \nonumber \\
		&=: \eta_{\widehat{c},1} + \eta_{\widehat{c},2}. \label{hatc1}
	\end{align}
	$\bullet$ Estimation of $\eta_{\widehat{c},1}$: We proceed as follows
	\begin{align}
		\eta_{\widehat{c},1} &= \widehat{c}^{skew}(\mathbf{u}, \mathbf{v}_h) - \widehat{c}^{skew}_h(\mathbf{u}, \mathbf{v}_h) \nonumber \\
		&= \dfrac{1}{2} \sum_{E \in \Omega_h} \Big[ \Big( (\nabla \mathbf{u}) \boldsymbol{\mathcal{B}}, \mathbf{v}_h \Big) - \Big( ( \boldsymbol{\Pi}^{0,E}_{k-1} \nabla \mathbf{u})\boldsymbol{\mathcal{B}},  \boldsymbol{\Pi}^{0,E}_k \mathbf{v}_h\Big) + \Big( (\boldsymbol{\Pi}^{0,E}_{k-1} \nabla   \mathbf{v}_h)\boldsymbol{\mathcal{B}},  \boldsymbol{\Pi}^{0,E}_k \mathbf{u}\Big) -\Big((\nabla \mathbf{v}_h) \boldsymbol{\mathcal{B}}, \mathbf{u} \Big)  \Big]   \nonumber \\ 
		&= \dfrac{1}{2} \sum_{E \in \Omega_h} \Big[ \Big( (\nabla \mathbf{u}) \boldsymbol{\mathcal{B}}, \mathbf{v}_h - \boldsymbol{\Pi}^{0,E}_k \mathbf{v}_h\Big) + \Big( (\nabla \mathbf{u} - \boldsymbol{\Pi}^{0,E}_{k-1} \nabla \mathbf{u})\boldsymbol{\mathcal{B}},  \boldsymbol{\Pi}^{0,E}_k \mathbf{v}_h\Big) \,+ \nonumber \\ & \qquad \quad \Big( ( \boldsymbol{\Pi}^{0,E}_{k-1} \nabla \mathbf{v}_h)\boldsymbol{\mathcal{B}},  \boldsymbol{\Pi}^{0,E}_k \mathbf{u} - \mathbf{u} \Big) + \Big( \big( \boldsymbol{\Pi}^{0,E}_{k-1} \nabla \mathbf{v}_h -  \nabla \mathbf{v}_h\big) \boldsymbol{\mathcal{B}}, \mathbf{u} \Big) \Big] \nonumber \\
		&= \dfrac{1}{2} \sum_{E \in \Omega_h} \Big[ \Big( (\nabla \mathbf{u}) \boldsymbol{\mathcal{B}} - \boldsymbol{\Pi}^{0,E}_{k-1}\big((\nabla \mathbf{u}) \boldsymbol{\mathcal{B}} \big), \mathbf{v}_h - \boldsymbol{\Pi}^{0,E}_k \mathbf{v}_h\Big) + \Big( (\nabla \mathbf{u} - \boldsymbol{\Pi}^{0,E}_{k-1} \nabla \mathbf{u}),  (\boldsymbol{\Pi}^{0,E}_k \mathbf{v}_h)\boldsymbol{\mathcal{B}}^T - \boldsymbol{\Pi}^{0,E}_{k-1}(\mathbf{v}_h \boldsymbol{\mathcal{B}}^T)\Big) \nonumber \\ & \qquad \quad + \Big( ( \boldsymbol{\Pi}^{0,E}_{k-1} \nabla \mathbf{v}_h)\boldsymbol{\mathcal{B}},  \boldsymbol{\Pi}^{0,E}_k \mathbf{u} - \mathbf{u} \Big) + \Big( \big( \boldsymbol{\Pi}^{0,E}_{k-1} \nabla \mathbf{v}_h -  \nabla \mathbf{v}_h\big), \mathbf{u}\boldsymbol{\mathcal{B}}^T - \boldsymbol{\Pi}^{0,E}_{k-1}(\mathbf{u}\boldsymbol{\mathcal{B}}^T) \Big) \Big] \nonumber\\
		&=: \widehat{c}_1 +  \widehat{c}_2 +  \widehat{c}_3+  \widehat{c}_4. \label{hatc3}
	\end{align}
	{	Using Lemma \ref{lemmaproj1}, we obtain
		\begin{align}
			\widehat{c}_1 +  \widehat{c}_3 &\leq C \sum_{E \in \Omega_h } \Big[h^{k+1}_E \|\boldsymbol{\mathcal{B}}\|_{k,\infty,E} \|\mathbf{u}\|_{k+1,E} \|\nabla \mathbf{v}_h\|_{0,E} + h^{k+1}_E \mathcal{B}_E \|\mathbf{u}\|_{k+1,E} \|\nabla \mathbf{v}_h\|_{0,E} \Big] \nonumber \\
			&\leq C \sum_{E \in \Omega_h } h^{k}_E \delta_E (\|\boldsymbol{\mathcal{B}}\|_{k,\infty,E} + \mathcal{B}_E) \|\mathbf{u}\|_{k+1,E}  \vertiii{(\mathbf{v}_h,q_h)}_{E}. \label{hatc4}
	\end{align}}
	Concerning the second term $\widehat{c}_2$, we proceed as follows:
	\begin{align}
		\widehat{c}_2 & \leq C \sum_{E \in \Omega_h }h^k_E \|\mathbf{u}\|_{k+1,E} \|(\boldsymbol{\Pi}^{0,E}_k \mathbf{v}_h)\boldsymbol{\mathcal{B}}^T - \boldsymbol{\Pi}^{0,E}_{k-1}(\mathbf{v}_h \boldsymbol{\mathcal{B}}^T)\|_{0,E} \nonumber \\
		& \leq C \sum_{E \in \Omega_h }h^k_E \|\mathbf{u}\|_{k+1,E} \Big[\|(\boldsymbol{\Pi}^{0,E}_k \mathbf{v}_h - \mathbf{v}_h) \boldsymbol{\mathcal{B}}^T \|_{0,E} + \| \mathbf{v}_h\boldsymbol{\mathcal{B}}^T - \boldsymbol{\Pi}^{0,E}_{k-1}(\mathbf{v}_h \boldsymbol{\mathcal{B}}^T)\|_{0,E}\Big] \nonumber \\
		& \leq C \sum_{E \in \Omega_h }h^k_E \|\mathbf{u}\|_{k+1,E} \Big[ h_E\mathcal{B}_E\|\nabla \mathbf{v}_h \|_{0,E} + h_E \| \boldsymbol{\mathcal{B}}\|_{1,\infty,E} \|\nabla \mathbf{v}_h\|_{0,E} \Big] \nonumber \\
		& \leq C \sum_{E \in \Omega_h } h^k_E  \delta_E (\mathcal{B}_E + \| \boldsymbol{\mathcal{B}}\|_{1,\infty,E})\|\mathbf{u}\|_{k+1,E} \vertiii{(\mathbf{v}_h,q_h)}_E. \label{hatc5}
	\end{align}
	{Using the assumption \textbf{(A2)}, orthogonality of $L^2$-projector and Lemma \ref{lemmaproj1}, we obtain
		\begin{align}
			\widehat{c}_4 &\leq C \sum_{E \in \Omega_h }  \|\mathbf{u}\boldsymbol{\mathcal{B}}^T - \boldsymbol{\Pi}^{0,E}_{k-1}(\mathbf{u}\boldsymbol{\mathcal{B}}^T) \|_{0,E} \| \boldsymbol{\Pi}^{0,E}_{k-1} \nabla \mathbf{v}_h -  \nabla \mathbf{v}_h\|_{0,E} \nonumber \\
			& \leq C \sum_{E \in \Omega_h} h^{k}_E |\mathbf{u}\boldsymbol{\mathcal{B}}^T|_{k,E} \| \nabla (\mathbf{I}- \boldsymbol{\Pi}^{0,E}_k)\mathbf{v}_h\|_{0,E} \nonumber \\
			& \leq C \sum_{E \in \Omega_h} h^{k}_E \|\boldsymbol{\mathcal{B}}\|_{k,\infty,E} \|\mathbf{u}\|_{k,E}\| \nabla (\mathbf{I}- \boldsymbol{\Pi}^{\nabla,E}_k)\mathbf{v}_h\|_{0,E} \qquad \text{ \big(applying \eqref{est3} \big)}\nonumber \\
			&\leq C \sum_{E \in \Omega_h } h ^{k}_E \Lambda_E \|\boldsymbol{\mathcal{B}}\|_{k,\infty,E} \|\mathbf{u}\|_{k+1,E} \vertiii{(\mathbf{v}_h,q_h)}_{E}. \label{hatc6}
	\end{align}}
	$\bullet$ Estimation of $\eta_{\widehat{c},2}$: Employing the continuity of the projection operator and Lemma \ref{lemmaproj2}, yields
	\begin{align}
		\eta_{\widehat{c},2} &\leq C\sum_{E \in \Omega_h } \mathcal{B}_E \Big( \|\nabla(\mathbf{u} - \mathbf{u}_I)\|_{0,E} \|\mathbf{v}_h\|_{0,E} + \|\nabla \mathbf{v}_h\|_{0,E} \|\mathbf{u} - \mathbf{u}_I\|_{0,E}\Big) \nonumber \\
		& \leq C\sum_{E \in \Omega_h } \mathcal{B}_E \Big( h^{k}_E \|\mathbf{u} \|_{k+1,E} \|\mathbf{v}_h\|_{0,E} + h^{k+1}_E \|\mathbf{u} \|_{k+1,E} \|\nabla \mathbf{v}_h\|_{0,E} \Big) \nonumber \\
		& \leq C\sum_{E \in \Omega_h } \Big(\frac{1}{\sqrt{\gamma}} + \delta_E \Big) \mathcal{B}_E  h^{k}_E \|\mathbf{u} \|_{k+1,E} \vertiii{(\mathbf{v}_h, q_h)}_E. \label{hatc02}
	\end{align}
	Finally, combining  \eqref{hatc4}, \eqref{hatc5}, \eqref{hatc6}  and  \eqref{hatc02}, we obtain the result \eqref{hatc0}. 
\end{proof}

\begin{theorem} \label{conver:hatc}
	Under the assumptions \textbf{(A1)} and \textbf{(A2)}, let $(\mathbf{u}, p) \in \mathbf{V} \times Q$  and $(\mathbf{u}_h, p_h) \in \mathbf{V}^k_h \times Q^k_h$  be the solution of problem \eqref{discf} and VE problem \eqref{nvem}, respectively. Further we consider $\tau_{1,E} \sim h_E$, $\tau_{2,E} \sim \mathcal{O}(1)$ and $\tau_{3,E} \sim h^2_E$, then it holds that
	\begin{align}
		\vertiii{(\mathbf{u}- \mathbf{u}_h, p-p_h)}^2 &\lesssim \sum_{E \in \Omega_h} h^{2k}_E \Big[ \mathscr{R}\big({h_E}, \mu, \boldsymbol{\mathcal{B}}, \gamma, \delta_E, \Lambda_E\big) + \frac{\mathcal{B}_E^2}{\gamma} \Big]\|\mathbf{u}\|^2_{k+1,E} + \sum_{E \in \Omega_h} h^{2k}_E \|p\|^2_{k,E} \, + \nonumber \\
		& \qquad  \sum_{E \in \Omega_h} h^{2k+2}_E  \Lambda_E^2 h^2_E \|\mathbf{f}\|^2_{k,E}, \label{main_hatc} 	
	\end{align}
	{where $\mathscr{R}\big(h_E,\mu, \boldsymbol{\mathcal{B}}, \gamma, \delta_E, \Lambda_E\big):= 1 + \mu+ \frac{1}{\alpha} + h_E \mathcal{B}_E^2 + \gamma h^2_E +  \gamma^2 h_E^2 \Lambda^2_E + \Lambda_E^2\|\boldsymbol{\mathcal{B}}\|_{k,\infty,E}^2 + \delta_E^2\big( \|\boldsymbol{\mathcal{B}}\|_{k,\infty,E} + \mathcal{B}\|_{1,\infty,E} + \mathcal{B}_E \big)^2 $.}
\end{theorem}
\begin{proof}
	The proof of Theorem \ref{conver:hatc} follows from Theorem \ref{convergence}  and Lemma \ref{hatc}. 
\end{proof}

\begin{remark} \label{vel_c}
	Notably, the convergence Theorem \ref{convergence} shows \texttt{uniform convergence in the energy norm as $\mu \rightarrow 0$ for the generalized Oseen problem \eqref{modeleq}.} Furthermore, for the case $\boldsymbol{\mathcal{B}} \neq \mathbf{0}$ and $\gamma \rightarrow 0$, the estimate \eqref{main_hatc} become quasi-uniform in the energy norm. The above estimate is not applicable for $\gamma=0$, i.e., for the Stokes problem, similar to \cite{mvem15}.
\end{remark}

\section{Numerical results} 
\label{sec-5}
This section is dedicated to validating the stability and robustness of the proposed stabilized virtual element scheme  \eqref{nvem} from the practical standpoint. To achieve this, we conduct a series of four numerical examples. The first two examples focus on scalar and variable coefficient problems to validate the theoretical results. In the last two examples, we discuss the classic benchmark Navier-Stokes problem and an Oseen problem characterized by layers. The computational domain is the unit square  $\Omega=(0, 1)^2$, and the sample of computational meshes is illustrated in Figure \ref{samp}. Furthermore, for simplicity, we will follow the same definitions of stabilization parameters defined in Remark \ref{rem_tau}. We present our numerical results for VEM orders $k=1$ and $k=2$. We remark that the numerical behavior of the VE problem \eqref{nvem} along with the bilinear forms $c^{skew}_h(\cdot, \cdot)$ and $\widehat{c}^{skew}_h(\cdot, \cdot)$ is observed almost similar, as found in \cite{mvem16}. Therefore, we will show the numerical results for
the VE problem \eqref{nvem} along with the form $c^{skew}_h(\cdot, \cdot)$.
\begin{figure}[h]
	\centering
	\subfloat[$\Omega_1$]{\includegraphics[height=3cm, width=4cm]{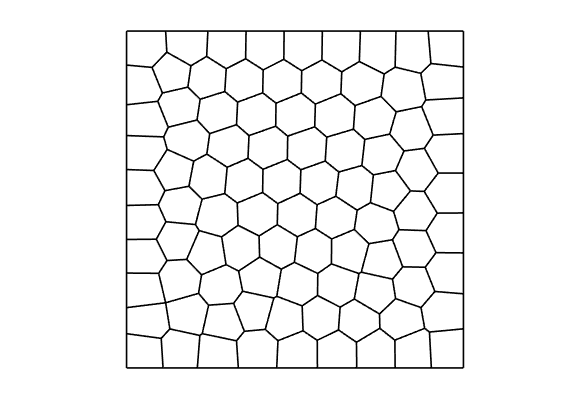}}
	\subfloat[$\Omega_2$]{\includegraphics[height=3cm, width=3cm]{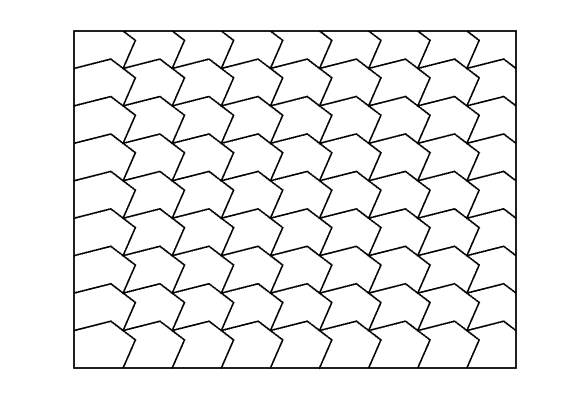}}~~
	\subfloat[$\Omega_3$]{\includegraphics[height=3cm, width=3cm]{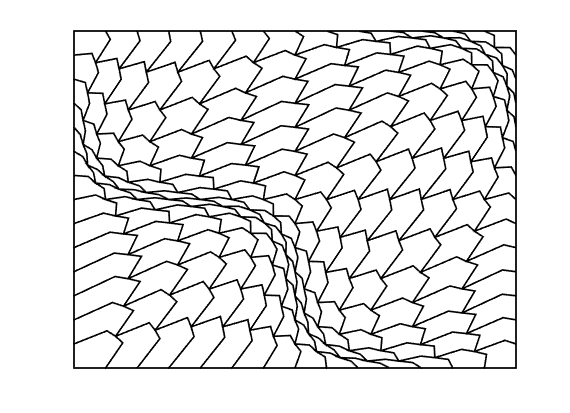}}~~
	\subfloat[$\Omega_4$]{\includegraphics[height=3cm, width=3cm]{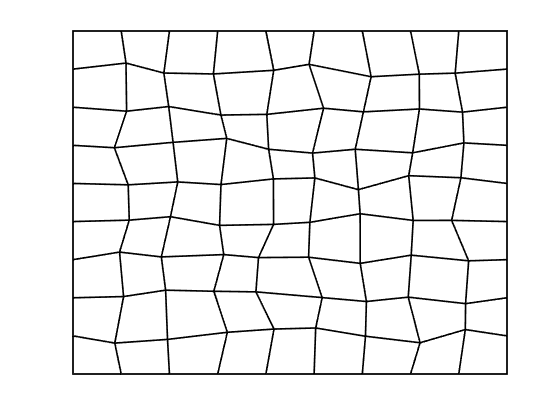}}
	\caption{ Computational meshes.}
	\label{samp} 
\end{figure}

The absolute computational errors for the numerical examples are computed as follows:  \newline

$\bullet$ The velocity error in $H^1$-semi norm \qquad	$E^{\mathbf{u}}_{H^1_0} := \sqrt{\sum\limits_{E\in\Omega_h}  \| \nabla (\mathbf{u} - \boldsymbol{\Pi}^{\nabla,E}_k \mathbf{u}_h ) \|_E^2}$. 

$\bullet$ The velocity error in $L^2$-norm \qquad $ E^\mathbf{u}_{L^2} := \sqrt{\sum\limits_{E\in\Omega_h} \| \mathbf{u}-\boldsymbol{\Pi}^{0,E}_k \mathbf{u}_h\|^2_E}$.

$\bullet$ The pressure error in $L^2$-norm \qquad  $ E^p_{L^2} := \sqrt{\sum\limits_{E\in\Omega_h} \| p-{\Pi}^{0,E}_k p_h\|^2_E}$.

\begin{figure}[h]
	\centering{ \Large \textbf{Diffusion-dominated regime: $\mathbf{\boldsymbol{\mu}=1}, \boldsymbol{\gamma=1}$.}}
	\centering
	\subfloat[$\Omega_1$]{\includegraphics[width=8cm]{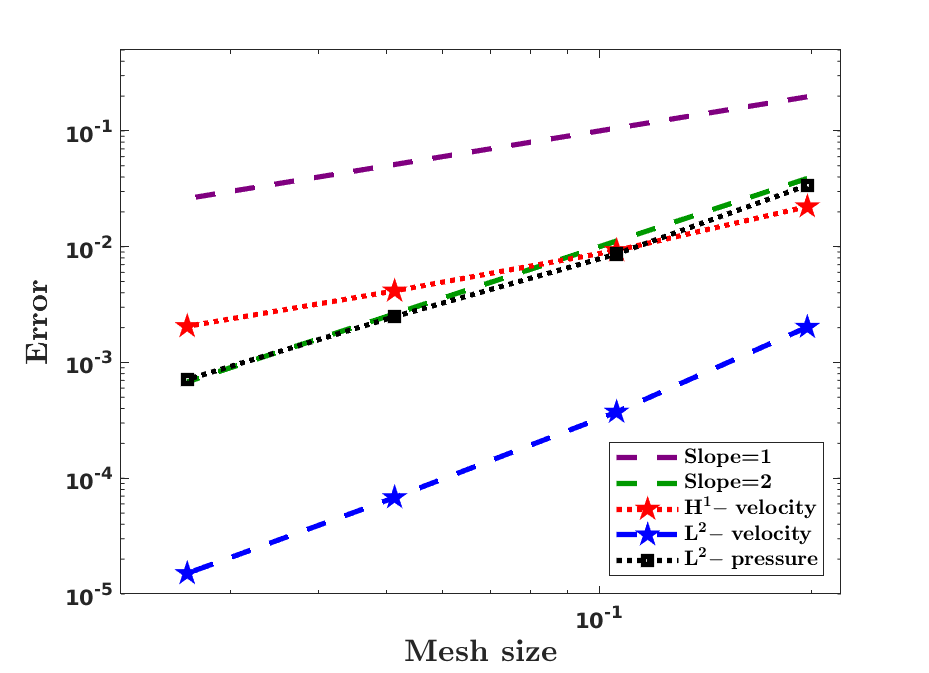}}
	\subfloat[$\Omega_1$]{\includegraphics[width=8cm]{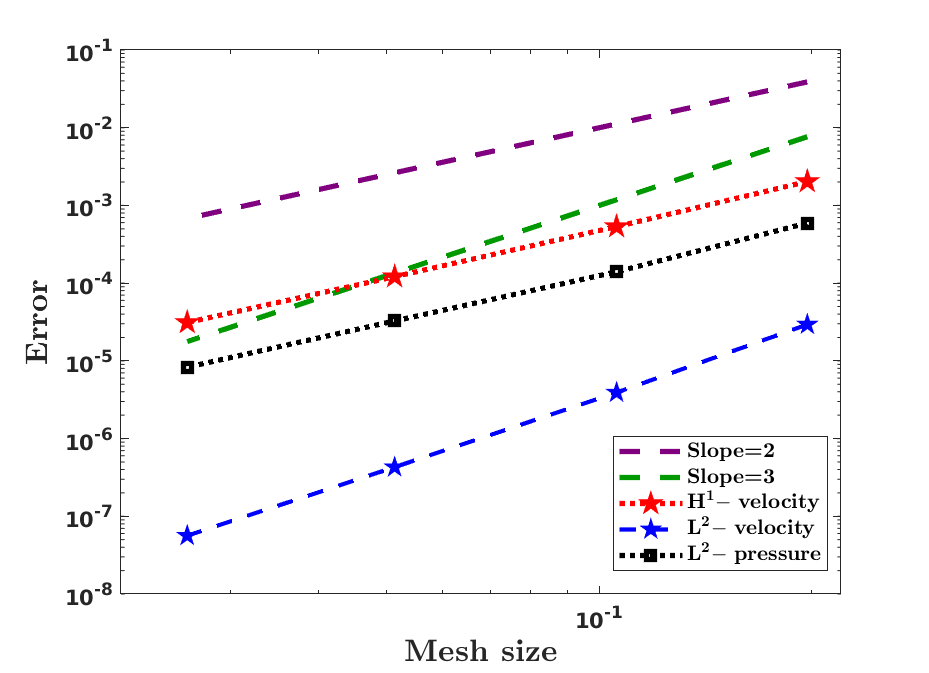}}\\
	\subfloat[$\Omega_3$]{\includegraphics[width=8cm]{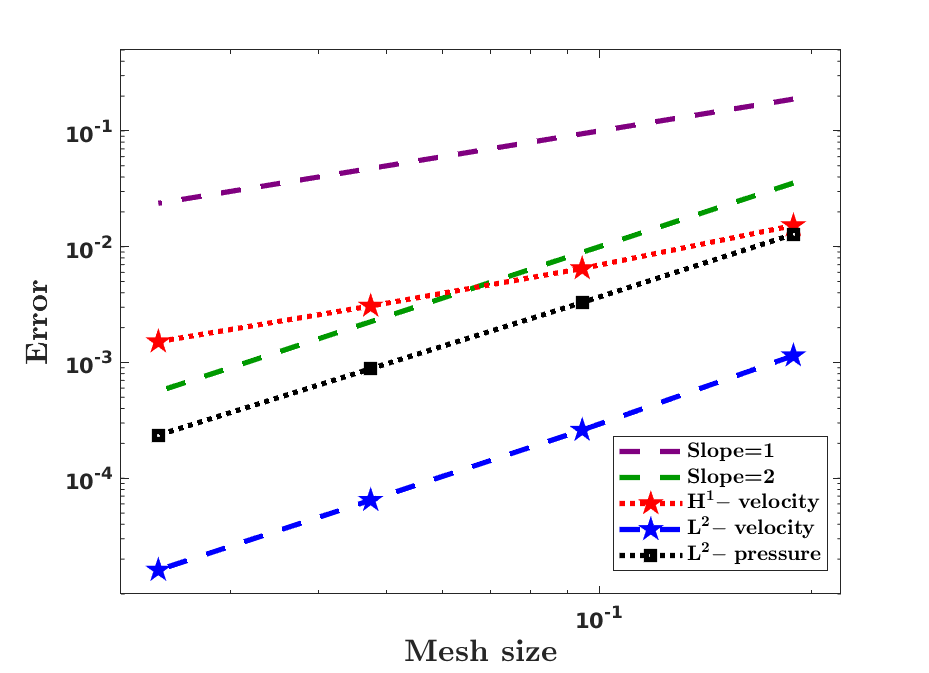}}
	\subfloat[$\Omega_3$]{\includegraphics[width=8cm]{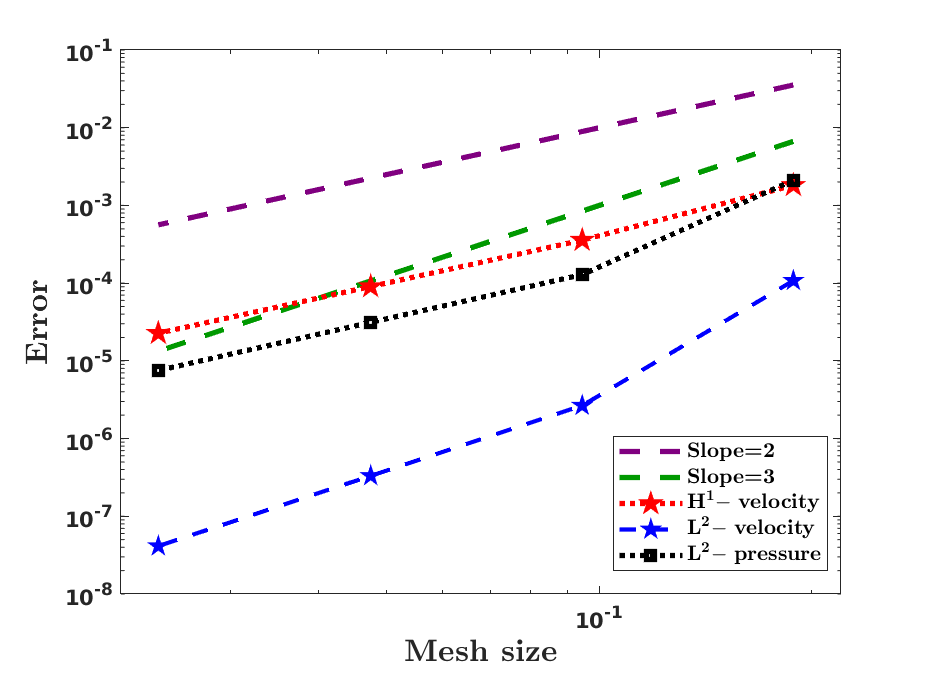}}\\
	\caption{Example 1. Convergence curves for VEM order $k=1$ (left) and $k=2$ (right).}
	\label{ex1_cng1} 
\end{figure}

\subsection{Example 1}
\label{case1}
In this example, our attention is directed towards a scalar-coefficient Oseen problem aimed at validating error estimates outlined in Section \ref{sec-4}. We consider the scalar convective flow field $\boldsymbol{\mathcal{B}}=(1, 1)^T$, and the load term $\mathbf{f}$ is chosen such that its exact solution is given by: 

\begin{align}
	\mathbf{u}(x,y):=\begin{bmatrix}
		2x^2y(2y-1)(x-1)^2(y-1) \\
		-2xy^2(2x-1)(x-1)(y-1)^2
	\end{bmatrix}, 
	\qquad p(x,y)=-2xy^2(2x-1)(x-1)(y-1)^2. \nonumber
\end{align}

We employ computational meshes $\Omega_1$ and $\Omega_3$ to investigate the convergence behavior of the proposed method for different cases, considering VEM order $k=1$ and $k=2$. \newline
\noindent $\bullet$ \textbf{Diffusion-dominated regime:}  We use $\mu=1$ and $\gamma=1$ to represent the diffusion-dominated case. In this context, our focus lies in validating that the proposed stabilized scheme aligns with theoretical expectations. The expected optimal rates for a regular problem, are $\mathcal{O}(h^k)$ for $H^1$-norm of velocity and $\mathcal{O}(h^{k})$ for the $L^2$-norm of the pressure. Although we have not derived the error estimates for velocity in the $L^2$-norm, the interpolation estimates suggest the best possibility of the optimal rates in the $L^2$-norm is $\mathcal{O}(h^{k+1})$. This also holds for convection-dominated problems.

The convergence curves depicted in Figure \ref{ex1_cng1} confirm the theoretical findings for VEM orders $k=1$ and $k=2$. Notably, on meshes $\Omega_1$ and $\Omega_3$, we observe a convergence rate of order $\mathcal{O}(h^2)$ for the $L^2$-norm of pressure with the lowest equal-order VEM. This shows the superconvergence in pressure, although we have not derived superconvergence estimates for pressure, which is consistent with FEM literature \cite{mfem31,fem11}.

\noindent $\bullet$ \textbf{Convection-dominated regime:} Here, we will discuss the two different cases: in the first, we consider $\mu=10^{-8}$ and $\gamma=1$, while in the second case, we set  $\mu=0$ and $\gamma=0$, considering this as a critical case. 

The convergence curves shown in Figure \ref{ex1_cng2} conclude that the proposed stabilized VEM method exhibits a strong agreement with theoretical expectations for the convection-dominated regime. Notably, the superconvergence in pressure is observed for the meshes $\Omega_1$ and $\Omega_3$ with VEM order $k=1$. Furthermore, we have investigated the second case $\mu=0$ and $\gamma=0$ to address Remark \ref{last1}. In this critical case, stabilized VEM seems to handle computational demands effectively, confirming optimal convergence rates with the superconvergence in $L^2$-rate of pressure for the lowest equal-order VEM.  

\begin{figure}[h]
	\centering{ \Large \textbf{Convection-dominated regime: $\mathbf{\boldsymbol{\mu}=10^{-8}}, \boldsymbol{\gamma=0}$.}}
	\subfloat[$\Omega_1$]{\includegraphics[width=8cm]{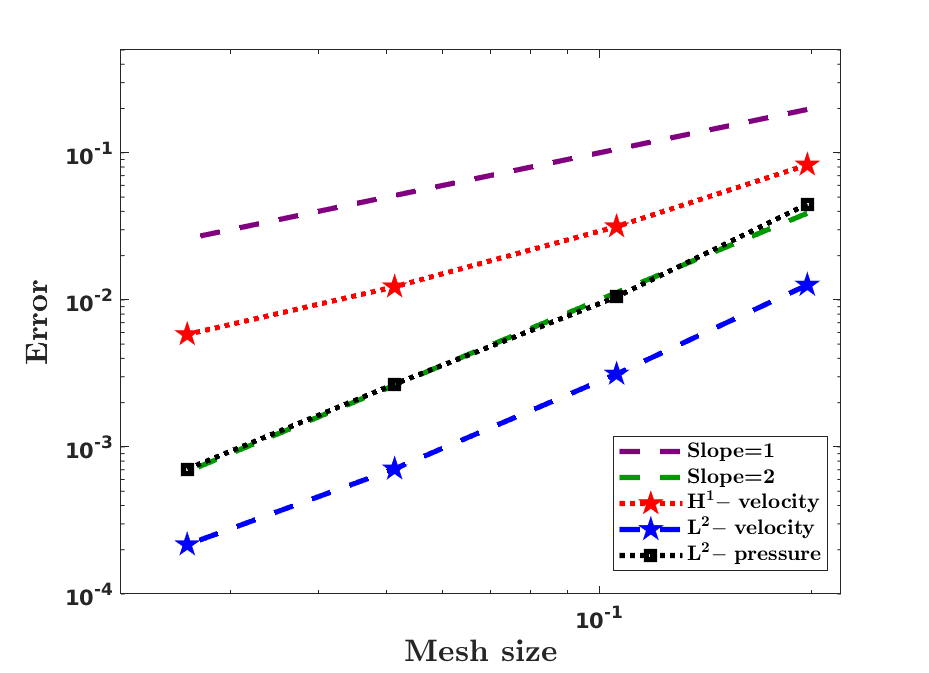}}
	\subfloat[$\Omega_1$]{\includegraphics[width=8cm]{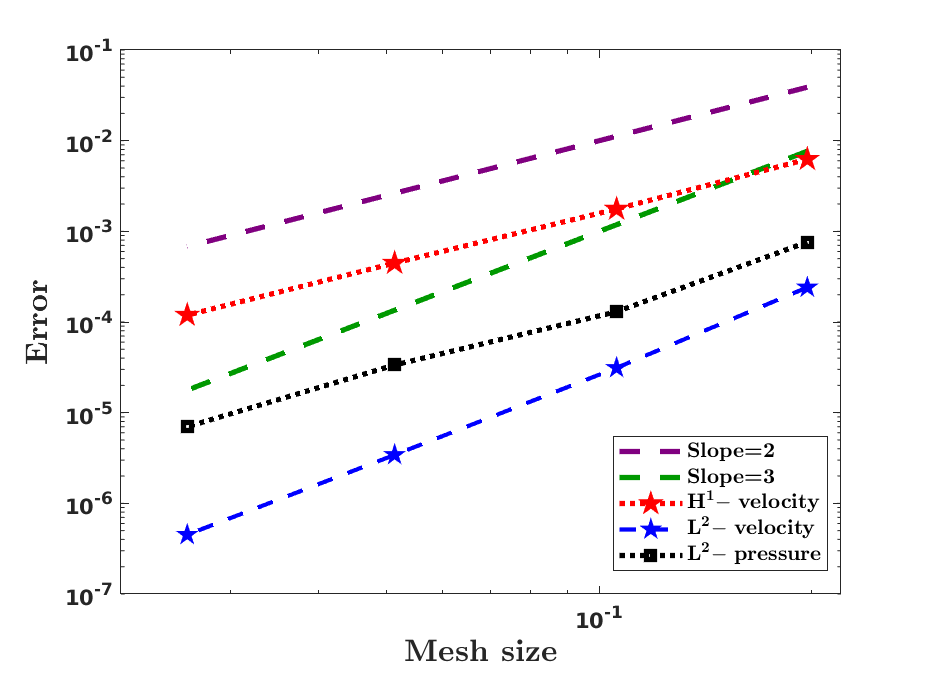}}\\
	\subfloat[$\Omega_3$]{\includegraphics[width=8cm]{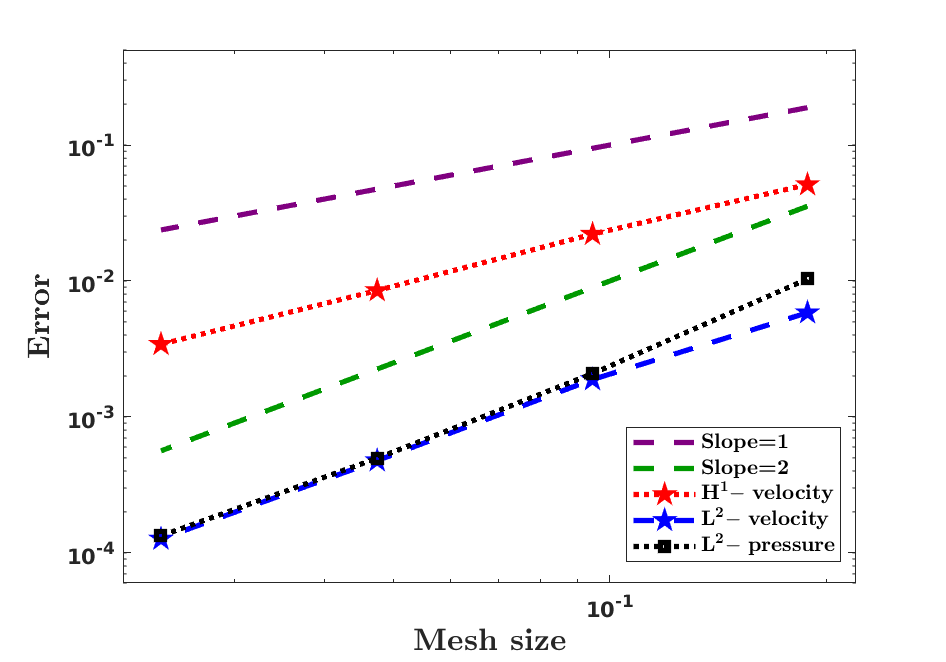}}
	\subfloat[$\Omega_3$]{\includegraphics[width=8cm]{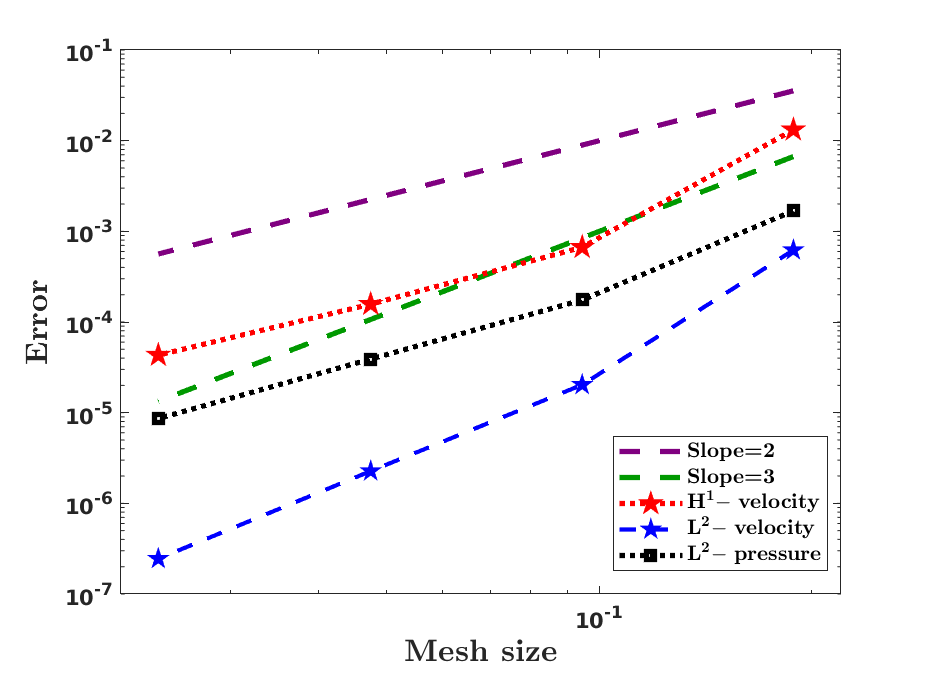}}\\
	\caption{Example 1. Convergence curves for VEM order $k=1$ (left) and $k=2$ (right).}
	\label{ex1_cng2} 
\end{figure}

\begin{figure}[h]
	\centering{ \Large \textbf{No diffusivity parameter: $\mathbf{\boldsymbol{\mu}=0, \boldsymbol{\gamma=0}}$.}}
	\subfloat[$\Omega_1$]{\includegraphics[width=8cm]{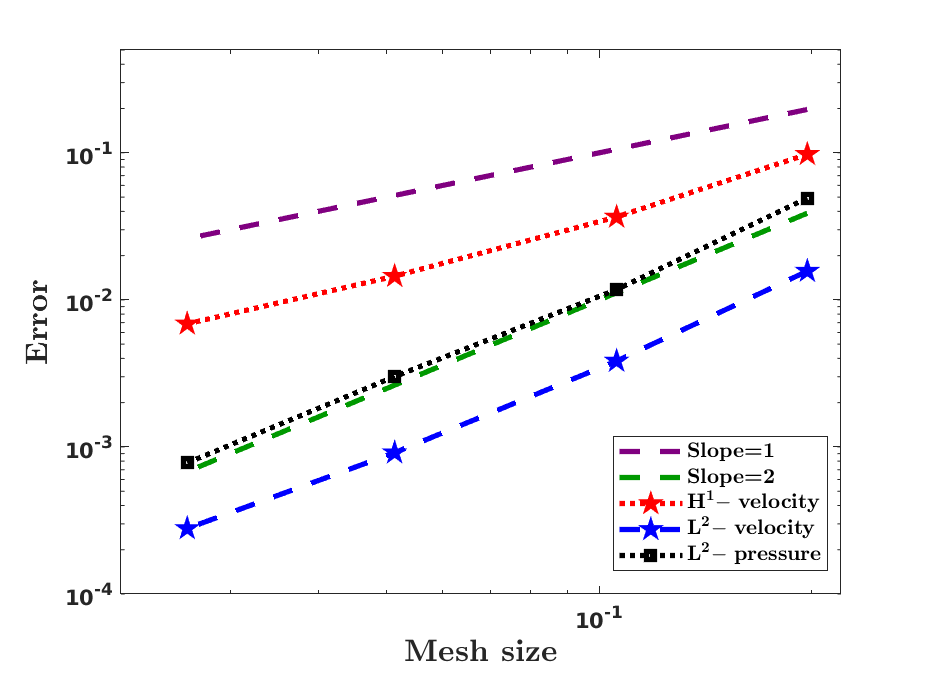}}
	\subfloat[$\Omega_1$]{\includegraphics[width=8cm]{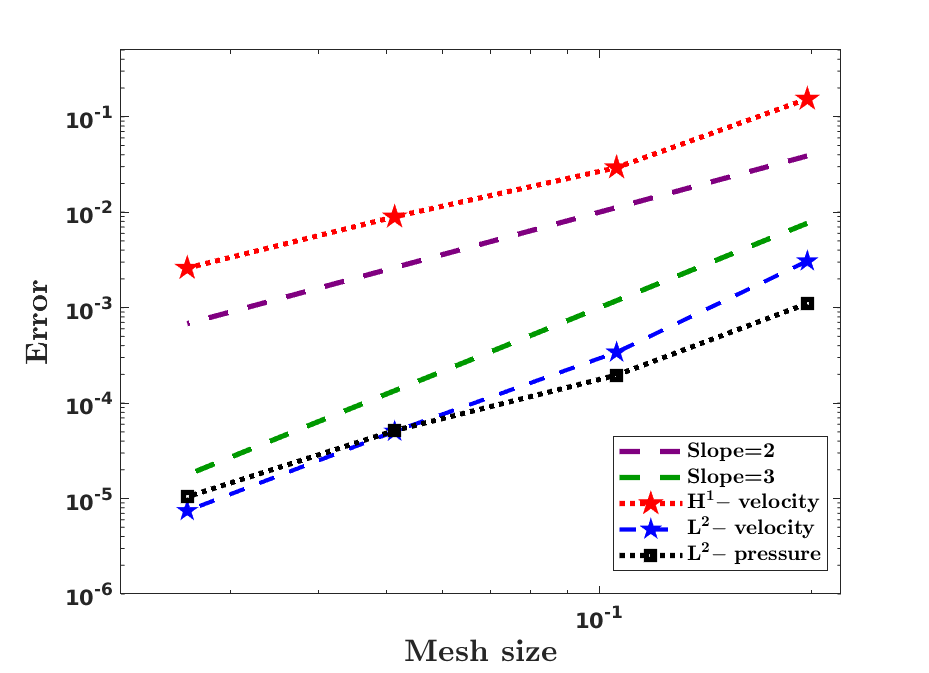}}\\
	\subfloat[$\Omega_3$]{\includegraphics[width=8cm]{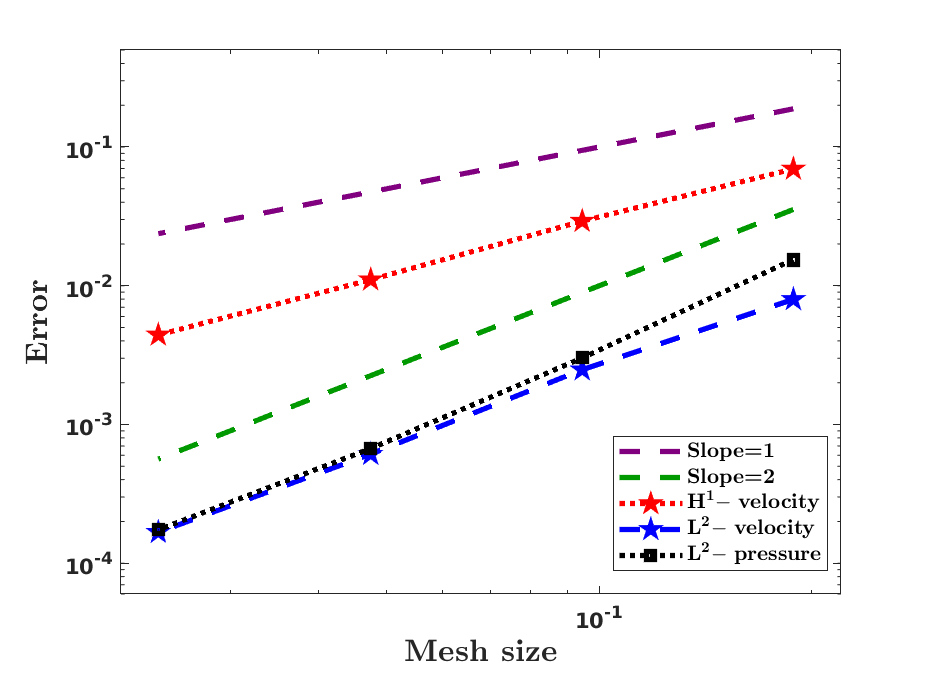}}
	\subfloat[$\Omega_3$]{\includegraphics[width=8cm]{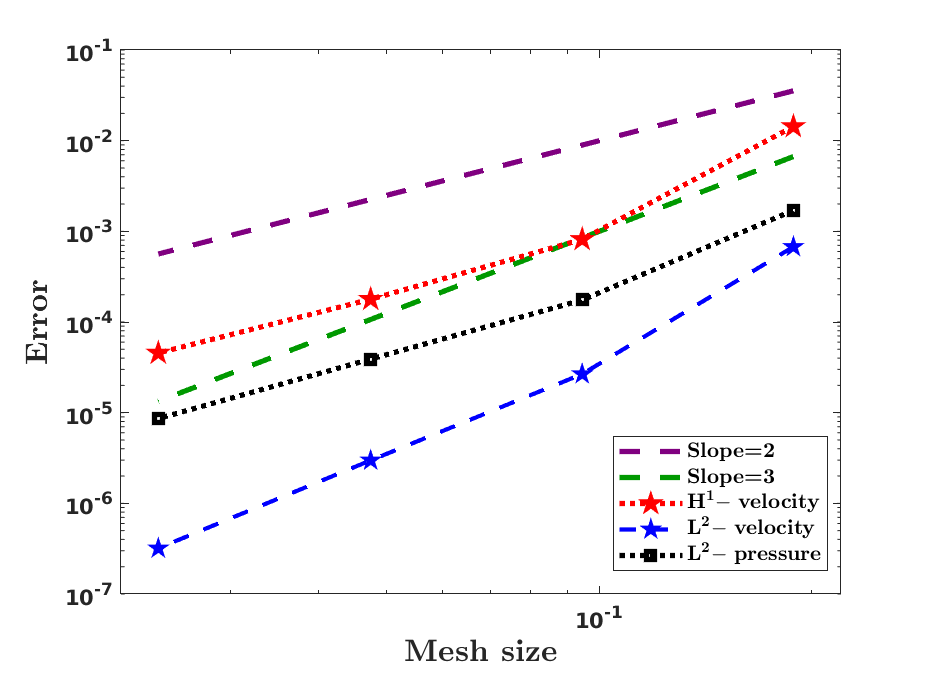}}\\
	\caption{Example 1. Convergence curves for VEM order $k=1$ (left) and $k=2$ (right).}
	\label{ex1_cng3} 
\end{figure}

\begin{figure}[h]
	\centering
	\subfloat[Velocity field]{\includegraphics[height=5cm, width=6cm]{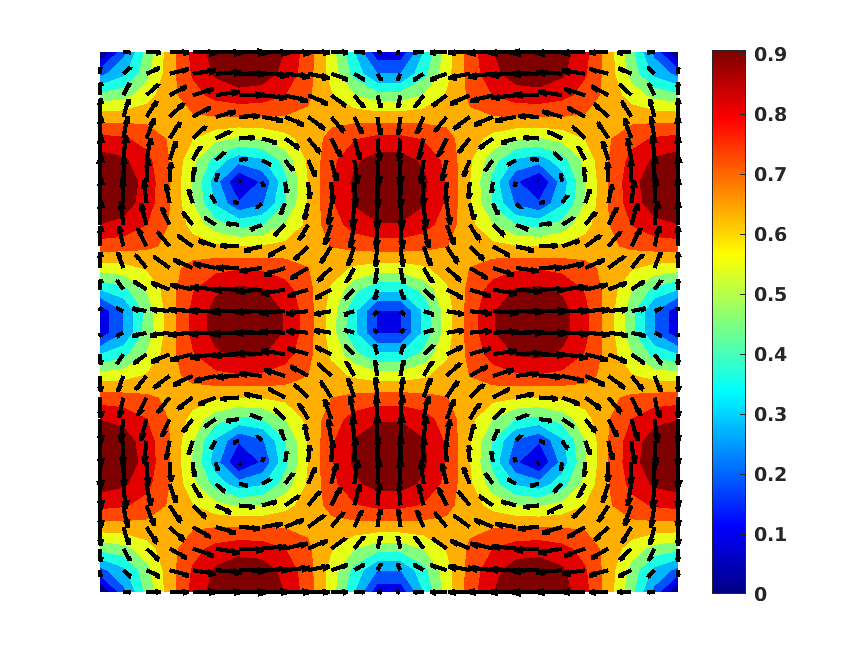}}
	\subfloat[Pressure.]{\includegraphics[height=5cm, width=6cm]{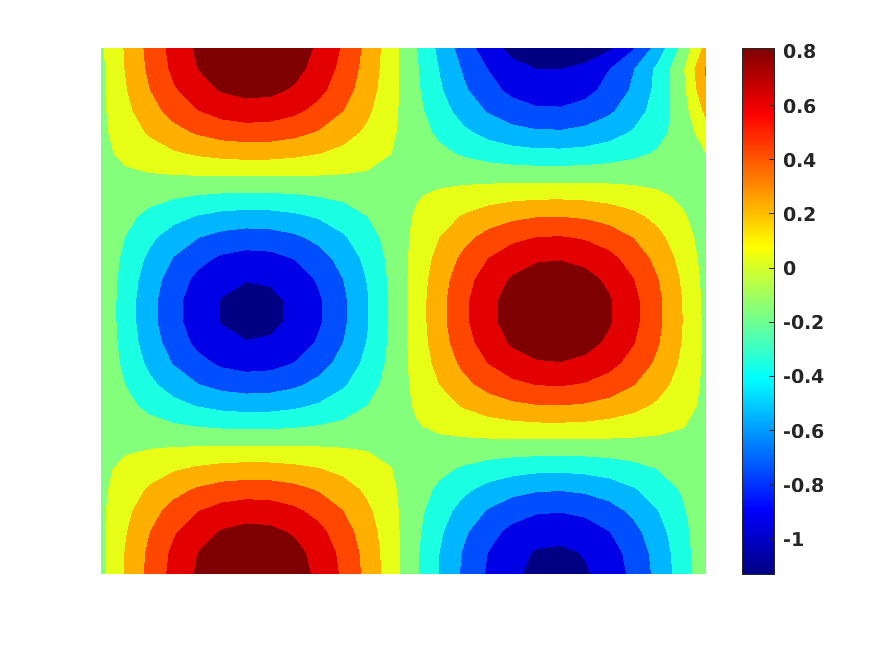}}
	\caption{Example 2. Stabilized velocity field and pressure for diffusion parameter $\mu=10^{-8}$.}
	\label{ex2_surf} 
\end{figure}
\begin{figure}[h]
	\centering{ \Large \textbf{Diffusion-dominated regime: $\mathbf{\boldsymbol{\mu}=1}$.}}
	\subfloat[$\Omega_2$]{\includegraphics[width=6.5cm]{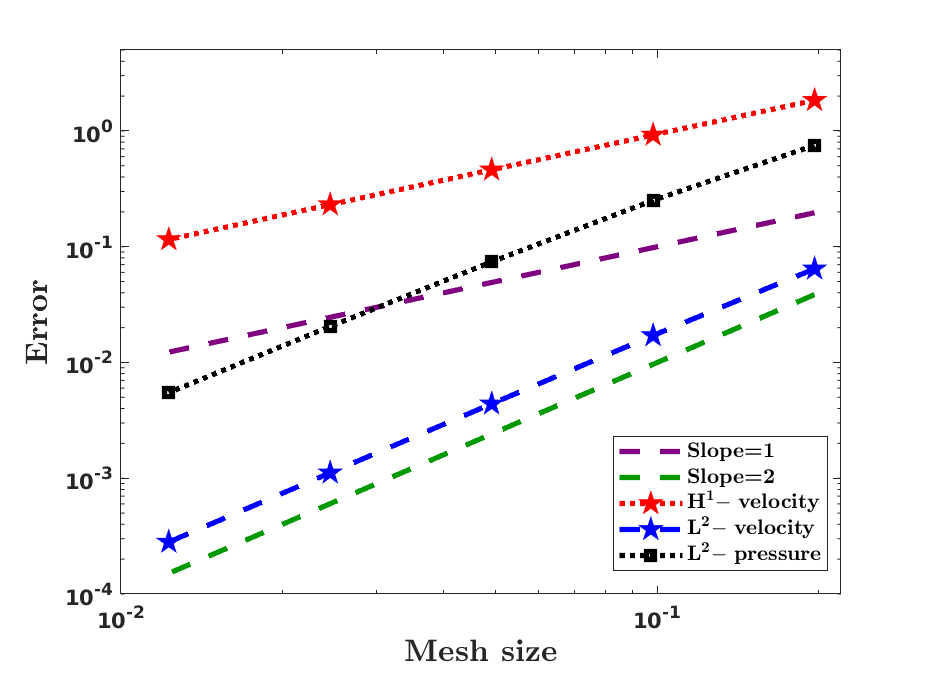}}
	\subfloat[$\Omega_2$]{\includegraphics[width=6.5cm]{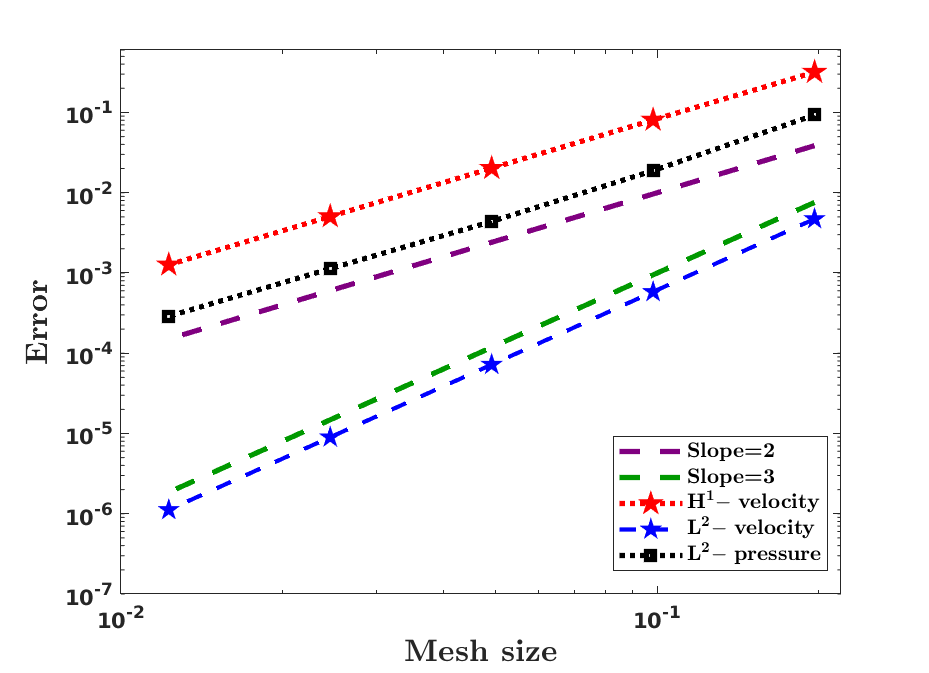}}\\
	\subfloat[$\Omega_4$]{\includegraphics[width=6.5cm]{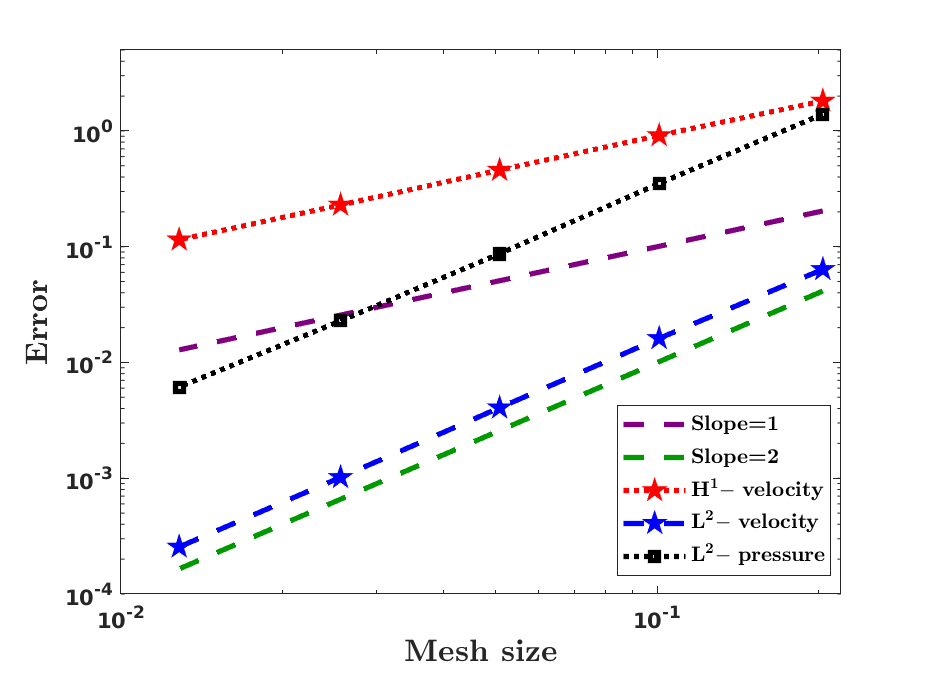}}
	\subfloat[$\Omega_4$]{\includegraphics[width=6.5cm]{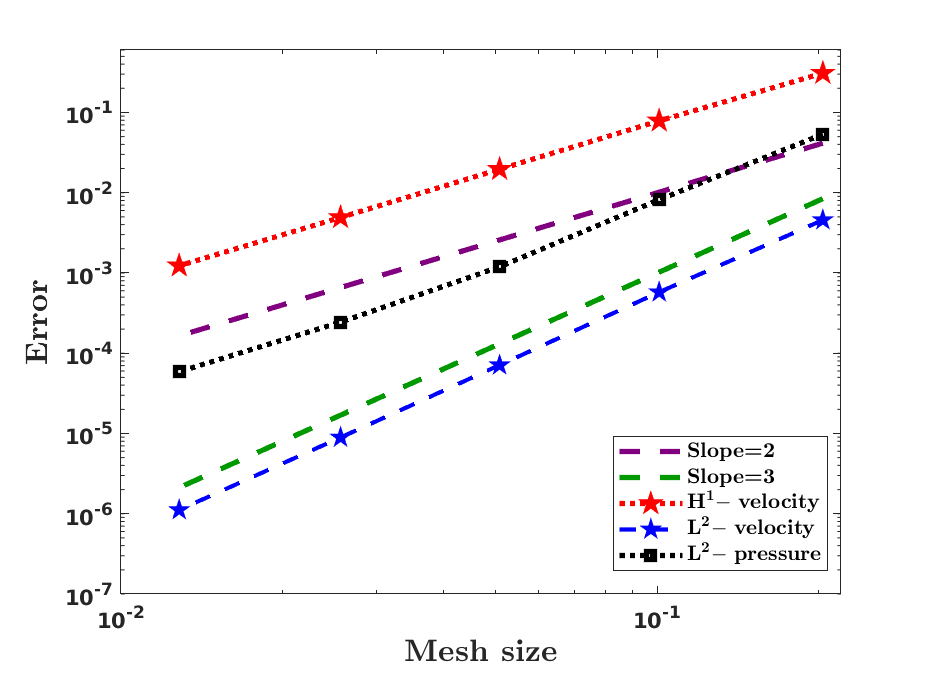}}
	\caption{Example 2. Convergence curves for VEM order $k=1$ (left) and $k=2$ (right).}
	\label{ex2_cng1} 
\end{figure}

\subsection{Example 2}
\label{case2}
In the second example, we introduce a variable coefficient-based problem \eqref{modeleq} to validate the theoretical estimates for both diffusion-dominated and convection-dominated regimes. Initially, we set $\gamma=0$ to emphasize this as a critical case, as highlighted in Remark \ref{last1}. We consider a variable convective flow field $\boldsymbol{\mathcal{B}}$, defined by

\begin{align}
	\boldsymbol{\mathcal{B}}(x,y):=\begin{bmatrix}
		10y(x+y^2)^4 + \frac{1}{3}\\
		-5(x+y^2)^4-\frac{1}{2}
	\end{bmatrix} \nonumber
\end{align}
Furthermore, the right-hand side source/load term and the boundary conditions are determined by the exact solution, as outlined below.
\begin{align}
	\mathbf{u}(x,y)=\begin{bmatrix}
		\sin(2\pi x) \cos(2\pi y) \\
		{-\cos(2\pi x) \sin(2\pi y)}
	\end{bmatrix}, 
	\qquad p(x,y)=\sin(2\pi x) \cos(2\pi y). \nonumber
\end{align}

\begin{figure}[h]
	\centering{ \Large \textbf{Convection-dominated: $\mathbf{\boldsymbol{\mu}=10^{-8}}$.}}
	\subfloat[$\Omega_2$]{\includegraphics[width=7.25cm]{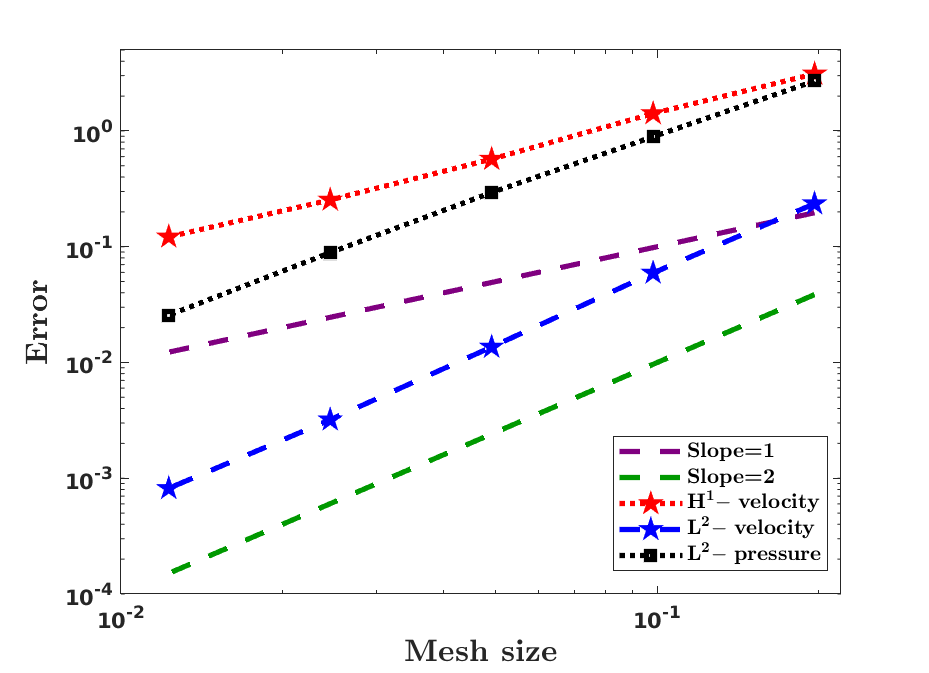}}
	\subfloat[$\Omega_2$]{\includegraphics[width=7cm]{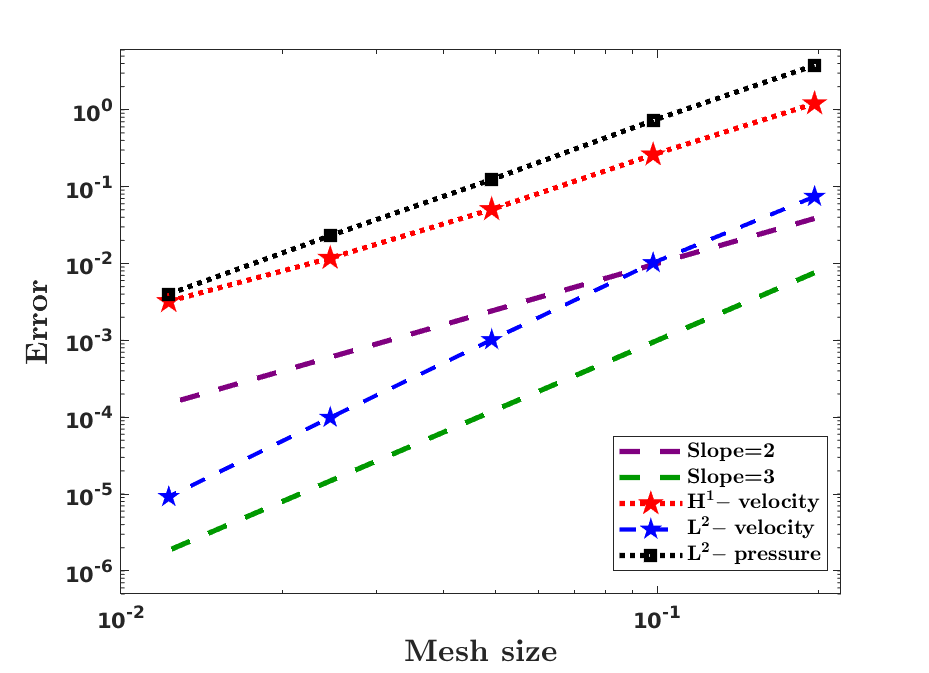}}\\
	\subfloat[$\Omega_4$]{\includegraphics[width=7.25cm]{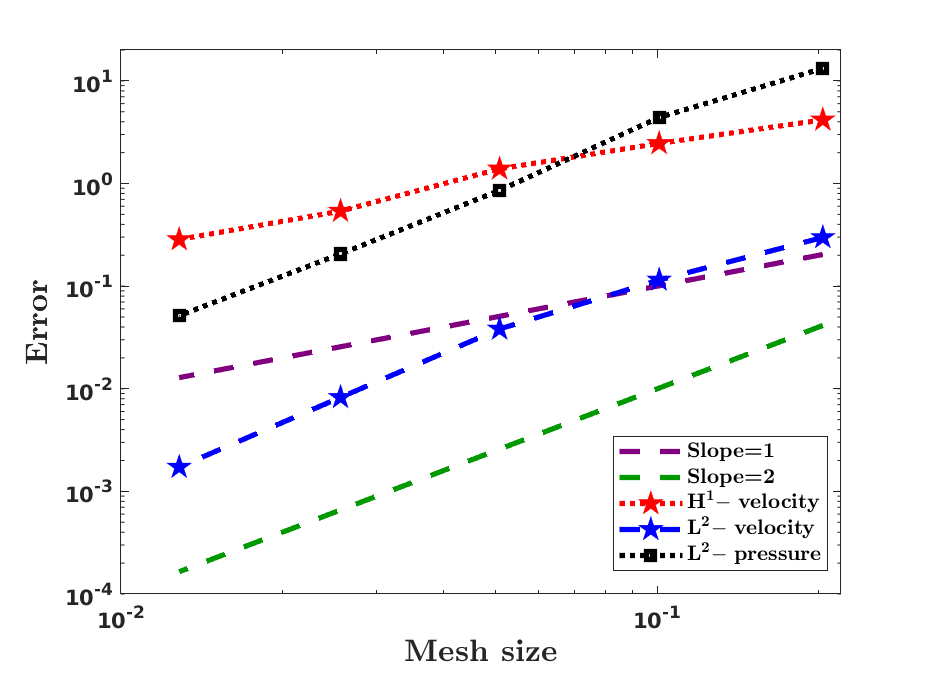}}
	\subfloat[$\Omega_4$]{\includegraphics[width=7.25cm]{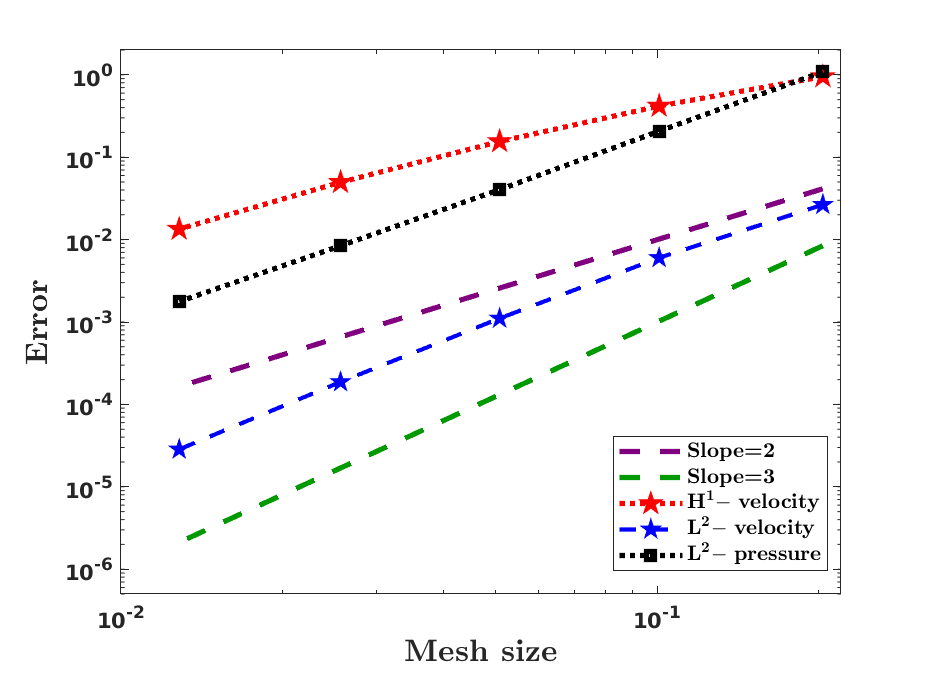}}
	\caption{Example 2. Convergence curves for VEM order $k=1$ (left) and $k=2$ (right).}
	\label{ex2_cng2} 
\end{figure}

We investigate the convergence behavior of the proposed stabilized VEM for variable flow field employing the meshes $\Omega_2$ and $\Omega_4$ for VEM order $k=1$ and $k=2$. The magnitude of the discrete velocity and discrete pressure are depicted in Figure \ref{ex2_surf}. The convergence study is given as follows: \newline
\noindent $\bullet$ \textbf{Diffusion-dominated regime:} For this case, we consider $\mu=1$. In Figure \ref{ex2_cng1}, we depict the convergence rates for velocity and pressure in $H^1$ and $L^2$-norms. The convergence plots show the optimal rates for velocity in $H^1$ and $L^2$-norms, and for the pressure in $L^2$-norm, as expected. Notably, we again observe the superconvergence of pressure in the $L^2$-norm for the lowest equal-order VEM.

\begin{figure}[h]
	\centering
	\subfloat[$r_1=r_2=0.1$.]{\includegraphics[width=6.5cm]{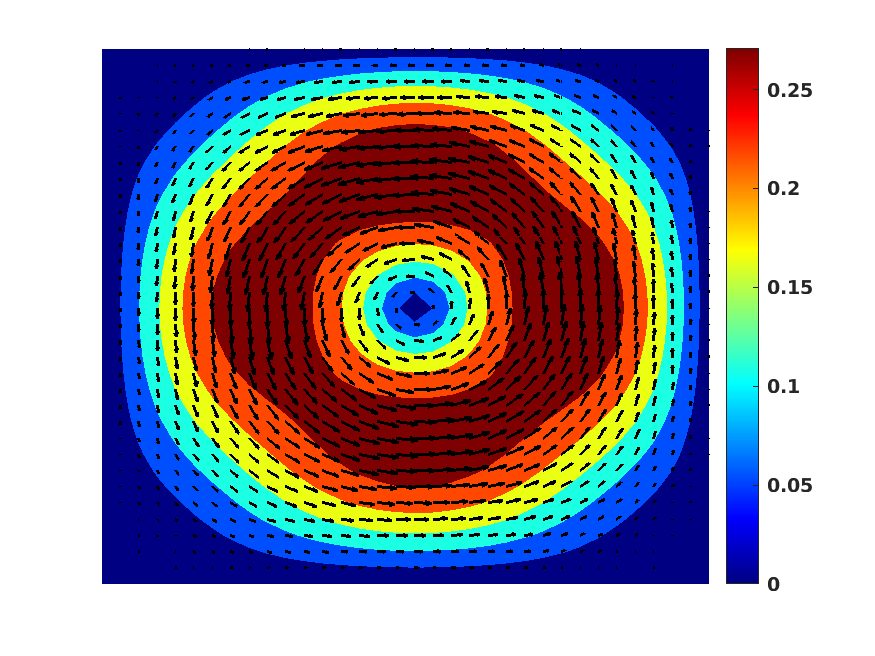}}
	\subfloat[$r_1=r_2=0.1$.]{\includegraphics[width=6.5cm]{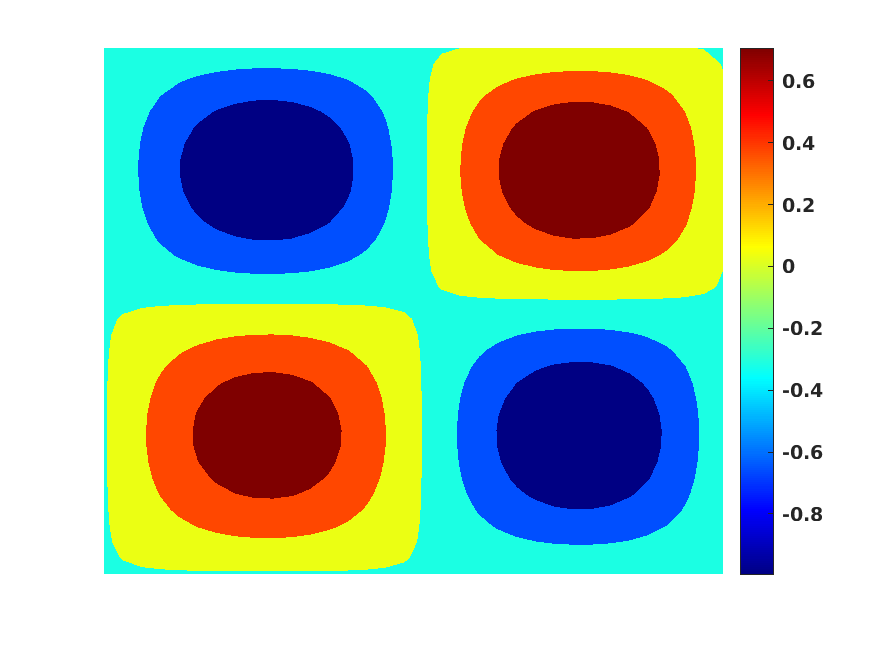}}\\
	\subfloat[$r_1=r_2=1.1$.]{\includegraphics[width=6.5cm]{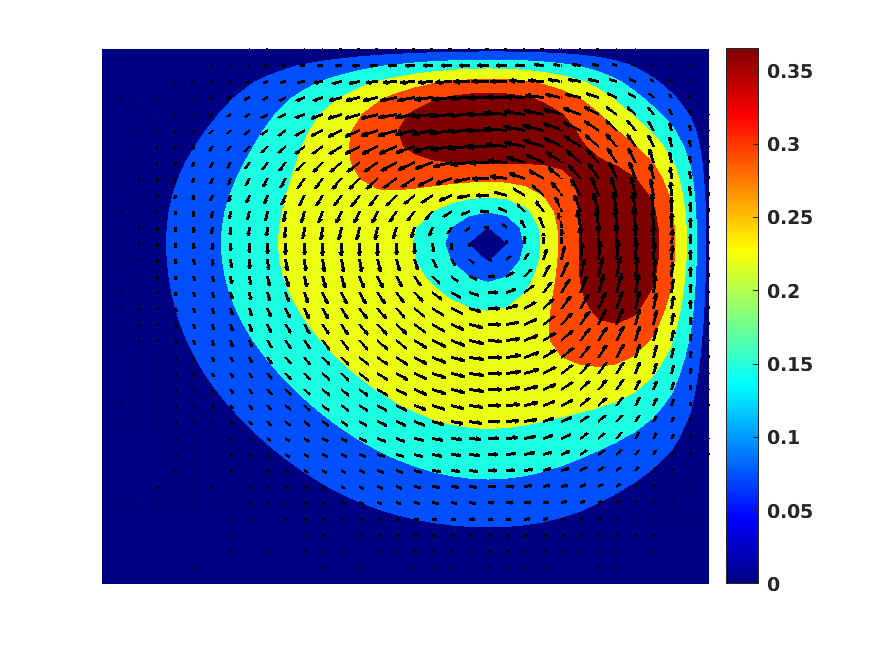}}
	\subfloat[$r_1=r_2=1.1$.]{\includegraphics[width=6.5cm]{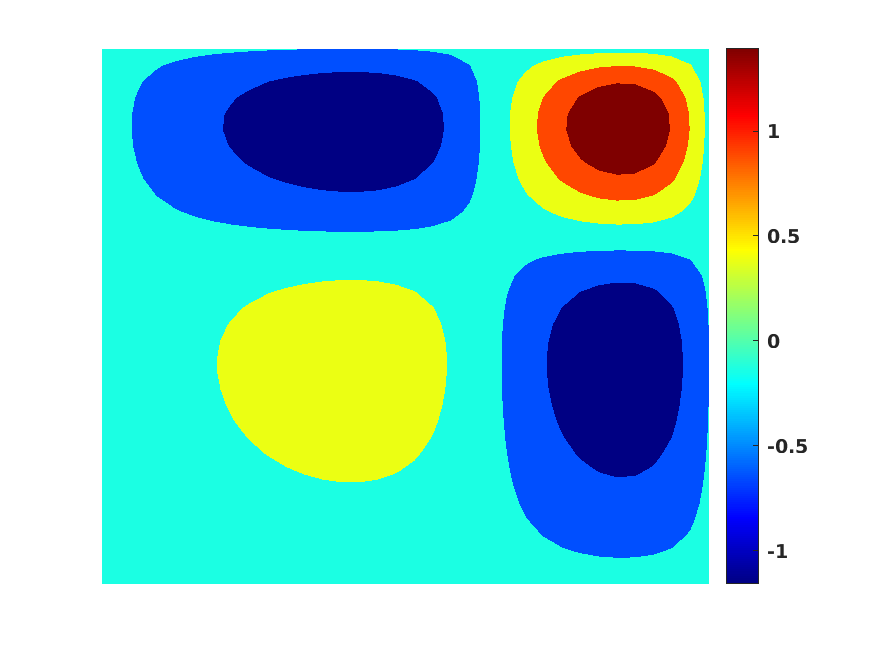}}\\
	\subfloat[$r_1=r_2=4.1$.]{\includegraphics[width=6.5cm]{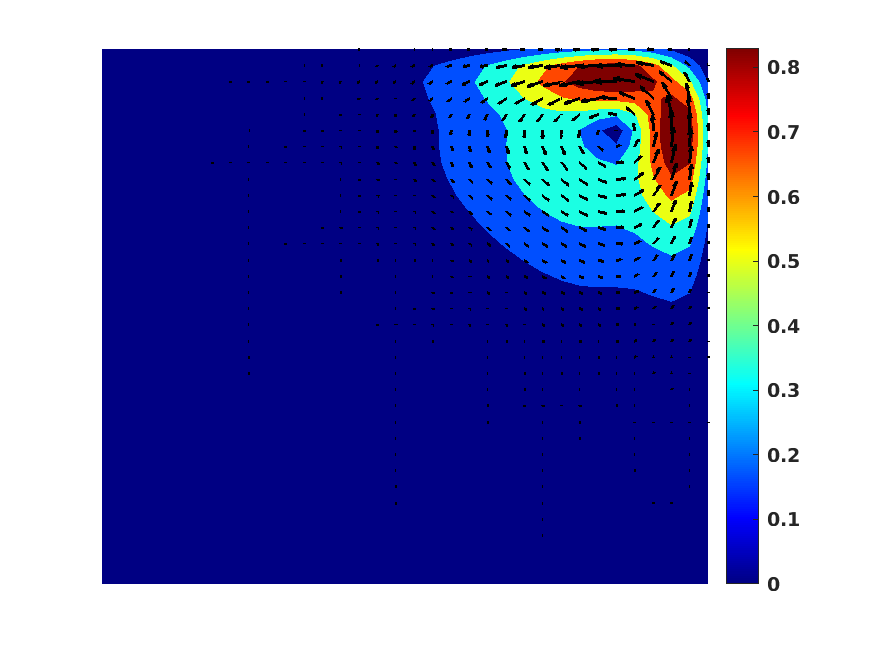}}
	\subfloat[$r_1=r_2=4.1$.]{\includegraphics[width=6.5cm]{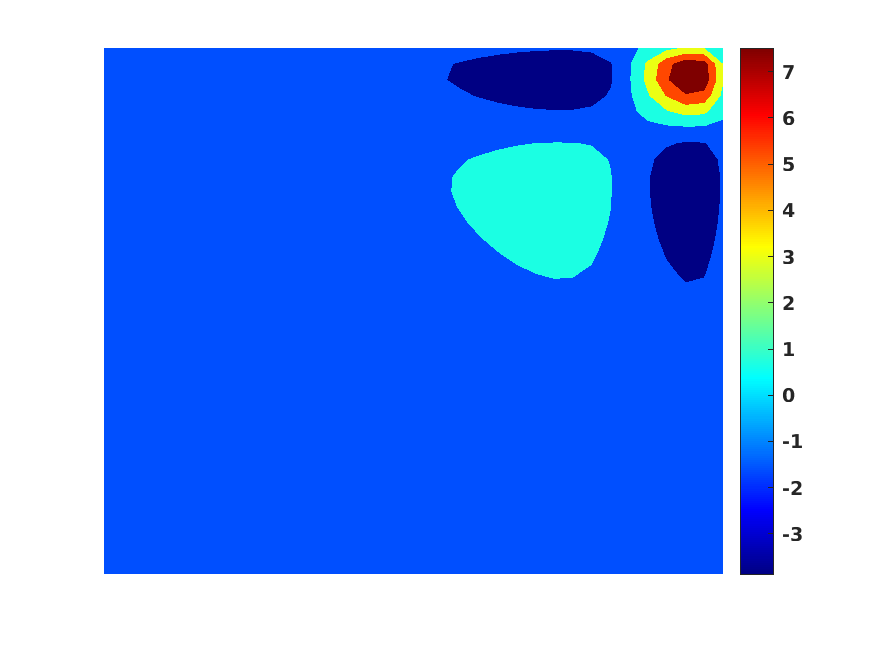}}
	\caption{Example 3. The stabilized velocity and pressure for different values of $r_1$ and $r_2$ employing $\Omega_3$.}
	\label{ex3_vor} 
\end{figure}

\noindent $\bullet$ \textbf{Convection-dominated regime:} In this case, {we will discuss the convection-dominated case  $\mu=10^{-8}$. We remark that the performance of the proposed stabilized method remains almost identical for $\mu=0$ as in the case $\mu=10^{-8}$. Therefore, we show the convergence plots only for the convection-dominated case $\mu=10^{-8}$ in Figure \ref{ex2_cng2}.}  It is worth noting that the proposed method once again confirms the derived convergence rates in Theorem \ref{convergence}, particularly highlighting the superconvergence of pressure in the $L^2$-norm for the lowest equal-order. This underscores the robustness of the proposed stabilized method.

\subsection{ Example 3 }
\label{case3}
In the third example, we introduce a classic benchmark linearized Navier-Stokes problem \cite{mfem38}, aiming to investigate the robustness and efficiency of the proposed method for the convection-dominated regime. In this example, we consider $\gamma=1$ and the convective flow field $\boldsymbol{\mathcal{B}}$ is defined as
\begin{align}
	\boldsymbol{\mathcal{B}}(x,y):=\begin{bmatrix}
		\dfrac{r_2}{2 \pi}\Big(1-\cos\Big( \dfrac{2\pi(e^{r_1x}-1)}{e^{r_1}-1}\Big) \Big) \sin\Big( \dfrac{2\pi(e^{r_2y}-1)}{e^{r_2}-1}\Big) \dfrac{e^{r_2y}}{e^{r_2}-1}\\ \\
		- \dfrac{r_1}{2 \pi} \sin\Big( \dfrac{2\pi(e^{r_1x}-1)}{e^{r_1}-1}\Big) \Big(1-\cos\Big( \dfrac{2\pi(e^{r_2y}-1)}{e^{r_2}-1}\Big) \Big) \dfrac{e^{r_1x}}{e^{r_1}-1}
	\end{bmatrix}, \nonumber
\end{align}
where $r_1>0$ and $r_2>0$ are two constants. Furthermore, the load term and boundary condition are {prescribed by the exact solution of the Navier-Stokes equation, given as follows:}
\begin{align*}
	\mathbf{u}(x,y)=\boldsymbol{\mathcal{B}}(x,y), \qquad 	p(x,y)= r_1r_2 \sin\Big( \dfrac{2\pi(e^{r_1x}-1)}{e^{r_1}-1}\Big) \sin\Big( \dfrac{2\pi(e^{r_2y}-1)}{e^{r_2}-1}\Big) \Big) \dfrac{e^{r_1x}e^{r_2y}}{(e^{r_1}-1)(e^{r_2}-1)}.
\end{align*}

{Additionally, the velocity field $\mathbf{u}$ exhibits a counterclockwise vortex confined within a unit-box, as depicted in Figure \ref{ex3_vor}, depending on the constants $r_1$ and $r_2$.} The center of this vertex is denoted by $(x_c,y_c)$, defined as follows:
\begin{align*}
	(x_c,y_c)= \Big( \dfrac{1}{r_1} \log\Big(\dfrac{e^{r_1}+1}{2} \Big),  \dfrac{1}{r_2} \log\Big(\dfrac{e^{r_2}+1}{2}\Big)\Big).
\end{align*}
As $r_1$ increases, the center rapidly shifts towards the right-hand vertical side, while increasing $r_2$ causes it to converge towards the top edge, for instance, see Figure \ref{ex3_vor}. We remark that we have performed this experiment on distorted/skewed non-convex mesh $\Omega_3$.

\begin{table}[t!]
	\setlength{\tabcolsep}{15pt}
	\centering
	\caption{Example 3. The measures for computational errors and rates for $\mu=10^{-8}$ and $r_1=r_2=0.1$ with VEM order $k=1$.}
	\label{ex3_tb1}
	\begin{tabular}{ccccccc}
		\toprule
		{$h$}&  {$E^\mathbf{u}_{H^1}$} & {rate} & {$E^{\mathbf{u}}_{L^2}$} & {rate} & {$E^p_{L^2}$} & {rate}  \\
		\midrule
		$1/5$  & 6.188046e-01 & --   & 6.073925e-02 & --   & 2.463143e-02 & -- \\ 
		$1/10$ & 3.241310e-01 & 0.93 & 2.177106e-02 & 1.48 & 6.056155e-03 & 2.02 \\ 
		$1/20$ & 1.353110e-01 & 1.26 & 6.694913e-03 & 1.70 & 1.468586e-03 & 2.04  \\ 
		$1/40$ & 5.235536e-02 & 1.37 & 1.712536e-03 & 1.97 & 3.531445e-04 & 2.07  \\ 
		$1/80$ & 2.444840e-02 & 1.10 & 4.259555e-04 & 2.01 & 8.906175e-05 & 1.99 \\ 
		\bottomrule		
	\end{tabular}
\smallskip
\end{table}

\begin{table}[t!]
\setlength{\tabcolsep}{15pt}
\centering
	\caption{ Example 3. The measures for computational errors and rates for $\mu=10^{-8}$ and $r_1=r_2=1.1$ with VEM order $k=1$.}
	\label{ex3_tb2}
	\begin{tabular}{ccccccc}
		\toprule
		{$h$}&  {$E^\mathbf{u}_{H^1}$} & {rate} & {$E^{\mathbf{u}}_{L^2}$} & {rate} & {$E^p_{L^2}$} & {rate}  \\
		\midrule
		$1/5$  & 9.263157e-01 & --   & 8.330560e-02 & --   & 3.441664e-02   & -- \\ 
		$1/10$ & 4.798103e-01 & 0.95 & 3.278720e-02 & 1.35 & 9.307929e-03 & 1.89 \\ 
		$1/20$ & 2.070410e-01 & 1.21 & 9.811696e-03 & 1.74 & 2.467906e-03 & 1.92  \\ 
		$1/40$ & 8.266392e-02 & 1.32 & 2.435389e-03 & 2.01 & 5.968625e-04 & 2.05  \\ 
		$1/80$ & 4.033414e-02 & 1.04 & 6.437663e-04 & 1.92 & 1.481498e-04 & 2.01 \\ 
		\bottomrule		
	\end{tabular}
\smallskip
\end{table}

\begin{table}[t!]
	\setlength{\tabcolsep}{15pt}
	\centering
	\caption{Example 3. The measures for computational errors and rates for $\mu=0$ and $r_1=r_2=0.1$ with VEM order $k=1$.}
	\label{ex3_tbc1}
	\begin{tabular}{ccccccc }
		\toprule
		{$h$}&  {$E^\mathbf{u}_{H^1}$} & {rate} & {$E^{\mathbf{u}}_{L^2}$} & {rate} & {$E^p_{L^2}$} & {rate}  \\
		\midrule
		$1/5$  & 6.122127e-01 & --   & 5.692209e-02 & --   & 2.383631e-02 & -- \\ 
		$1/10$ & 3.153532e-01 & 0.95 & 2.033043e-02 & 1.49 & 5.826189e-03 & 2.03 \\ 
		$1/20$ & 1.305485e-01 & 1.27 & 6.229335e-03 & 1.71 & 1.408812e-03 & 2.05  \\ 
		$1/40$ & 5.128066e-02 & 1.35 & 1.606711e-03 & 1.96 & 3.424596e-04 & 2.04  \\ 
		$1/80$ & 2.418057e-02 & 1.09 & 4.122650e-04 & 1.96 & 8.758062e-05 & 1.97 \\ 
		\bottomrule		
	\end{tabular}
\smallskip
\end{table}

\begin{table}[t!]
	\setlength{\tabcolsep}{15pt}
	\centering
	\caption{Example 3. The measures for computational errors and rates for $\mu=0$ and $r_1=r_2=1.1$ with VEM order $k=1$.}
	\label{ex3_tbc2}
	\begin{tabular}{ccccccc}
		\toprule
		{$h$}&  {$E^\mathbf{u}_{H^1}$} & {rate} & {$E^{\mathbf{u}}_{L^2}$} & {rate} & {$E^p_{L^2}$} & {rate}  \\
		\midrule
		$1/5$  & 9.263157e-01 & --   & 8.330560e-02 & --   & 3.441664e-02 & -- \\ 
		$1/10$ & 4.798103e-01 & 0.95 & 3.278720e-02 & 1.35 & 9.307929e-03 & 1.89 \\ 
		$1/20$ & 2.070410e-01 & 1.21 & 9.811696e-03 & 1.74 & 2.467906e-03 & 1.92  \\ 
		$1/40$ & 8.266392e-02 & 1.32 & 2.435389e-03 & 2.01 & 5.968625e-04 & 2.05  \\ 
		$1/80$ & 4.033414e-02 & 1.04 & 6.437663e-04 & 1.92 & 1.481498e-04 & 2.01 \\ 
		\bottomrule		
	\end{tabular}
\smallskip
\end{table}

\begin{table}[t!]
	\setlength{\tabcolsep}{15pt}
	\centering
	\caption{Example 3. The measures for computational errors and rates for $\mu=0$ and $r_1=r_2=0.1$ with VEM order $k=2$.}
	\label{ex3_tbc3}
	\begin{tabular}{ccccccc}
		\toprule
		{$h$}&  {$E^\mathbf{u}_{H^1}$} & {rate} & {$E^{\mathbf{u}}_{L^2}$} & {rate} & {$E^p_{L^2}$} & {rate}  \\
		\midrule
		$1/5$  & 5.604041e-01 & --   & 3.174523e-02 & --   & 1.164492e-02 & -- \\ 
		$1/10$ & 2.054330e-02 & 4.78 & 5.605955e-04 & 5.82 & 1.857552e-03 & 2.64 \\ 
		$1/20$ & 5.015119e-03 & 2.03 & 5.315684e-05 & 3.39 & 4.862285e-04 & 1.93 \\ 
		$1/40$ & 1.535043e-03 & 1.71 & 7.888760e-06 & 2.75 & 1.252258e-04 & 1.96  \\ 
		$1/80$ & 4.298564e-04 & 1.83 &  1.07476e-06 & 2.88 & 3.054348e-05 & 2.04 \\ 
		\bottomrule		
	\end{tabular}
\smallskip
\end{table}

\begin{table}[t!]
	\setlength{\tabcolsep}{15pt}
	\centering 
	\caption{Example 3. The measures for computational errors and rates for $\mu=0$ and $r_1=r_2=1.1$ with VEM order $k=2$.}
	\label{ex3_tbc4}
	\begin{tabular}{ccccccc}
		\toprule
		{$h$}&  {$E^\mathbf{u}_{H^1}$} & {rate} & {$E^{\mathbf{u}}_{L^2}$} & {rate} & {$E^p_{L^2}$} & {rate}  \\
		\midrule
		$1/5$  & 5.987862e-01 & --   & 2.772020e-02 & --   & 1.751584e-02 & -- \\ 
		$1/10$ & 3.745935e-02 & 3.99 & 9.875557e-04 & 4.81 & 4.230723e-03 & 2.05 \\ 
		$1/20$ & 8.778334e-03 & 2.09 & 1.006252e-04 & 3.30 & 1.127445e-03 & 1.91 \\ 
		$1/40$ & 2.578734e-03 & 1.78 & 1.370100e-05 & 2.88 & 2.745577e-04 & 2.04  \\ 
		$1/80$ & 7.203278e-04 & 1.84 & 1.895303e-06 & 2.85 & 7.076284e-05 & 1.96 \\ 
		\bottomrule		
	\end{tabular}
\smallskip
\end{table}

\noindent $\bullet$ \textbf{Convection-dominated regime:} In Tables \ref{ex3_tb1} and \ref{ex3_tb2}, the magnitudes of computational errors $E^\mathbf{u}_{H^1}$,  $E^{\mathbf{u}}_{L^2}$ and $E^p_{L^2}$ are shown along with the convergence rates for convection-dominated regime $\mu=10^{-8}$, employing the lowest equal-order element pairs. It is noteworthy that the proposed method demonstrates excellent agreement across varying values of constants $r_1$ and $r_2$, achieving optimal convergence rates with the superconvergence of pressure in the  $L^2$-norm for VEM order $k=1$, consistent with the findings in \cite{mfem38}.  \newline
\noindent $\bullet$ \textbf{Critical case (zero diffusivity):} In Tables \ref{ex3_tbc1}--\ref{ex3_tbc4}, we have investigated the computational error and convergence behavior of the proposed method with the negligible diffusion {parameter} $\mu=0$ for VEM order $k=1$ and $k=2$. The numerical results presented in Tables \ref{ex3_tbc1}--\ref{ex3_tbc4} demonstrate the optimal convergence of the proposed method for both VEM orders $k=1$ and $k=2$, as observed in their respective norms. Additionally, the pressure is observed to be superconverging for VEM order $k=1$. {These results indicate the robustness and efficiency of the proposed method in handling the {negligible} diffusion parameter without paying much cost compared to residual-based stabilization techniques.}

\subsection{Example 4}
\label{case4}
In this last example, we focus on investigating the efficiency and robustness of the proposed method for a classic benchmark problem characterized by boundary layers. In order to achieve this, we consider the convective flow field $\boldsymbol{\mathcal{B}}=(1,1)^T$ and $\gamma=0$. Additionally, the load term and boundary condition are determined by the exact solution
\begin{align}
	\mathbf{u}(x,y)=\begin{bmatrix}
		y - \dfrac{1-e^{y/\mu}}{1-e^{1/\mu}} \\ \\
		x - \dfrac{1-e^{x/\mu}}{1-e^{1/\mu}}
	\end{bmatrix}, 
	\qquad p(x,y)= x-y. \nonumber
\end{align}

\begin{figure}[h]
	\centering
	\subfloat[First component $u_1$, Order $1$.]{\includegraphics[width=6.5cm]{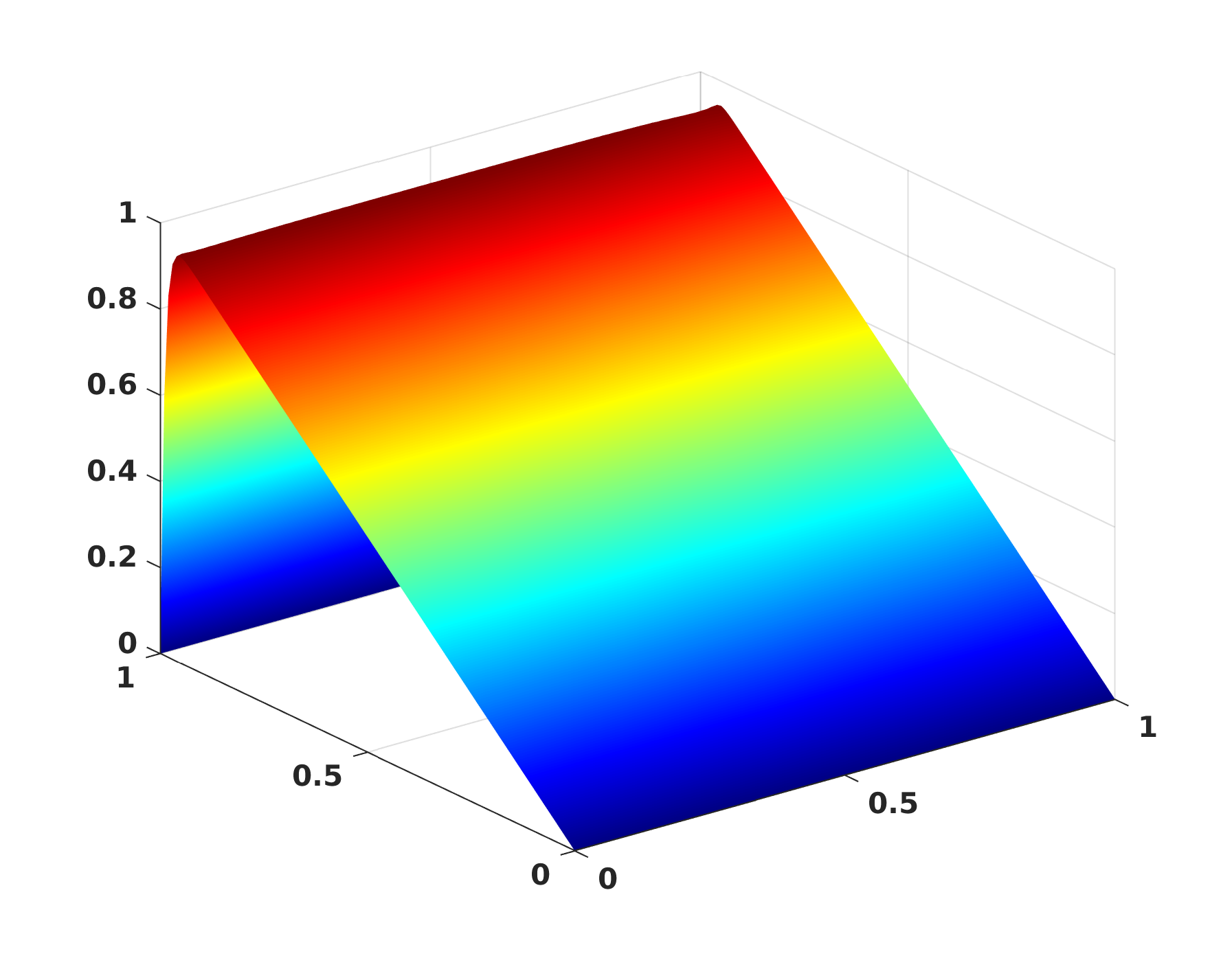}}
	\subfloat[Second component $u_2$, Order $1$.]{\includegraphics[width=6.5cm]{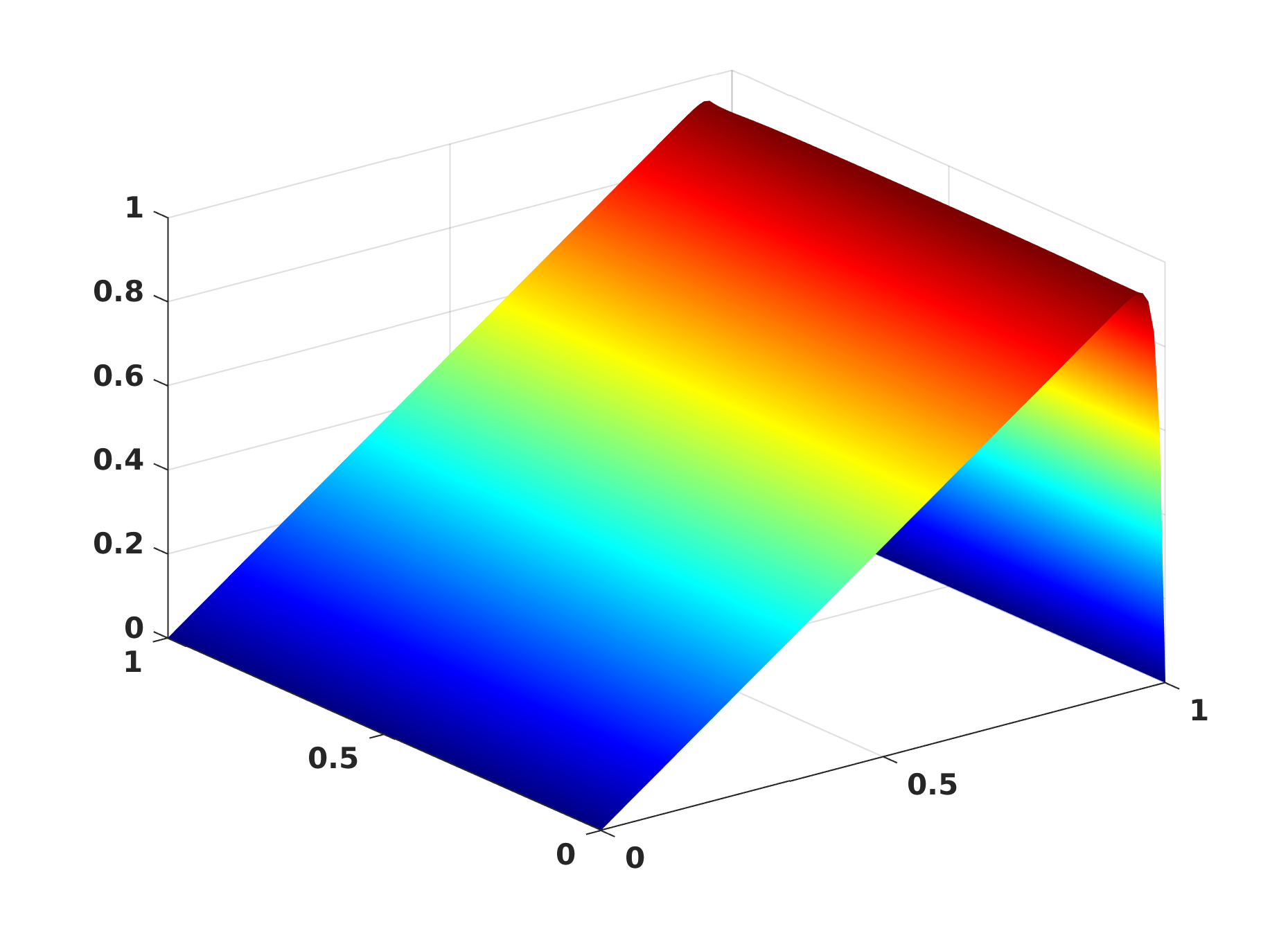}}
	\caption{Example 4. Discrete velocities obtained from stabilized VEM for diffusion parameter $\mu=10^{-2}$.}
	\label{ex4_vel0}
	\subfloat[Order $1$, $\mu=10^{-4}$.]{\includegraphics[width=6.5cm]{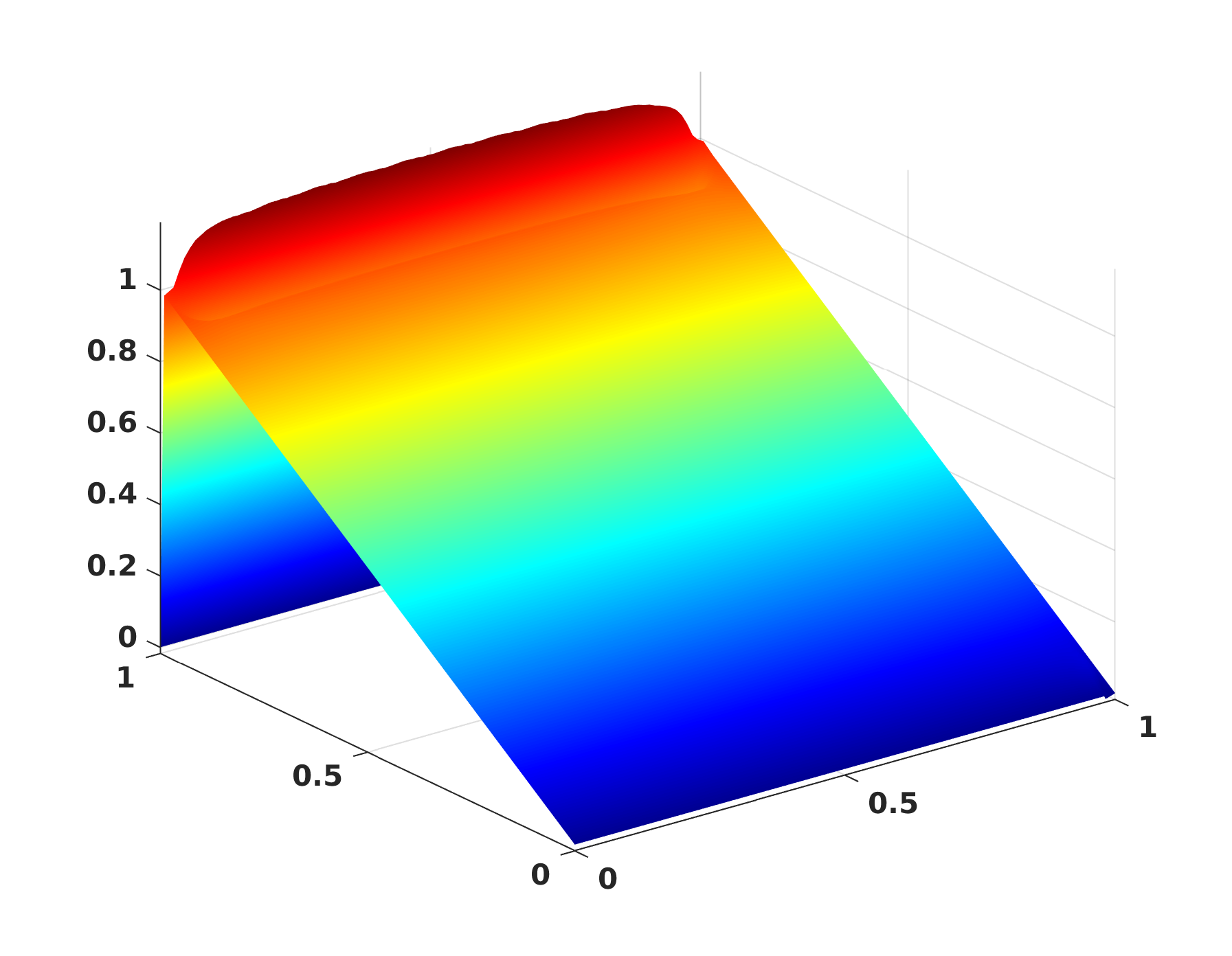}}
	\subfloat[Order $2$, $\mu=10^{-4}$.]{\includegraphics[width=6.5cm]{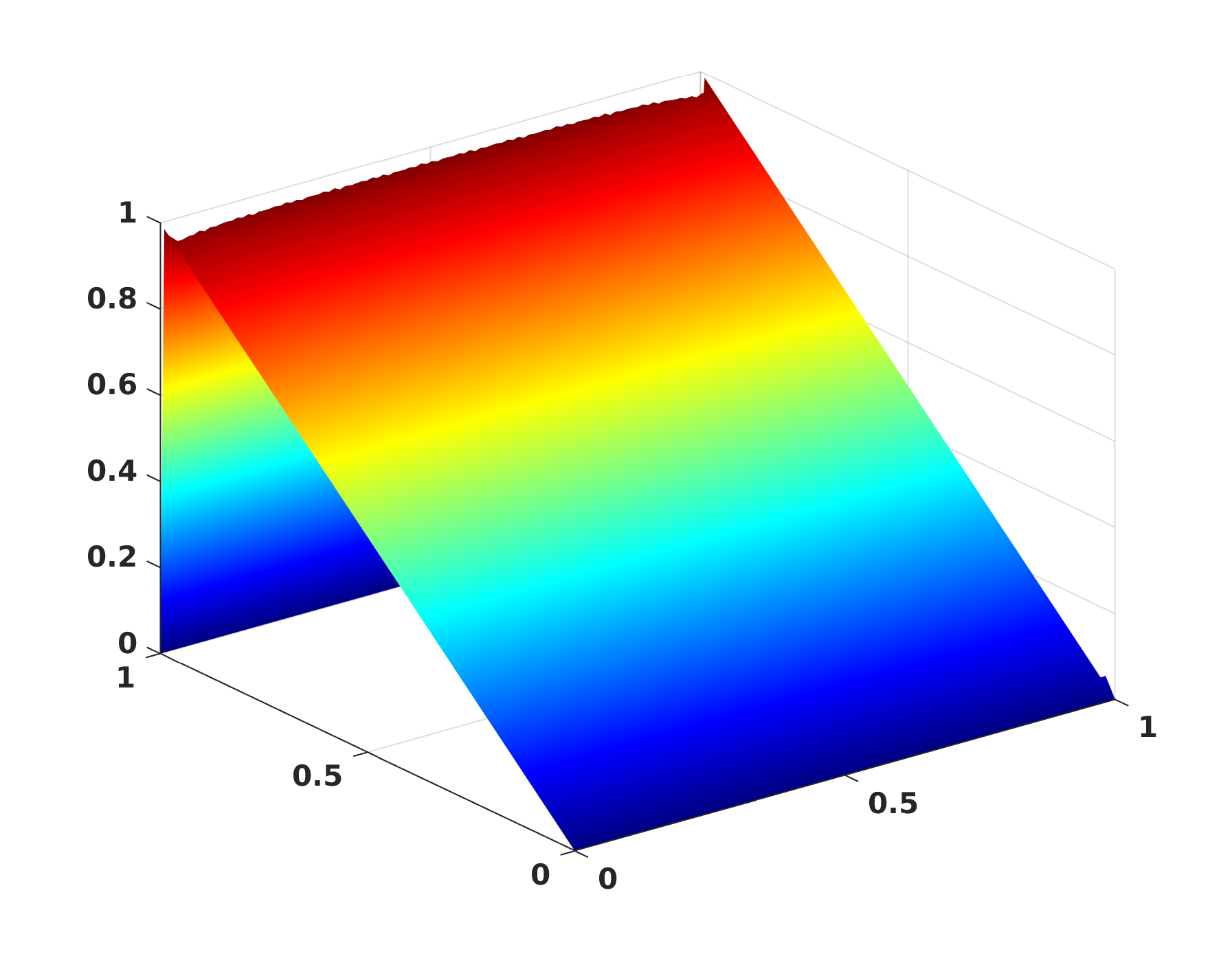}}\\
	\subfloat[Order $1$, $\mu=10^{-6}$.]{\includegraphics[width=6.5cm]{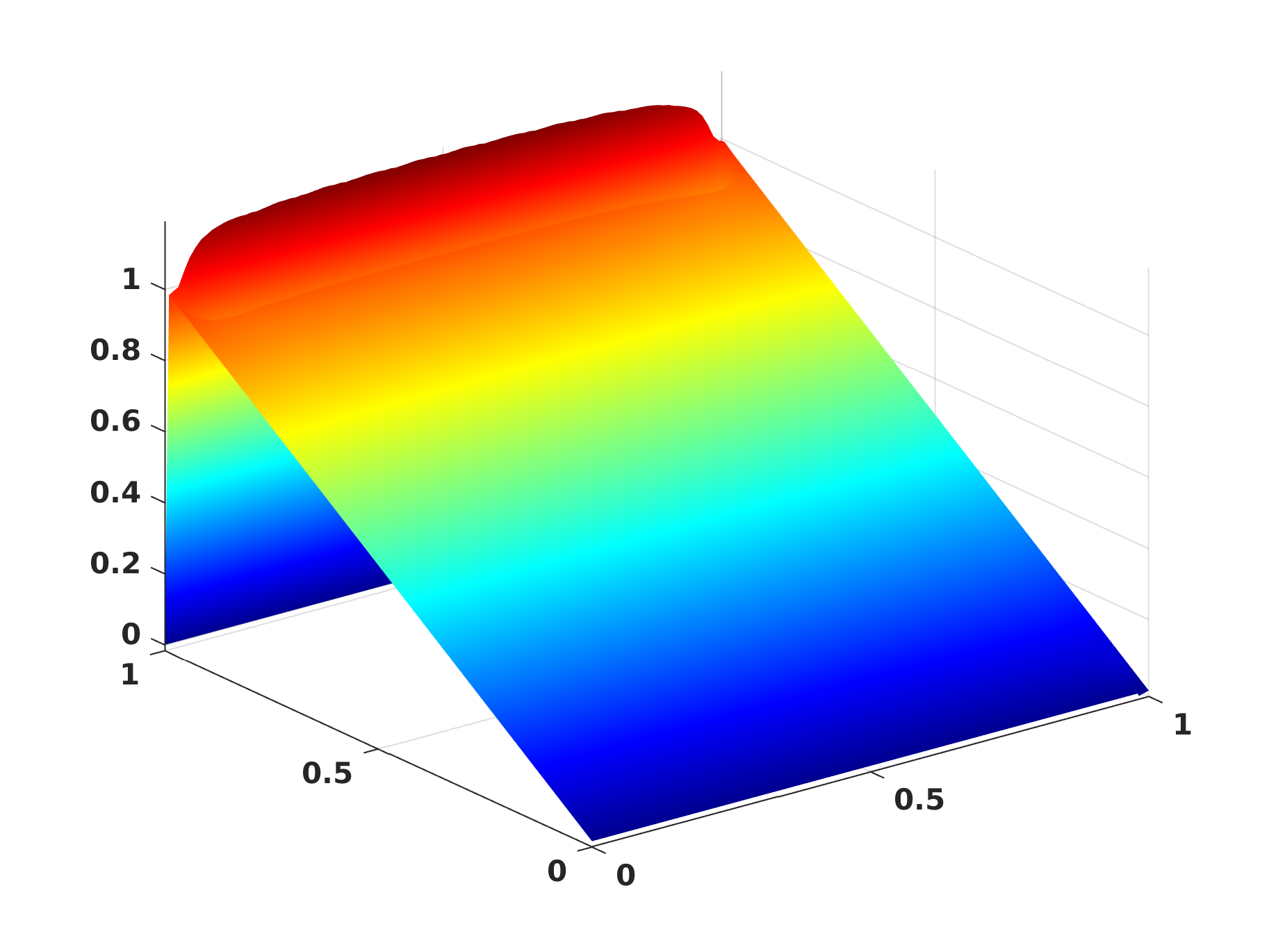}}
	\subfloat[Order $2$, $\mu=10^{-6}$.]{\includegraphics[width=6.5cm]{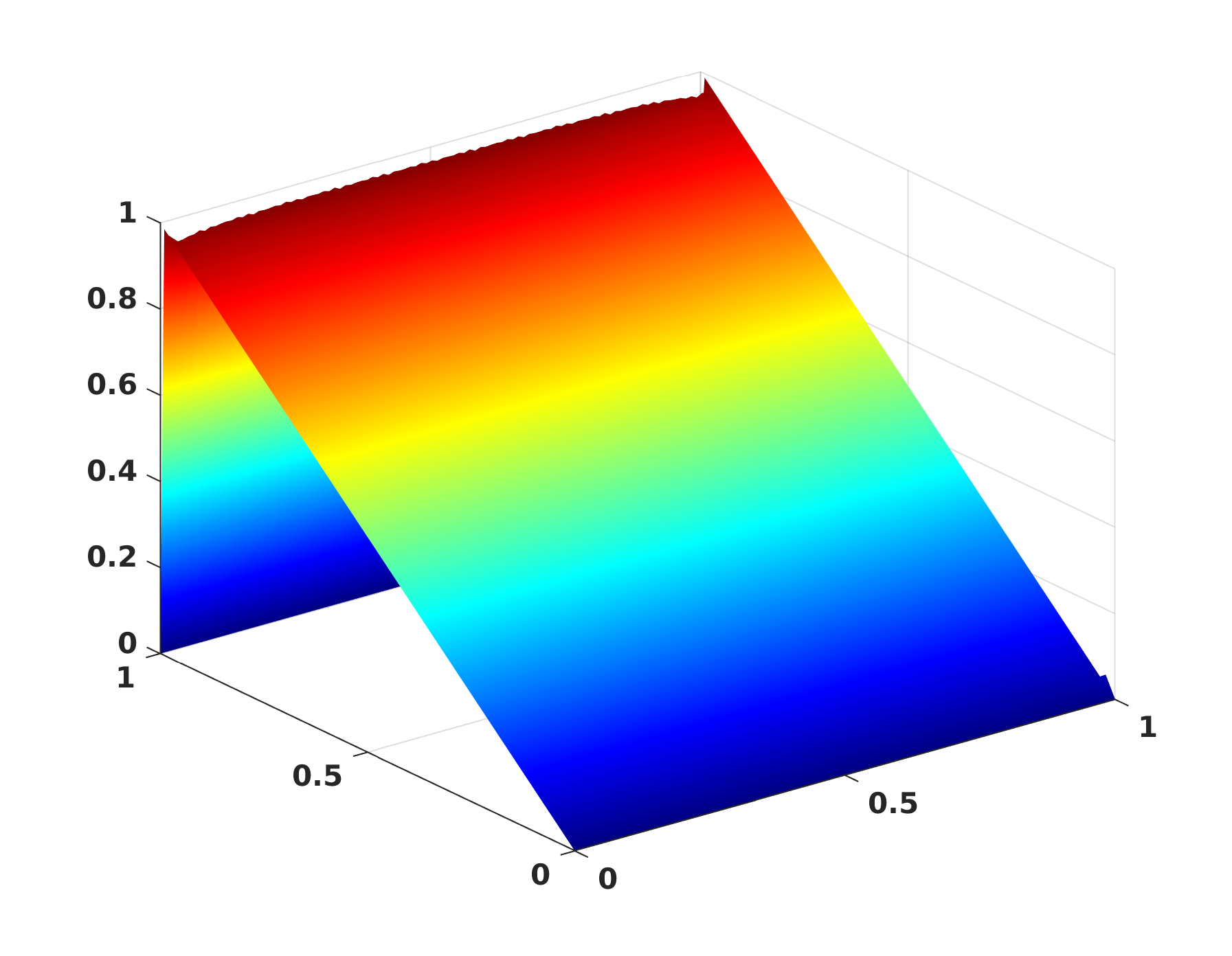}}
	\caption{ Example 4. First component of stabilized discrete velocity for different values of diffusion parameter $\mu$.}
	\label{ex4_vel} 
\end{figure}

\begin{figure}[h]
	\centering
	\subfloat[$\mu=10^{-2}$.]{\includegraphics[width=7cm]{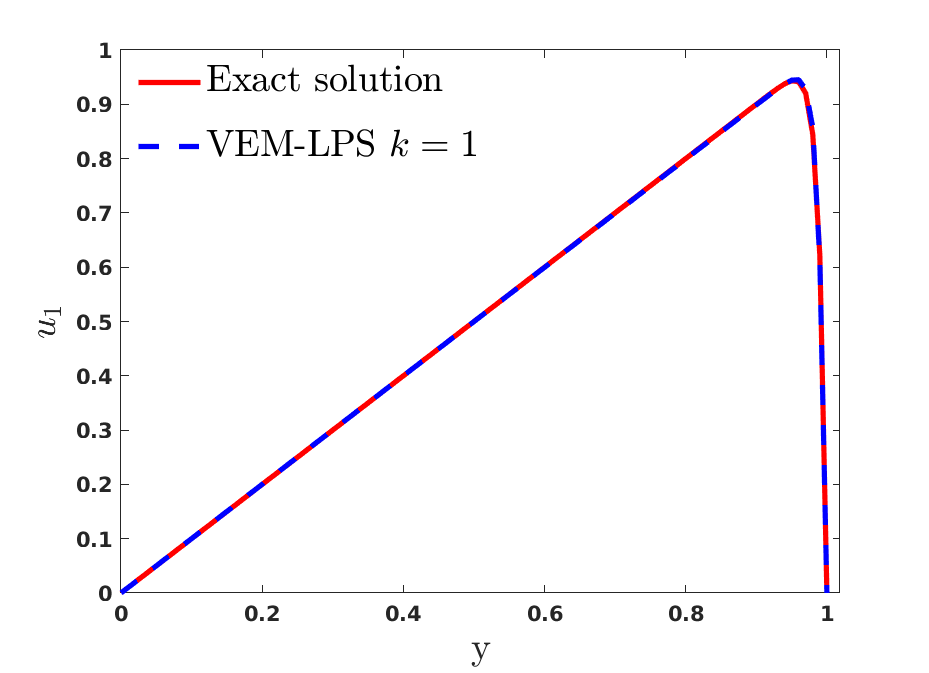}}
	\subfloat[$\mu=10^{-4}$.]{\includegraphics[width=7cm]{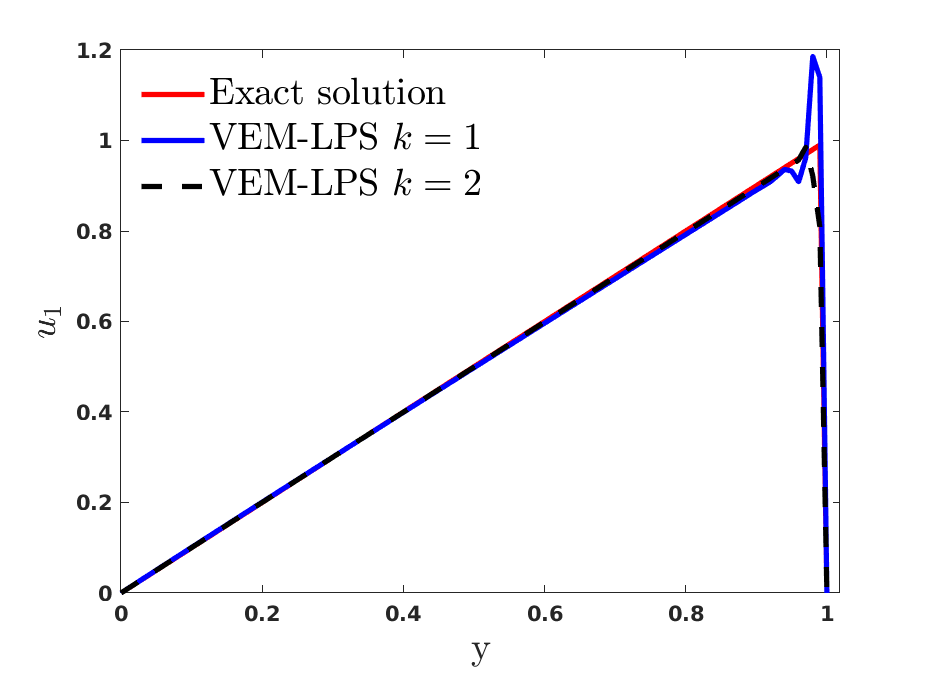}}\\
	\subfloat[$\mu=10^{-6}$.]{\includegraphics[width=7cm]{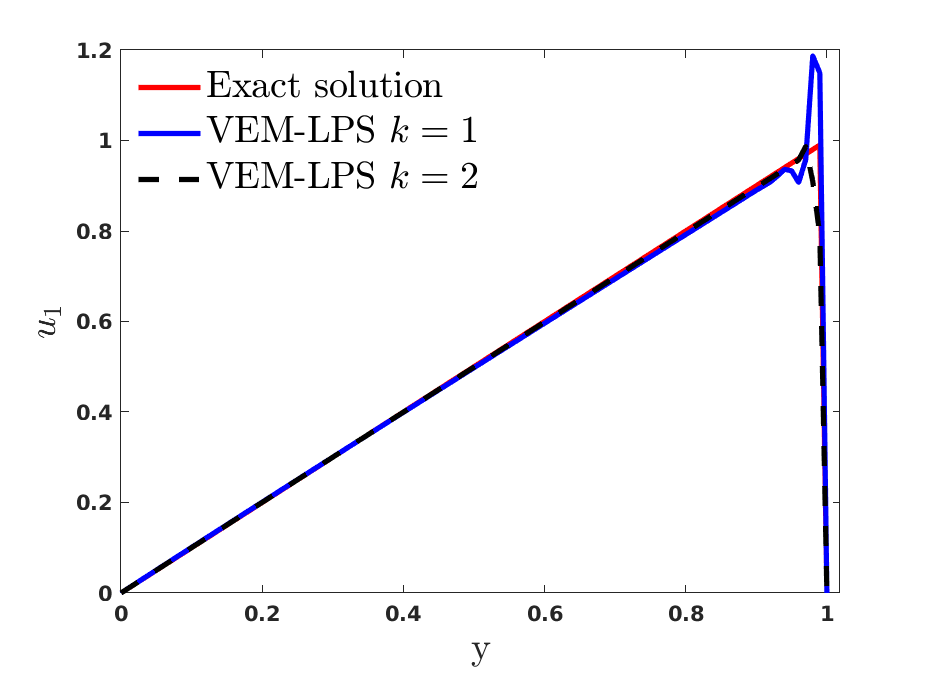}}
	\caption{Example 4. Comparison of the cross-section of the first component of the exact and discrete velocities about $x=0.5$ for VEM order $k=1,2.$}
	\label{ex4_cross} 
\end{figure}

We have conducted this simulation for diffusivity parameters $\mu=10^{-2}, 10^{-4}$ and $10^{-6}$, considering VEM order $k=1$ and $k=2$ using $\Omega_2$. The stabilized discrete velocity is depicted in Figures \ref{ex4_vel0} and \ref{ex4_vel} for VEM order $k=1$ and $k=2$. We observed from Figures \ref{ex4_vel0} and \ref{ex4_vel} that the proposed method exhibits almost oscillation-free solutions near the boundary layers or away from the boundary layers for VEM order $k=1$ and $k=2$. For $\mu=10^{-2}$, the stabilized discrete velocities precisely coincide with the exact velocities of the problem for the lowest equal-order element pairs. Hence, we have omitted showing the discrete velocity for $k=2$ in Figure \ref{ex4_vel0}. 

For $\mu=10^{-4}$, the stabilized discrete velocity is fully oscillation-free, and the boundary layers are visible at $y=1$, which is effectively captured by the proposed stabilized VEM, see Figure \ref{ex4_vel}(a). Additionally, the visibility of the boundary layers for $k=1$ is almost removed by employing the higher order VEM, i.e., $k=2$, for instant review see Figure \ref{ex4_vel}(b). We achieve precisely identical numerical findings for $\mu=10^{-6}$ as well. In Figure \ref{ex4_cross}, we have shown the behavior of the cross-section of the stabilized discrete velocity about $x=0.5$ for various diffusion parameters $\mu$. Thus, we conclude that the most remarkable feature of the proposed VEM is that its accuracy and efficiency increase as we increase the order of VEM, and stabilized discrete solutions seem to converge exactly towards their exact counterparts, as illustrated in Figure \ref{ex4_cross}.  

\section{Conclusion} \label{sec-6}
We propose and analyze a new stabilized virtual elements for addressing the generalized Oseen equation for different cases, including the convection-dominated regime, employing equal-order element pairs. The stabilization techniques are based on local projection-based stabilizing terms. The stabilizing terms are free from higher-order derivative terms, easier to implement compared to residual-based stabilization methods, and do not allow coupling between the element pairs. Further, we have established the stability for the proposed stabilized method, which also aligns with the stability of the Stokes and Brinkman equations, without imposing additional conditions on the parameters.  Moreover, we have derived the error estimates in the energy norm, which is {quasi-robust} with respect to parameters.  Additionally, we have performed several numerical examples validating the optimal convergence rates for different cases, including convection-dominated regimes and zero diffusivity parameter.

\subsection*{Data Availability}
Data sharing is not applicable to this article as no datasets were generated or analyzed.

\section*{Declarations}
\subsection*{Conflict of interest} The authors declare no conflict of interest during the current study.


\section*{Appendix}
	\noindent \textbf{Proof of Theorem \ref{rob_stab}.} To derive the inf-sup constant $ \Theta \geq \frac{1}{4}$ especially for the case $\mu \rightarrow 0$, we follow \cite[Lemma 5]{braack2011equal}. The proof of Theorem \ref{rob_stab} follows in the following steps: \newline
	\noindent \textbf{Step 1.} We use $\widehat{\mathbf{y}}_h:=(\mathbf{u}_h,p_h) \in \mathbf{V}^k_h \times Q^k_h$ with $\vertiii{\widehat{\mathbf{y}}_h}^2= |\widehat{\mathbf{y}}_h|^2_a + |\widehat{\mathbf{y}}_h|^2_b$, where
	\begin{align}
		|\widehat{\mathbf{y}}_h|^2_a &= \mu \|\nabla \mathbf{u}_h\|^2_{0, \Omega} +  \gamma \| \mathbf{u}_h\|^2_{0, \Omega} + + \mathcal{L}_{1,h}(\mathbf{u}_h,\mathbf{u}_h) + \mathcal{L}_{2,h}(\mathbf{u}_h,\mathbf{u}_h) + \mathcal{L}_{3,h}(p_h,p_h), \nonumber \\
		|\widehat{\mathbf{y}}_h|^2_b&= \alpha ^2 \|p_h\|^2_{0, \Omega}. \nonumber 
	\end{align}
	Recalling the estimation of \eqref{wellpos-1}, we infer
	\begin{align}
		(A_h + \mathcal{L}_h)[\widehat{\mathbf{y}}_h, \widehat{\mathbf{y}}_h]& \geq  \min \big\{ 1, \lambda_{1\ast}, \lambda_{2\ast} \big\} \Big( \mu \| \nabla \mathbf{u}_h\|^2_{0, \Omega}  + \gamma \|\mathbf{u}_h\|^2_{0,\Omega} + \mathcal{L}_{1,h}(\mathbf{u}_h,\mathbf{u}_h) + \mathcal{L}_{2,h}(\mathbf{u}_h,\mathbf{u}_h)+ \mathcal{L}_{3,h}(p_h,p_h) \Big) \nonumber \\
		& \geq \lambda_{0\ast} |\widehat{\mathbf{y}}_h|^2_a, \label{newwellpos-1}
	\end{align}
	where $\lambda_{0\ast}:= \min \big\{ 1, \lambda_{1\ast}, \lambda_{2\ast} \big\}$. \newline
	\noindent \textbf{Step 2.} We choose $\mathbf{w}_h =\dfrac{ \xi_0\|p_h\|_{0,\Omega}}{\| \nabla \mathbf{r}_h\|_{0,\Omega}} \mathbf{r}_h$, then we obtain the following:
	\begin{align}
		\|\nabla \mathbf{w}_h\|_{0,\Omega} = \xi_0 \|p_h\|_{0,\Omega}, \qquad \,\,\text{with} \,\, \xi_0=\min \big\{ 1, \dfrac{\Theta_1}{2} \big\}. \label{newwellpos-2}
	\end{align}
	Employing Eq. \eqref{newwellpos-2} and Lemma \ref{infsup} , we infer 
	\begin{align}
		b_{h}(\mathbf{w}_h, p_h) \geq \xi_0\|p_h\|_{0,\Omega} \Big(\Theta_1 \|p_h\|_{0,\Omega} - \Theta_2 [\mathcal{L}_{3,h}(p_h,p_h)]^{\frac{1}{2}} \Big). \label{newwellpos-3}
	\end{align}
	\noindent \textbf{Step 3.} Let $\widehat{\mathbf{x}}_h:=(\mathbf{v}_h,0)=(-\alpha \Theta_3 \mathbf{w}_h,0)$ with positive constant $\Theta_3$, and its value will be determined later. Employing the definition of $|\cdot|_a$, it gives
	\begin{align}
		|\widehat{\mathbf{x}}_h|_a^2=\alpha^2\Theta_3^2\big( \mu \| \nabla \mathbf{w}_h\|^2_{0, \Omega}  + \gamma \|\mathbf{w}_h\|^2_{0,\Omega} + \mathcal{L}_{1,h}(\mathbf{w}_h,\mathbf{w}_h) + \mathcal{L}_{2,h}(\mathbf{w}_h,\mathbf{w}_h) \big). \nonumber 
	\end{align}
	Using the stability property of the projection operator, \eqref{vem-a}, \eqref{reg1} and \eqref{newwellpos-2}, yields
	\begin{align}
		|\widehat{\mathbf{x}}_h|_a^2&=\alpha^2\Theta_3^2\big( \mu \| \nabla \mathbf{w}_h\|^2_{0, \Omega}  + \gamma C_P^2 \|\nabla \mathbf{w}_h\|^2_{0,\Omega} + C_{gen} \|\nabla \mathbf{w}_h\|^2_{0,\Omega} \big) \nonumber \\
		& \leq \xi_0^2\alpha^2\Theta_3^2\big( \mu  + \gamma C_P^2 + C_{gen}  \big) \|p_h\|^2_{0,\Omega} \nonumber \\ 
		& \leq \xi_0^2\Theta_3^2\big( \mu  + \gamma C_P^2 + C_{gen}  \big) |\widehat{\mathbf{y}}_h|^2_b. \nonumber 
	\end{align}
	We fix $\Theta_3= \xi_0^{-1}\big( \mu  + \gamma C_P^2 + C_{gen}  \big)^{-1/2}$. Thus, we obtain
	\begin{align}
		\vertiii{\widehat{\mathbf{x}}_h}=|\widehat{\mathbf{x}}_h|_a \leq |\widehat{\mathbf{y}}_h|_b \leq \vertiii{\widehat{\mathbf{y}}_h}. \label{newwellpos-4}
	\end{align}
	\noindent \textbf{Step 4.} For given $\widehat{\mathbf{y}}_h, \widehat{\mathbf{x}}_h \in \mathbf{V}^k_h \times Q^k_h$, we proceed as follows:
	\begin{align}
		(A_h+\mathcal{L}_h)[\widehat{\mathbf{y}}_h,\widehat{\mathbf{x}}_h]&= a_{h}(\mathbf{u}_h,\mathbf{v}_h) - b_{h}(\mathbf{v}_h,p_h)+ c^{skew}_{h}(\mathbf{u}_h,\mathbf{v}_h) +  d_{h}(\mathbf{u}_h,\mathbf{v}_h) + \mathcal{L}_{1,h}(\mathbf{u}_h,\mathbf{v}_h) +  \mathcal{L}_{2,h}(\mathbf{u}_h,\mathbf{v}_h)\nonumber \\ 
		&=: s_{11} + s_{12} +s_{13} +s_{14} +s_{15} +s_{16}. \label{newwell_b}
	\end{align}
	Recalling the estimation of $s_1$, we have 
	\begin{align}
		s_{11} \leq \max\big\{ 1, \lambda_1^\ast\} \mu \|\nabla \mathbf{u}_h\|_{0,\Omega} \|\nabla \mathbf{v}_h\|_{0,\Omega} \leq C_{gen} |\widehat{\mathbf{y}}_h|_a |\widehat{\mathbf{x}}_h|_a. \label{newwell1}
	\end{align}
	To estimate $s_{12}$, we use \eqref{newwellpos-2} and \eqref{newwellpos-3}:
	\begin{align}
		s_{12}&= \alpha \Theta_3 b_h(\mathbf{w}_h,p_h) \nonumber \\
		& \geq \xi_0 \alpha \Theta_3 \|p_h\|_{0,\Omega} \Big(\Theta_1 \|p_h\|_{0,\Omega} - \Theta_2 [\mathcal{L}_{3,h}(p_h,p_h)]^{\frac{1}{2}} \Big) \nonumber \\
		& \geq \xi_0 \Theta_3 \Big( \big(\Theta_1 \alpha^{-1} - \frac{\Theta_2}{4} \big) |\widehat{\mathbf{y}}_h|^2_b - \Theta_2 |\widehat{\mathbf{y}}_h|^2_a \Big). \label{newwell2}	
	\end{align}
	Using the integration by parts, gives
	\begin{align}
		s_{13} &= \frac{1}{2}\sum_{E \in \Omega_h} \Big( -2 \int_E (\nabla \boldsymbol{\Pi}^{0,E}_k \mathbf{v}_h) \boldsymbol{\mathcal{B}} \cdot \boldsymbol{\Pi}^{0,E}_k \mathbf{u}_h dE +  \int_{\partial E} \boldsymbol{\mathcal{B}} \cdot\mathbf{n}^E \mathbf{u}_h \cdot(\boldsymbol{\Pi}^{0,E}_k \mathbf{v}_h - \mathbf{v}_h) ds + \int_{\partial E} \boldsymbol{\mathcal{B}} \cdot\mathbf{n}^E \mathbf{u}_h \cdot \mathbf{v}_h ds \nonumber \\
		& \qquad \int_{\partial E} \boldsymbol{\mathcal{B}} \cdot \mathbf{n}^E ( \mathbf{v}_h - \boldsymbol{\Pi}^{0,E}_k \mathbf{v}_h) \cdot \boldsymbol{\Pi}^{0,E}_k \mathbf{u}_h ds \Big). \label{newwell2a} 	
	\end{align}
	Applying the fact $\boldsymbol{\mathcal{B}} \in W^{1,\infty}(\Omega)$ and $\mathbf{u}_h, \mathbf{v}_h \in \mathbf{V}$, then it holds that $\sum_{E \in \Omega_h} \int_{\partial E} \boldsymbol{\mathcal{B}} \cdot\mathbf{n}^E \mathbf{u}_h \cdot \mathbf{v}_h ds =0 $. Recalling Lemmas \ref{inverse} and \ref{trace} and the bound \eqref{bd1}, it gives
	\begin{align}
		\Big|\int_{\partial E} \boldsymbol{\mathcal{B}} \cdot\mathbf{n}^E \mathbf{u}_h \cdot(\boldsymbol{\Pi}^{0,E}_k \mathbf{v}_h - \mathbf{v}_h) ds \Big| &\leq  C \boldsymbol{\mathcal{B}}_E \big(h^{-1/2}_E \|\mathbf{u}_h\|_{0,E} + h^{1/2}_E  \|\nabla \mathbf{u}_h\|_{0,E}\big) h^{1/2}_E \|\nabla \mathbf{v}_h\|_{0,E}  \nonumber \\
		& \leq C \boldsymbol{\mathcal{B}}_E \|\mathbf{u}_h\|_{0,E} \|\nabla \mathbf{v}_h\|_{0,E}. \label{newwell3}
	\end{align}
	Following the above, there holds
	\begin{align}
		\Big|\int_{\partial E} \boldsymbol{\mathcal{B}} \cdot\mathbf{n}^E ( \mathbf{v}_h - \boldsymbol{\Pi}^{0,E}_k \mathbf{v}_h) \cdot \boldsymbol{\Pi}^{0,E}_k \mathbf{u}_h ds \Big| 
		& \leq C \boldsymbol{\mathcal{B}}_E \|\mathbf{u}_h\|_{0,E} \|\nabla \mathbf{v}_h\|_{0,E}. \label{newwell4}
	\end{align}
	Combining \eqref{newwell2a}, \eqref{newwell3} and \eqref{newwell4}, we arrive at
	\begin{align}
		s_{13} &\leq C_{gen} \mathcal{B} \|\mathbf{u}_h\|_{0,\Omega} \|\nabla \mathbf{v}_h\|_{0,\Omega} \nonumber \\
		&\leq C_{gen} \alpha \Theta_3 \xi_0 \min \Big\{ \frac{C_P}{\sqrt{\mu}}, \frac{1}{\sqrt{\gamma}}\Big\}  \mathcal{B} \ |\widehat{\mathbf{y}}_h|_a \|p_h\|_{0,\Omega} \nonumber \\
		& \leq \frac{2C_P C_{gen} \Theta_3 \xi_0  \mathcal{B}}{\sqrt{\mu} + C_P \sqrt{\gamma} } |\widehat{\mathbf{y}}_h|_a |\widehat{\mathbf{y}}_h|_b \nonumber \\
		& \leq \frac{2C^2_P C^2_{gen} \Theta_3 \xi_0  \mathcal{B}^2}{\mu + C^2_P \gamma}  |\widehat{\mathbf{y}}_h|^2_a + \frac{\xi_0 \Theta_3}{2} |\widehat{\mathbf{y}}_h|^2_b. \label{newwell5}
	\end{align}
	Concerning $s_{14}$, $s_{15}$ and $s_{16}$, we proceed as follows
	\begin{align}
		s_{14} + s_{15} + s_{16} &\leq C\max\big\{1, \lambda_2^\ast \big\} \gamma \|\mathbf{u}_h\|_{0,\Omega} \|\mathbf{v}_h\|_{0,\Omega} + \mathcal{L}^{1/2}_{1,h}(\mathbf{u}_h,\mathbf{u}_h) \mathcal{L}^{1/2}_{1,h}(\mathbf{v}_h,\mathbf{v}_h) + \mathcal{L}^{1/2}_{2,h}(\mathbf{u}_h,\mathbf{u}_h) \mathcal{L}^{1/2}_{2,h}(\mathbf{v}_h,\mathbf{v}_h) \nonumber \\
		& \leq  C_{gen} |\widehat{\mathbf{y}}_h|_a |\widehat{\mathbf{x}}_h|_a. \label{newwell6}
	\end{align}
	Substituting \eqref{newwell1}, \eqref{newwell2}, \eqref{newwell5} and \eqref{newwell6} into \eqref{newwell_b}, yields
	\begin{align}
		(A_h+\mathcal{L}_h)[\widehat{\mathbf{y}}_h,\widehat{\mathbf{x}}_h] &\geq - C_{gen} |\widehat{\mathbf{y}}_h|_a |\widehat{\mathbf{x}}_h|_a + \xi_0 \Theta_3 \Big( \Theta_1 \alpha^{-1} - \frac{\Theta_2}{4} - \frac{1}{2}\Big) |\widehat{\mathbf{y}}_h|^2_b - \xi_0 \Theta_3 \Big(\Theta_2 + \frac{2C^2_P C^2_{gen} \mathcal{B}^2}{\mu + C^2_P \gamma} \Big) |\widehat{\mathbf{y}}_h|^2_a \nonumber \\
		&\geq C_2 |\widehat{\mathbf{y}}_h|^2_b - C_1 |\widehat{\mathbf{y}}_h|^2_a,
	\end{align}
	where last line is obtained using the Young inequality and \eqref{newwellpos-4}, and the constants $C_2$ and $C_1$ are given by
	\begin{align}
		C_2:&= \xi_0 \Theta_3 \Big( \Theta_1 \alpha^{-1} - \frac{\Theta_2}{4} - \frac{1}{2}\Big) - \frac{C_{gen}}{2}, \nonumber \\
		C_1:&= \xi_0 \Theta_3 \Big(\Theta_2 +  \frac{2C^2_P C^2_{gen} \mathcal{B}^2}{\mu + C^2_P \gamma} \Big) + \frac{C_{gen}}{2}. \nonumber
	\end{align}
	\noindent \textbf{Step 5.} Following \cite[Lemma 5]{braack2011equal} with $\rho=3/2$ and $C_0=\lambda_{0\ast}$, we stress that the inf-sup constant $\Theta\geq \frac{1}{4}$ under the condition 
	\begin{align*}
		C_2 \geq \max \Big\{\frac{5}{8}, \frac{\lambda_{0\ast} +C_1}{4\lambda_{0\ast}-1}\Big\} \qquad \text{for any}\,\, \lambda_{0\ast} > \frac{1}{4}. 
	\end{align*}
	To satisfy this condition, we first fix $\lambda_{0\ast} = \frac{1}{2}$ and then choose $C_2=1+C_1$, which gives the following definition of $\alpha$ appears in \eqref{tnorm}:
	\begin{align}
		\alpha = \frac{4 \Theta_1}{2 + 5 \Theta_2 + 4(1+ C_{gen}) (\mu + C_P^2 \gamma + C_{gen})^{1/2} + \frac{8C_P^2 C_{gen}^2 \mathcal{B}^2}{\mu + C_P^2 \gamma}}. \label{alpha2}
	\end{align}
	Hereafter, the well-posedness follows from \textbf{Step 6} of Theorem \ref{wellposed}. \qed

		\bibliographystyle{plain}
		\bibliography{references}
		
	\end{document}